\documentclass[10pt]{amsart}
\usepackage{amsmath,epsf}
\usepackage{mathrsfs}
\usepackage{hyperref}
\usepackage{bm}

\usepackage{amssymb,stmaryrd,mathtools,amsthm}
\usepackage{enumitem,xcolor}

\usepackage{tikz} \usetikzlibrary{matrix,arrows,decorations.pathmorphing}
\usepackage{tikz-cd}

\theoremstyle{plain}
\newtheorem{theorem}[subsection]{Theorem}
\newtheorem*{theorem*}{Theorem}
\newtheorem{corollary}[subsection]{Corollary}
\newtheorem*{corollary*}{Corollary}
\newtheorem{proposition}[subsection]{Proposition}
\newtheorem{proposition*}{Proposition}
\newtheorem{lemma}[subsection]{Lemma}
\newtheorem{fact}[subsection]{Fact}

\theoremstyle{definition}
\newtheorem{definition}[subsection]{Definition}
\newtheorem{example}[subsection]{Example}

\theoremstyle{remark}
\newtheorem{remark}[subsection]{Remark}

\newtheorem{notation}[subsection]{Notation}

\providecommand{\Z}{\mathbb{Z}}
\providecommand{\C}{\mathbb{C}}

\providecommand{\F}{\mathbb{F}}
\providecommand{\N}{\mathbb{N}}

\providecommand{\bbG}{\mathbb{G}}

\providecommand{\bbD}{\mathbb{D}}

\providecommand{\p}{\mathfrak{p}}

 \providecommand{\cF}{\mathscr{F}}

\providecommand{\cA}{\mathscr{A}}
\providecommand{\cB}{\mathscr{B}}
\providecommand{\cC}{\mathscr{C}}

\providecommand{\cD}{\mathscr{D}}

\providecommand{\cP}{\mathscr{P}}
\providecommand{\cS}{\mathscr{S}}

\providecommand{\bsi}{{\bm{\sigma}}}

\providecommand{\spec}{\mathop{\rm Spec}}

\providecommand{\et}{\text{\'et}}
\providecommand{\Ob}{\mathop{\rm Ob}}
\providecommand{\id}{\mathop{\rm id}\nolimits}

\providecommand{\Aut}{\mathop{\rm Aut}\nolimits}

\providecommand{\op}{{\mathop{\rm op}\nolimits}}

\providecommand{\Fix}{\mathop{\rm Fix}\nolimits}

\renewcommand{\ker}{\mathop{\rm ker}\nolimits}

\providecommand{\Eq}{\mathop{\rm Eq}\nolimits}

\providecommand{\colim}{\mathop{\rm colim}}

\providecommand{\lbr}{\llbracket}
\providecommand{\rbr}{\rrbracket}

\providecommand{\Gal}{\mathop{\rm Gal}}

\providecommand{\Split}{\text{\rm Split}}

\providecommand{\ld}{\text{\rm ld}}

\providecommand{\Set}{\mathbf{Set}}
\providecommand{\Rng}{\mathbf{Rng}}

\providecommand{\Prof}{\mathbf{Prof}}
\providecommand{\Topl}{\mathbf{Top}}

\providecommand{\Gp}{\mathbf{Gr}}

\providecommand{\Alg}{\text{-}\mathbf{Alg}}

\providecommand{\Corr}{\text{\rm Corr}}
\providecommand{\coCorr}{\text{\rm coCorr}}

\providecommand{\diff}{\bsi\text{-}}

\providecommand{\ov}{{\kern-1pt\mathord{/}\kern-.7pt}}
\providecommand{\Ov}{{\kern-1pt\sslash\kern-.7pt}}

\providecommand{\diff}{\text{\it Diff}}

\providecommand{\forg}[1]{\lfloor#1\rfloor}
\providecommand{\ssig}[1]{\lceil#1|}
\providecommand{\psig}[1]{|#1\rceil}

\providecommand{\lexp}[2]{{\vphantom{#2}}^{#1}{\kern-.1ex#2}}
\providecommand{\lsub}[2]{{\vphantom{#2}}_{#1}{\kern-.1ex#2}}

\providecommand{\lrexp}[3]{{\vphantom{#2}}^{#1}{\kern-.1ex#2^#3}}
\providecommand{\ldexp}[3]{{\vphantom{#2}}^{#1}{\kern-.1ex#2_#3}}

\makeatletter
\newcommand*{\doublerightarrow}[2]{\mathrel{
  \settowidth{\@tempdima}{$\scriptstyle#1$}
  \settowidth{\@tempdimb}{$\scriptstyle#2$}
  \ifdim\@tempdimb>\@tempdima \@tempdima=\@tempdimb\fi
  \mathop{\vcenter{
    \offinterlineskip\ialign{\hbox to\dimexpr\@tempdima+1em{##}\cr
    \rightarrowfill\cr\noalign{\kern.5ex}
    \rightarrowfill\cr}}}\limits^{\!#1}_{\!#2}}}
\makeatother

\newcommand{\mw}[1]{{\color{red}#1}}

\begin{document}

%37B10, Symbolic dynamics
%18E50   Categorical Galois Theory
%12H10 Difference algebra
% 18D40 Internal categories and groupoid

\title{Difference Galois theory and dynamics}
\date{\today}

\author{Ivan Toma{\v s}i{\'c}} 
\address{Ivan Toma{\v s}i{\'c}\\
         School of Mathematical Sciences\\
  	Queen Mary University of London\\
         London, E1 4NS\\
        United Kingdom}
\email{i.tomasic@qmul.ac.uk}

\author{Michael Wibmer} 
\address{Michael Wibmer\\
Institute of Analysis and Number Theory\\ 
Graz University of Technology\\
Kopernikusgasse 24\\
8010 Graz, Austria}
\email{wibmer@math.tugraz.at}

\begin{abstract}
We develop a Galois theory for difference ring extensions, inspired by Magid's separable Galois theory for ring extensions and by Janelidze's categorical Galois theory. Our difference Galois theorem states that the category of difference ring extensions split by a chosen Galois difference ring extension is classified by actions of the associated \emph{difference profinite Galois groupoid}. In particular, \emph{difference locally \'etale} extensions of a difference ring are classified by its \emph{difference profinite fundamental groupoid}.

The emergence of difference profinite spaces, viewed as discrete dynamical systems in the realm of topological dynamics, leads us to investigate the interaction of difference algebra and symbolic dynamics. 
As an application of this interaction, we prove the near-rationality of a certain difference zeta function counting solutions of systems of difference algebraic equations over algebraic closures of finite fields with Frobenius. 
\end{abstract}

\maketitle

%    Dedication.  If the dedication is longer than a line or two
%    remove the centering instructions and the line break.
%\cleardoublepage
%\thispagestyle{empty}
%\vspace*{13.5pc}
%\begin{center}
%  Dedication text (use \\[2pt] for line break if necessary)
%\end{center}
%\cleardoublepage

%\setcounter{tocdepth}{3}

%\tableofcontents

\section{Introduction}

\subsection{Difference algebra}

Difference algebra  was founded by Ritt in the 1930s as a study of rings and modules equipped with distinguished endomorphisms called \emph{difference operators} (\cite{cohn},\cite{levin}). In particular, a \emph{difference ring} is a pair
$$
(R,\sigma_R)
$$
consisting of a commutative unital ring and a ring endomorphism $\sigma_R:R\to R$. %\mw{-p}

A \emph{homomorphism of difference rings} 
$$
f:(R,\sigma_R)\to (S,\sigma_S)
$$
is a homomorphism $f$ of {the} underlying rings commuting with {the} difference operators, i.e.,
$$
\sigma_S\circ f=f\circ\sigma_R.
$$
We thus obtain a category of difference rings
$$
\diff\Rng.
$$
If $R$ is a difference ring, we will use the notation
$$
\forg{R}
$$
to denote the underlying ring.

By analogy, we define categories of difference objects in an arbitrary category {(\ref{cat-diff-obj})}. The most interesting  
instances {that} we shall use are \emph{difference sets} $\diff\Set$, \emph{difference groups} $\diff\Gp$, \emph{difference topological spaces} $\diff\Topl$, \emph{difference profinite spaces} $\diff\Prof$. 

Difference rings arise in the study of recursions, equations in the calculus of finite differences, functional equations and many other contexts. To emphasise the connection to dynamics, note that the algebras of suitable classes of (observable) functions on a discrete dynamical system $(X,\sigma_X)\in \diff\Topl$ are naturally endowed with a difference ring structure induced by precomposition with the shift $\sigma_X$.

\subsection{Difference field extensions}

A \emph{difference field} is a difference ring whose underlying ring is a field.

{While there exists a successful Galois theory of linear difference equations as described in \cite{singer}, a Galois theory for difference field and difference ring extensions in the spirit of Grothendieck's Galois theory has not been identified.}

Let $L/k$ be a difference field extension. Ritt-style difference algebra suggests that the group of $\diff\Rng$-automorphisms
$$
\Aut(L/k)=\{g\in\Aut(\forg{L}/\forg{k}): \sigma_L\circ g=g\circ \sigma_L\}
$$
may be relevant to study this extension. However, even under the assumption that $\forg{L}/\forg{k}$ is a Galois extension, this group of strict automorphisms is `too small' for the purpose of classifying difference field subextensions of $L/k$. 

The authors have identified (\cite[2.20]{ive-tgs}, \cite[1.23]{ive-mich-babb}) that, even when $\sigma_L$ is not invertible, 
$$
\Gal(\forg{L}/\forg{k})
$$
carries a structure of a difference group through a group endomorphism $()^\sigma:\Gal(\forg{L}/\forg{k})\to \Gal(\forg{L}/\forg{k})$ satisfying
$$
g\circ\sigma_L=\sigma_L \circ g^\sigma,
$$
and found that its \emph{difference actions} are better suited for the purpose. On the other hand, the approach to simply consider transformations of underlying objects does not generalise well. 

To gain a better understanding of the situation, let us discuss \emph{difference group actions}.
In the category of difference sets, an action of a difference group $G$ on a difference set $X$ defined through a morphism of difference sets
$$
\mu:G\times X\to X
$$
satisfying the relevant properties \emph{does not} correspond to a morphism from $G$ to $\Aut(X)$. Indeed, $G$ is a difference group, and $\Aut(X)$ is an ordinary group. 

On the other hand, using the \emph{cartesian closed structure} of the category $\diff\Set$ (\ref{diffset-cart-cl}), this action can equivalently be given by a \emph{difference group} homomorphism
$$
G\to \Aut[X],
$$
where the latter is the \emph{internal automorphism group} of $X$, as explained in \ref{diff-gp-ac-intaut}. The fact that
$$
\Aut(X)=\Fix(\Aut[X])=\{g\in\Aut[X] :\ \sigma(g)=g\},
$$
shows precisely how much smaller  the group of strict automorphisms is compared to the `correct' difference group $\Aut[X]$.

Our Galois Theorem~\ref{diff-Galois-thm-fields} for a difference field extension $L/k$  clarifies the connection between the difference profinite group $(\Gal(\forg{L}/\forg{k}),()^\sigma)$ and a suitable internal automorphism group.

\subsection{In search of a difference Galois theory}

Generally speaking, a \emph{Galois correspondence} is an equivalence of categories of the form
$$
\cC\simeq BG
$$
between a category $\cC$ of 
objects we wish to study and the category $BG$ of actions of a suitable algebraic `group like' 
object $G$, which allows us to say that $G$ \emph{classifies} the objects of $\cC$. %\mw{See \cite{borceux-janelidze}.}

Historically, the most important instances were the following. 

\begin{enumerate}
\item In \emph{classical Galois theory} of a Galois field extension $L/k$,  the subextensions are classified through closed subgroups of the profinite Galois group $$\Gal(L/k).$$ 
A more powerful statement is that there is an anti-equivalence of categories betweeen $k$-algebras \emph{split} by $L$ and profinite actions of $\Gal(L/k)$ (cf.~\ref{field-ext-gal-desc}).

\item In algebraic topology, given a path-connected space $X$ and a point $x\in X$, the \emph{fundamental group}
$$
\pi_1(X,x)
$$
classifies covering spaces of $X$.

\item In algebraic geometry, to a great extent uniting both previous theories, given a connected scheme $X$ and a geometric point $\bar{x}$ in $X$, \emph{Grothendieck's Galois theory} states that the profinite \emph{\'etale fundamental group} 
$$
\pi_1^\et(X,\bar{x})
$$
classifies the category of finite \'etale covers of $X$ through its finite continuous actions \cite{sga1}.
\end{enumerate}

Our objective is to classify suitably defined extensions of a difference ring %$k$ 
(thought of as `\'etale covers' of associated difference schemes). 

If we were to attempt to generalise Grothendieck's Galois theory for this task (bearing in mind that we should use internal/enriched Hom-objects and automorphism groups), we would  encounter the following obstacles.

\begin{enumerate}
\item 
Grothendieck's formalism from \cite{sga1} relies on working in a Boolean universe, where subobjects have complements, and \emph{connected} objects play a crucial role.  The category of difference sets is not Boolean, and, as we shall see later, the spectrum of a difference ring can be \emph{totally disconnected} with no fixed points nor fixed components, so we may not be able to select a base geometric difference point and consider the relevant fibre functor.
\item There exist algebraic difference field extensions with no finite difference subextensions, so the difference theory cannot mimic the classical passage through finite extensions/covers. 
\end{enumerate}

Obstacle (1) points us in the direction of considering a theory of a \emph{difference fundamental groupoid} as opposed to a difference group, while (2) indicates that it should be a difference \emph{topological} groupoid. Our paper pursues this line of research.  

On the other hand, the first author developed an approach to difference algebra through topos theory and categorical logic in \cite{ive-topos}, and the key idea from those areas is that in contexts with not enough points (`fixed points' when it comes to difference sets), everything is governed by \emph{internal localic groupoids}, as in \cite{joyal-tierney}, \cite{bunge-moerdijk}, \cite{bunge-92}, \cite{bunge-04}. In a forthcoming paper, the first author will formulate the theory of the \emph{difference localic fundamental group} reliant on the general theory from \cite{bunge-04}.

 \subsection{Difference Galois correspondence}
 
 Our development of \emph{difference Galois groupoids} is inspired by Magid's separable Galois theory of ring extensions \cite{magid}, which is formulated over an arbitrary ring and overcomes the difficulties of dealing with rings with a totally disconnected space of connected components. 
 
 The theory works particularly well using the framework of Janelidze's categorical Galois theory from \cite{janelidze} and \cite{borceux-janelidze}. We start by lifting a general context for Janelidze theory to the associated categories of difference objects in \ref{si-Janelidze}, and we apply this principle to transition from Magid's theory for rings to difference rings as follows.
 
 We consider the opposite category of the category of difference rings
 $$
 \cA=(\diff\Rng)^\op,
 $$
 which the reader can think of as the category of affine difference schemes, and the category
 $$
 \cP=\diff\Prof
 $$
 of difference profinite sets. The \emph{difference Pierce spectrum} is the functor
 $$
 S:\cA\to \cP,
 $$
 assigning to a difference ring $R$ the difference profinite set of connected components of $\spec(\forg{R})$ with the shift induced by $\sigma_R$ (equivalently, the spectrum of the Boolean algebra of idempotents of $R$ with an induced $\sigma_R$-action, as in  \ref{difference-context}). 
 
 Its right adjoint % \emph{difference ring of continuous functions} functor
 $$
 C:\cP\to \cA
 $$
 assigns to a difference profinite space $X$ the difference ring of continuous functions $$(\mathop{\rm Cont}(\forg{X},\Z), f\mapsto f\circ\sigma_X).$$
 This adjunction extends to adjunctions
 $$
 S_k\dashv C_k
 $$
 between slice categories $\cA_{\ov k}$ and $\cP_{\ov S(k)}$ for any $k\in\cA$. 
 
 Given a morphism $f:L\to k$ in $\cA$, we define the subcategory 
 $$
 \Split_k(f)
 $$
 of $\cA_{\ov k}$ consisting of those $k$-algebras $A$ which are \emph{split} by $f$ in the sense that the pullback $f^*(A)=L\otimes_kA$ is in the image of the functor $C_L$ (to be thought of as a `coproduct of copies of $L$ indexed by a difference profinite space').  Split objects play the role of separable/\'etale objects from classical Galois theories.
 
The next step is to identify morphisms $f:L\to k$ in $\cA$ of \emph{relative Galois descent} % for which the base change functor $f^*$ is monadic and which are suitable to 
that play the role analogous to normal/Galois objects in the classical theories (\ref{janelidze}). These are auto-split morphisms such that $\forg{L}/\forg{k}$ is \emph{componentially locally strongly separable}, a condition familiar from \cite{magid}.
  
Given a morphism $f\colon L\to k$ of relative Galois descent, the \emph{difference Galois groupoid} is 
$$
\begin{tikzcd}[cramped, column sep=normal, ampersand replacement=\&]
{\Gal[f]=\left(S(L\otimes_kL) \right.}\ar[yshift=2pt]{r}{} \ar[yshift=-2pt]{r}[swap]{} \&{\left.S(L)\right),}
\end{tikzcd}
$$
a difference profinite groupoid with object of objects $S(L)$ and object of morphisms $S(L\otimes_kL)$.

Our main Theorem~\ref{diff-Galois-thm} is the following. 
\begin{theorem*}[Difference Galois correspondence]
There is an equivalence of categories
$$
\Split_k(f)\simeq B\Gal[f]=[\Gal[f],\cP]
$$
between the opposite category of $k$-algebras $A$ split by $L$ and the category of difference profinite $\Gal[f]$-actions,  realised by the functor
$$
A\mapsto S(L\otimes_kA).
$$
\end{theorem*}

\subsection{Difference fundamental groupoid}

A separable closure of a difference ring $k$ is a difference ring extension $\bar{k}$ such that $\forg{\bar{k}}$ is a separable closure of $\forg{k}$. We show (\ref{diff-sep-cl-exists}) that separable closures exist, but they are not unique up to a $k$-isomorphism.  

The \emph{fundamental difference groupoid} is the difference profinite groupoid
$$
\pi_1(k,\bar{k})=\Gal[\bar{k}\to k].
$$
It is aptly named because it classifies \emph{difference locally \'etale $k$-algebras}, as shown in 
Corollary~\ref{diff-pi1-correspondence}.

\subsection{Finite difference generation and presentation}

In a category of algebraic structures, \emph{freely generated} objects are typically obtained by applying the left adjoint of a suitable forgetful functor, while \emph{presentations} of objects are defined in terms of coequaliser diagrams involving free objects. By duality,  in a geometric context we should consider the right adjoint of the forgetful functor and equaliser diagrams. 

In \ref{ss:freely-si-gen}, for a suitable category $\cC$, we find left and right adjoints to the forgetful functor $\forg{\,}:\diff\cC\to \cC$,
$$
\ssig{\,}\dashv\forg{\,}\dashv \psig{\,}
$$
and we view them as \emph{free constructions} of difference objects. 

Moreover, in \ref{ss:dir-si-pres}, we identify a right adjoint $\psig{\,}$ to the forgetful functor 
$$
\forg{\,}^\Corr:\diff\cC\to \Corr(\cC), \ \ \  X\mapsto (X\xleftarrow{\id} X \xrightarrow{\sigma_X} X).
$$ 
into the category of correspondences in $\cC$, %(denoted by the same symbol as above), 
and, dually, a left adjoint $\ssig{\,}$ to the forgetful functor from $\diff\cC$ to cocorrespondences in $\cC$. 
These are viewed as constructions of \emph{directly presented difference objects}. 

The rigorous treatment of finite $\diff$generation and finite $\diff$presentation makes it straightforward to show that these constructions commute with the functors establishing the equivalence of categories in our difference Galois correspondence, establishing Corollaries~\ref{eq-with-subshifts} and \ref{fp-corr-sft} as follows.
\begin{corollary*}
Through the difference Galois correspondence, finitely $\diff$generated objects of $\Split_k(f)$ correspond to finitely $\diff$generated $\Gal[f]$-actions, and similarly for finitely $\diff$presented objects.
\end{corollary*}

\subsection{Symbolic dynamics}

The \emph{one-sided full shift} on a finite set $E_0$ is the profinite space
$$
X=E_0^\N
$$
together with the shift 
$$
\sigma:X\to X, \ \ \ \ (x_0,x_1,x_2,\ldots)\mapsto (x_1,x_2,\ldots).
$$
It is a profinite space with a continuous self-map, and we view it as an object of the category $\diff\Prof$.

A \emph{subshift} is a closed subset of some full shift which is {stable} under the shift map, considered as a subobject of a full shift in $\diff\Prof$. 

A \emph{subshift of finite type} is a subshift isomorphic to the space $X_\Gamma$ of infinite paths of a finite directed graph 
$$
\Gamma= (X_0\leftarrow E_0\rightarrow X_0)
$$
with vertices $X_0$ and edges $E_0$.

We observe that key concepts and objects of difference algebra, symbolic dynamics and algebraic dynamics are all obtained through the same categorical constructions of  $\diff$generation and $\diff$presentation, applied in a suitable context.

\begin{enumerate}
\item When $\cC=\Set$, the full shift on $E_0$ is in fact the difference set $\psig{E_0}$, and a subshift is a subobject of some $\psig{E_0}$. A directed graph $\Gamma$ as above can be viewed as a correspondence $X_0\leftarrow E_0\rightarrow X_0$ in $\Set$, and the subshift of finite type associated to $\Gamma$ is directly presented as
$$
X_\Gamma=\psig{X_0\leftarrow E_0\rightarrow X_0}.
$$
\item Given a difference ring $k$, a $k$-algebra $A$ is \emph{finitely $\diff$generated} if it is a quotient of $\ssig{A_0}^k$ for some finitely generated $\forg{k}$-algebra $A_0$, or, equivalently, if there exists a finite tuple $a$ in $A$ such that $A=k[a,\sigma(a),\sigma^2(a),\ldots]$. 

A $k$-algebra $A$ is \emph{finitely $\diff$presented} if it is a quotient of a finitely $\diff$generated $k$-algebra by a finitely $\diff$generated difference ideal $I$, i.e, $\sigma(I)\subseteq I$ and there exist a finite tuble $b$ in $I$ such that $b,\sigma(b),\ldots$ generates $I$ as an ideal.
We show in \ref{fin-pres-dir-pres} that $A$ is \emph{finitely $\diff$presented} if and only if it is of {the} form $$\ssig{A_0;C_0}^k$$ for some $k$-cocorrespondence with $A_0$ and $C_0$ finitely presented $\forg{k}$-algebras.
\item When $k$ is a difference ring, and $X_0\leftarrow C_0\to X_0$ is a $k$-correspondence in the category of affine $\forg{k}$-schemes, the \emph{path space}  %(cf.~\ref{path-space}) 
of this correspondence is the object of $\diff\Rng^\op$ (thought of as a `difference scheme')
$$
\psig{X_0;C_0}_k.
$$
This object is used in algebraic dynamics to construct arboreal representations as in \cite{ingram}.
\end{enumerate}

The above Corollary now yields the following \emph{translation mechanism between difference algebra and symbolic dynamics} (cf.~\ref{theo:translation-mech}). 

Given a finitely $\diff$generated difference field extension $L/k$ {with $\forg{L}/\forg{K}$ Galois}, the difference Galois group
$$
G=\Gal[L/k]
$$
is a \emph{group subshift of finite type}, and we have  anti-equivalences of categories between
\begin{enumerate}
\item finitely $\diff$generated $k$-algebras split by $L$ and subshifts with an action of $G$;
\item finitely $\diff$presented $k$-algebras split by $L$ and subshifts of finite type with an action of $G$.
\end{enumerate}
For the absolute case, if we choose a separable closure $\bar{k}$ of $k$, the difference fundamental group
$$
\pi_1(k,\bar{k})
$$
is a limit of group subshifts of finite type and we have analogous statements for finitely $\diff$generated/$\diff$presented \'etale $k$-algebras and subshifts/subshifts of finite type with $\pi_1(k,\bar{k})$-action.

{In particular, for $\forg{k}$ algebraically closed, finitely $\diff$generated/$\diff$presented \'etale $k$-algebras are equivalent to subshifts/subshifts of finite type.}

\subsection{Interactions between symbolic dynamics and difference algebra}

The connection between symbolic dynamics and difference algebra is not well-known. So far researchers on one side have mostly been unfamiliar with the other side. Therefore, a main goal of this article is also to formalize and popularize this connection. It seems clear that an understanding of the methods and perspectives of both sides will be mutually beneficial. For example, as explained in \cite[0.7]{Gromov:TopologicalInvariants}, the definition of the mean dimension, by now a standard invariant in dynamics, was inspired by the notion of dimension in difference algebra. 
It is also curious to see how certain constructions correspond to each other. For example, in difference algebra it is a standard procedure to rewrite a finite system of algebraic difference equations as a system of order $1$ by adding variables. In the symbolic dynamics world, this corresponds to the standard procedure of rewriting a subshift of finite type as a $1$-step subshift.

One of the possible explanations why the connection between difference algebra and symbolic dynamics has widely remained unnoticed is that difference algebra has long had a strong focus on fields and integral domains. Only in the last decades the importance of studying difference algebras with many zero-divisors has been recognized, for example in  \cite{singer}, \cite{OvchinnikovPogudinScanlon:DifferenceElimination}, \cite{PogudinScanlonWibmer:Solving}, \cite{Wibmer:Dimension}.

A particular case of our Galois correspondence theorem shows that, for a fixed difference field $k$, the category of subshifts is anti-equivalent to a rather special class of  finitely $\diff$generated $k$-algebras having many zero divisors (namely, the category of finitely $\diff$generated \'{e}tale $k$-algebras split by $\id\colon k\to k$). This shows that the study of difference algebras with many zero divisors is a very rich albeit difficult subject as it includes all of symbolic dynamics.

Using the above translation mechanism between the two areas, we establish a number of results in difference algebra %by using known facts and 
by adapting methods from symbolic dynamics.
%\begin{enumerate}
%\item 

(1) Given a difference ring $A$, we consider the spectrum of the underlying ring with the action induced by $\sigma_A$ as a difference topological space
$$
X=\spec(A)=(\spec(\forg{A}),\spec(\sigma_A))\in\diff\Topl.
$$
We consider the \emph{$\diff$topology} on $X$, whose closed subsets are the $\sigma_X$-stable closed subsets of $X$.
Usings facts on irreducible components of subshifts of finite type, we prove (\ref{finitely-si-comp}) that,
if $A$ is a finitely $\diff$presented \'etale difference algebra, then $\spec(A)$ has finitely many connected components in the $\diff$topology.

(2) By using facts on the  irreducible decomposition and on periodic points of subshifts of finite type, we address (\ref{strong-core-strong}) our own conjecture \cite[1.19]{ive-mich-babb} by proving that, 
 if $A$ is a finitely $\diff$presented \'etale difference algebra over a difference field $k$, then its \emph{strong core}
$$
\pi_0^\sigma(A)
$$
is a strongly $\diff$\'etale $k$-algebra.

%\item
(3) We define the notion of \emph{entropy}  
$$
h(A)
$$
of a finitely {$\diff$generated} \'etale difference algebra $A$ over a difference field $k$ in terms of  \emph{topological entropy} of its associated subshift  (\ref{alg-entropy}). 

We show (\ref{entropy-ld}) that,  for a finitely $\diff$generated separable extension $L/k$ of difference fields, 
$$
h(L)=\log(\ld(L/k)),
$$
where $\ld(L/k)$ is the \emph{limit degree}.
Hence, the notion of entropy of {difference} algebras generalises the concept of limit degree, well-known in difference algebra.

(4) Let the base difference field be $k=(\F_q,\id)$ and write $\bar{k}_n=(\bar{\F}_q,\mathop{\rm Frob}_q^n)$ for the algebraic closure of the finite field $\F_q$ equipped with the %\mw{$n$-th} 
power of the Frobenius map $x\mapsto x^{q^n}$. 
Let $A$ be a $k$-algebra,  and consider its number of $\bar{k}_n$-rational points,
$$
N_n=|k\Alg(A,\bar{k}_n)|.
$$
 We define the \emph{difference zeta function} of $A$ as the formal power series
$$
Z_A(t)=\exp\left(\sum_{n=1}^\infty \frac{N_n}{n} t^n\right).
$$ 
Ruelle \cite{ruelle-zeta} explains how the definition of the \emph{dynamical zeta function} of a discrete dynamical system $X$, which counts periodic points as fixed points of $\sigma_X^n$, is motivated by Weil conjectures, where $\F_{q^n}$-rational points of a scheme are thought of as fixed points of $\mathop{\rm Frob}_q^n$.
% \mw{?? Remove the whole next sentence because 1: at this point it is not clear what our problem is so this sentence does not make too much sense, maybe better to say something like ``Understanding the numbers $N_n$'' 2: below does not make sense, its already defined above} Our problem significantly generalises Weil conjectures, and we will present an appropriately modified dynamical zeta function below. 

In model theory, the study of the numbers $N_n$, as $n$ varies, is associated with Macintyre's conjecture that the theory of existentially closed difference fields is the limit theory of difference fields $\bar{k}_n$ posed in \cite{angus} and proved by Hrushovski in \cite{udi-frob}.

In number theory and algebraic geometry, the problem can be reformulated via counting points on algebraic correspondences twisted by powers of Frobenius, which is a subject of Deligne's conjecture proved in \cite{Pink} and \cite{Fujiwara} under certain properness assumptions on the correspondence. Recently, \cite{Varshavsky} and \cite{shuddhodan} worked toward algebraic-geometric proofs of Hrushovski's result in its full generality.

\begin{theorem*}[{\ref{zeta-near-rat}}] If $A$ is a {finitely $\diff$presented \'etale} $k$-algebra, then $Z_A(t)$ is \emph{near-rational}, i.e., its logarithmic derivative is rational.
\end{theorem*}

Our difference Galois correspondence sends $A$ to a subshift of finite type $X$, and the absolute Frobenius $a\mapsto a^q$ on $A$ corresponds to a subshift map $\phi: X\to X$. 
The key observation is that
$$
N_n=|\begin{tikzcd}[cramped, column sep=normal, ampersand replacement=\&]
{\Eq\left(X\right.}\ar[yshift=2pt]{r}{\sigma} \ar[yshift=-2pt]{r}[swap]{\phi^n} \&{\left.X\right)}
\end{tikzcd}|=|\{x\in X: \sigma(x)=\phi^n(x)\}|.
$$
The classical dynamical zeta {function} of $X$ counts the numbers $|\Fix(\sigma_X^n)|$, so our difference zeta {function} is indeed a novel object. Nevertheless, we are able to adapt the methods used to study the dynamical zeta {function} and complete our proof.

%\end{enumerate}

\subsection{Future work on interactions with dynamics}

In this paper, we {use} basic results from symbolic dynamics to prove important results in difference algebra. By discussing topological entropy and dynamical zeta functions, we barely {scratch} the surface of the full \emph{thermodynamic formalism} in topological dynamics as set out in \cite{ruelle-book}.  We believe that difference algebra counterparts of concepts related to ergodic theory, measure-theoretic entropy, (relative) pressure, (relative) variational principle, equilibrium states will lead to new results in difference algebra.

On the other hand, {we hope that} the categorical formulation of difference algebra, our Galois-theoretic perspective, the use of internal automorphism groups, and various decomposition theorems from difference algebra will lead to novel results in symbolic dynamics. 

As a proof of concept, we point the reader to \cite[7.2]{michael-gsft}, which shows that Babbitt's decomposition for a finitely $\diff$generated Galois extension $L/k$ of difference fields corresponds (via our difference Galois correspondence) to a suitable subnormal series for the group subshift of finite type $\Gal[L/k]$.

\subsection{Layout of the paper}

Section~\ref{s:diffalg} sets out the categorical foundations of difference algebra and introduces key objects and techniques to be used throughout the paper. In order to expedite presentation, we moved all the proofs to  Appendix~\ref{ap:diffalg}.

Section~\ref{s:galth} contains our main results on difference Galois theory as summarised above. 

Section~\ref{s:symb} recalls the fundamental results of symbolic dynamics  regarding shift spaces and subshifts of finite type, while Section~\ref{s:interactions} contains a collection of our key results on the interactions between difference algebra and symbolic dynamics. 

The reader only interested in connections to symbolic dynamics is invited to start reading from Section~\ref{s:symb} to fix the notation, and then move on to Section~\ref{s:interactions}, aided by the \emph{translation mechanism} Theorem~\ref{theo:translation-mech} that fledges out the relevance of our difference Galois theory from Section~\ref{s:galth} to symbolic dynamics.

The theory of internal categories, groupoids and their actions used throughout the paper is given in Appendix~\ref{s:groupoids},  and Appendix~\ref{sec: Appendix C} provides background on enriched function spaces in difference algebra needed for a full understanding of the Galois theory of difference fields in \ref{diff-Galois-thm-fields}.

\section{Difference Algebra}\label{s:diffalg}

In this section, we introduce categories of difference objects, difference algebraic structures, the notion of free and finite $\diff$generation, $\diff$presentation and direct $\diff$presentation. We consider difference groupoids and their actions, which will play a major role in the development of difference Galois theory in Section~\ref{s:galth}.

We present definitions and results only, and we give all the proofs in Appendix~\ref{ap:diffalg}.

\subsection{Categories of difference objects}\label{cat-diff-obj}

The \emph{category of difference objects} associated to a category $\cC$ is the category
$$
\diff\cC
$$
whose objects are pairs $(X,\sigma_X)$, where $X\in\Ob(\cC)$, and $\sigma_X\in\cC(X,X)$. A morphism $f:(X,\sigma_X)\to (Y,\sigma_Y)$ is a commutative diagram
 \begin{center}
 \begin{tikzpicture} 
\matrix(m)[matrix of math nodes, row sep=2em, column sep=2em, text height=1.5ex, text depth=0.25ex]
 {
 |(1)|{X}		& |(2)|{Y} 	\\
 |(l1)|{X}		& |(l2)|{Y} 	\\
 }; 
\path[->,font=\scriptsize,>=to, thin]
(1) edge node[above]{$f$} (2) edge node[left]{$\sigma_X$}   (l1)
(2) edge node[right]{$\sigma_Y$} (l2) 
(l1) edge  node[above]{$f$} (l2);
\end{tikzpicture}
\end{center}
in $\cC$, i.e., $f\in\cC(X,Y)$ such that $f\circ\sigma_X=\sigma_Y\circ f$. %\mw{-p}
We have the forgetful functor 
$$
\forg{\, }:\diff\cC\to\cC,\ \ \ 
\forg{(X,\sigma)}=X.
$$
\begin{notation}
In the sequel, whenever we write $$X\in \diff\cC,$$ {$X$ is a difference object, whose under lying object 
 $$\forg{X}\in\cC$$
 is an object of $\cC$ equipped with a
 $\cC$-endomorphism $\sigma_X$, i.e.,
 $$X=(\forg{X},\sigma_X).$$}
\end{notation}

\begin{notation}
For an object $X$ of a category $\cC$, we denote by $\cC_{\ov X}$ the \emph{slice category} over $X$. The objects of $\cC_{\ov X}$ are the $\cC$-morphism with codomain $X$ and the morphisms of $\cC_{\ov X}$ are the $\cC$-morphisms over $X$. Dually, we  have the \emph{coslice category} $X/\cC$ whose objects are morphisms with domain $X$.
\end{notation}

\begin{remark}\label{forg-creates-limits}
%\mw{?? maybe include proof} 
The forgetful functor $\forg{\,}$ \emph{creates limits and colimits}, i.e., it preserves and reflects limits/colimits and reflects the existence of limits/colimits. More concretely, if $\cC$ has limits or colimits of a certain type, so does $\diff\cC$.  
In particular:
\begin{enumerate}
\item if $e$ is a terminal object of $\cC$, then $(e,\id)$ is a terminal object of $\diff\cC$;
\item if $\cC$ has  products and $X,Y\in \diff\cC$  then 
$$
(\forg{X}\times \forg{Y},\sigma_X\times\sigma_Y)
$$ is their product in $\diff\cC$;
\item If $\cC$ has pullbacks/fibre products, $S\in\diff\cC$,  and $X\to S, Y\to S$ are objects of $(\diff\cC)_{\ov S})$, then $$(\forg{X}\times_{\forg{S}}\forg{Y},\sigma_X\times\sigma_Y)$$ is their pullback/product in 
 $(\diff\cC)_{\ov S}$.
\end{enumerate}
\end{remark}

\begin{example}
The most commonly used categories of difference objects in this paper will be:
\begin{enumerate}
\item the category of \emph{difference sets} $\diff\Set;$ %\mw{-p}
\item the category of \emph{difference groups} $\diff\Gp;$ %\mw{-p}
\item the category of \emph{difference rings} $\diff\Rng$, {where for the purpose of this paper all rings are assumed to be unital and commutative.}
\end{enumerate}

\end{example}

\subsection{Difference algebraic structures}

The theories of groups, rings and algebras over a given ring are \emph{algebraic} in the sense that their axioms can be expressed in terms of commutative diagrams involving finite limits (most commonly products). Hence, in any category with finite limits, we can speak of \emph{group objects}, \emph{ring objects}, and \emph{algebra objects} over a given ring object.

\begin{remark}
Given that $\diff\Set$ has finite limits, we have that
\begin{enumerate}
\item $\diff\Gp$ is the category of group objects in $\diff\Set$;
\item $\diff\Rng$ is the category of ring objects in $\diff\Set$;
\item given $k\in\diff\Rng$, the category of $k$-algebras
$$
k\Alg
$$
is the category of algebra objects over $k$ in $\diff\Set$. {Of course, this agrees with the usual notion of a difference algebra as in \cite{levin}: a difference ring with a $\forg{k}$-algebra structure such that the $\forg{k}$-algebra structure morphism is a morphism of difference rings.}
\end{enumerate} 
\end{remark}

\begin{remark}
If $k\in\diff\Rng$ is such that $\sigma_k$ is not the identity, then the category $k\Alg$ is not a category of difference objects of some natural category. Nevertheless, we find it useful to view it as an \emph{undercategory/coslice} of $\diff\Rng$ under $k$,
$$
k\Alg\simeq  k/\diff\Rng.
$$
%whose objects are $\diff\Rng$-morphisms $k\to A$, and morphisms between $k\xrightarrow{\varphi} A$ and $k\xrightarrow{\varphi'} A'$ are 
\end{remark}

\subsection{Difference categories and groupoids}\label{diff-cats}

In subsequent definitions, we use the notions of \emph{internal categories/groupoids} and \emph{internal diagrams/actions} familiar in category theory, as detailed in Appendix~\ref{internal-cats}.
\begin{definition}
A \emph{difference category} is an internal category $\C$ in  $\diff\Set$. A \emph{difference groupoid} is an internal groupoid $\bbG$ in the category $\diff\Set$. 
\end{definition}

\begin{remark}
 Explicitly, following \ref{internal-cats}, if $\C$ is a difference category, the object of objects $C_0$ and the object of morphisms $C_1$ are difference sets, and all structure morphisms $d_0,d_1, e, m$ are difference maps. Writing
 $$
 \cC=\forg{\C}
 $$
 for the underlying small category with objects $\forg{C_0}$ and morphisms $\forg{C_1}$, the morphisms
 ${\sigma_{C_0}}$ and ${\sigma_{C_1}}$ define a functor
 $
 \sigma_\cC:\cC\to \cC,
 $
 so a difference category $\C$ can be identified with the datum
 $$
 (\cC,\sigma_\cC),
 $$
 as noted in \cite{ive-topos}.
 \end{remark}

\begin{definition}
A \emph{difference diagram} on a difference category $\C$ is an object
$$
F\in [\C,\diff\Set],
$$
where we write $[\C,\diff\Set]$  for the category of internal diagrams on an internal category $\C$ in $\diff\Set$, as in \ref{internal-diags}.
If $\C$ is a difference groupoid, we may call $F$ a \emph{difference $\C$-action}. 
\end{definition}

\begin{remark}
Let $F$ be a difference diagram on a difference category $\C$. Then the underlying structure
$$
\cF=\forg{F}
$$ 
is a diagram on $\cC=\forg{\C}$, and it is readily verified (\cite{ive-topos}) using  \ref{internal-diags} that $\sigma_{F_0}$ defines a natural transformation
$$
\sigma_\cF: \cF\to \sigma_\cC^*\cF=\cF\circ \sigma_\cC,
$$
and that $F$ can be identified with a $\sigma_\cC$-equivariant diagram
$$
(\cF,\sigma_\cF).
$$
%as noted in \cite{ive-topos}.
\end{remark}

\subsection{Difference group actions}\label{diff-gp-actions}

In view of \ref{diff-cats}, a difference group is a difference groupoid whose object of objects is a point (a terminal object in $\diff\Set$), and a difference group action is a special case of a difference groupoid action. We provide details in the following.
 
\begin{definition} \label{defi: diff group action}
Let $G\in\diff\Gp$ be a difference group (considered as a group object in $\diff\Set$) with multiplication $\diff\Set$-morphism $m_G:G\times G\to G$ and identity section $e_G:1\to G$, and let $X\in\diff\Set$ be a difference set. An \emph{action} of $G$ on $X$ is a $\diff\Set$-morphism
$$
\mu: G\times X\to X
$$
that makes the diagrams
 \begin{center}
 \begin{tikzpicture} 
\matrix(m)[matrix of math nodes, row sep=2em, column sep=3em, text height=1.9ex, text depth=0.25ex]
 {
 |(1)|{G\times G\times X}		& |(2)|{G\times X} 	\\
 |(l1)|{G\times X}		& |(l2)|{X} 	\\
 }; 
\path[->,font=\scriptsize,>=to, thin]
(1) edge node[above]{$\id\times \mu$} (2) edge node[left]{$m_G\times\id$}   (l1)
(2) edge node[right]{$\mu$} (l2) 
(l1) edge  node[below]{$\mu$} (l2);
\end{tikzpicture}
 \begin{tikzpicture} 
\matrix(m)[matrix of math nodes, row sep=2em, column sep=3em, text height=1.9ex, text depth=0.25ex]
 {
 |(1)|{1\times X}		&  |(2)|{X}	\\
 |(l1)|{G\times X}		& |(l2)|{X} 	\\
 }; 
\path[->,font=\scriptsize,>=to, thin]
(2) edge (1) edge node[right]{$\id$} (l2) 
(1)  edge node[left]{$e_G\times\id$}   (l1)
(l1) edge  node[below]{$\mu$} (l2)
;
\end{tikzpicture}
\end{center}
commutative. This is equivalent to saying that $\forg{G}$ acts on $\forg{X}$ and, crucially, writing $g.x=\mu(g,x)$,
$$
\sigma_X(g.x)=\sigma_G(g).\sigma_X(x).
$$
\end{definition}

\subsection{Free generation of difference objects}\label{ss:freely-si-gen}

\begin{proposition}\label{psig-adj}
If $\cC$ has countable products, the forgetful functor $\forg{\,}: \diff\cC\to \cC$ admits a right adjoint
$$
\forg{\,}\dashv \psig{\,}.
$$
\end{proposition}

\begin{proposition}\label{psig-prop}
Let $\cC$ be a category and $S\in\diff\cC$. 
If $\cC$ has countable pullbacks {(i.e., the slice categories have countable products)}, the forgetful functor on the slice category  
$$\forg{\,}_S:\diff\cC_{\ov S}\to\cC_{\ov\forg{S}}$$ admits a right adjoint
$$
\forg{\,}_S \dashv \psig{\,}_S.
$$
\end{proposition}

By reversing all arrows in the above considerations, we can consider the forgetful functor on the coslice category
$$
\forg{\,}^S:S/\diff\cC\to \forg{S}/\cC
$$
and we obtain the following.
\begin{proposition}\label{ssig-prop}
If $\cC$ has countable pushouts, then $\forg{\,}^S$ admits a left adjoint 
$$
\ssig{\,}^S\dashv \forg{\,}^S.
$$
\end{proposition}

An important special case is for difference rings and algebras.

\begin{corollary}\label{sigmaization-algebras}

%The natural forgetful functor
%$$
%\forg{\,}:\diff\Rng\to\Rng
%$$
%returns the underlying ring of a difference ring. 
If $k$ is a difference ring, the natural forgetful functor
$$
\forg{\,}^k=\forg{\,}:k\Alg\to\forg{k}\Alg
$$
admits a left adjoint
$$
\ssig{\,}^k \dashv \forg{\,}^k.
$$
\end{corollary}

%\mw{Objects of $\diff\cC_{\ov S}$ of the form $\psig{X}_S$ for some object $X$ of $\cC$ will be called \emph{freely $\diff$generated over $S$}. Similarly, objects of $S/\diff\cC$ of the form $\ssig{X}^S$ for some object $X$ of $\cC$ will are called freely $\diff$generated over $S$.}

%\begin{proof}
%Suppose $A_0$ is a $\forg{k}$-algebra and denote $\varsigma=\sigma_k$. 
%Let $A_i=(A_0)_{\varsigma^i}=A_0\otimes_kk$ be the base change of $A_0$ along $\varsigma^i:k\to k$, and let $\sigma_i:A_i\to A_{i+1}$ be the induced ring morphism. Consider the tensor product over $k$ (the coproduct in $k\Alg$) 
%$$
%\ssig{A_0}^k=\otimes_{i\in\N}A_i,
%$$
%together with the ring morphism $\sigma:\ssig{A_0}^k\to \ssig{A_0}^k$,
%$$
%\sigma(a_0\otimes a_1\otimes\cdots)=1\otimes\sigma_0(a_0)\otimes\sigma_1(a_1)\otimes\cdots.
%$$
%This construction extends to a functor
%$$
%\ssig{\,}^k:\forg{k}\Alg\to k\Alg
%$$
%which is left adjoint to $\forg{\,}^k$, i.e., for any $A_0\in\forg{k}\Alg$ and $B\in k\Alg$,
%$$
%k\Alg(\ssig{A_0}^k,B)\simeq\forg{k}\Alg(A_0,\forg{B}).
%$$
%\end{proof}

\subsection{Free generation of difference diagrams and groupoid actions}

\begin{proposition}\label{psig-diags}
Let $\cS$ be a category with countable pullbacks, and let $\C$ be an internal category in $\diff\cS$. The forgetful functor
$$
\forg{\,}_\C: [\C,\diff\cS]\to [\forg{\C},\cS]
$$
admits a right adjoint
$$
\forg{\,}_\C\dashv \psig{\,}_\C.
$$
\end{proposition}

\begin{example}\label{psig-G-sets}
Let $G\in\diff\Gp$ be a difference group. The forgetful functor 
$$
\forg{\,}_G: [G,\diff\Set]\to [\forg{G},\Set]
$$
from difference $G$-actions to $\forg{G}$-actions admits a right adjoint
$$
\psig{\,}_G:[\forg{G},\Set]\to [G,\diff\Set], \ \ \psig{X_0}=\prod_{i\in\N} X_i,
$$
where $X_i$ is a copy of $X_0$ with action of $\forg{G}$ twisted via $\sigma_G^i$, and the difference operator $\sigma_{\psig{X_0}}$ is the left shift. 

\end{example}

\subsection{Direct presentation of difference objects}\label{ss:dir-si-pres}
%\mw{?? Generally speaking, I don't know why we have to deal so much with direct presentations, would it not be easier to just use presentations i.e., coequalizers?}
Let $\cC$ be a category and let $S\in\diff\cC$.
\begin{definition}\label{def-corresp}
We define the category 
$$\Corr_S=\Corr_S(\cC)$$
of \emph{self-$S$-correspondences in $\cC$} as follows.
\begin{enumerate}
\item\label{dpdiag} An object $(X_0;C_0)$ is a commutative diagram 
\begin{center}
 \begin{tikzpicture} 
\matrix(m)[matrix of math nodes, row sep=2em, column sep=2em, text height=1.5ex, text depth=0.25ex]
 {
 |(2)|{X_0}		& |(3)|{C_0} 	& |(4)|{X_0}\\
 |(l2)|{S}		& |(l3)|{S} 		& |(l4)|{S}\\
 }; 
\path[->,font=\scriptsize,>=to, thin]
(3) edge node[above]{${p_0}$} (2) (3) edge node[above]{${\sigma_0}$}   (4)
(l3) edge node[above]{${\mathop{\rm id}}$} (l2) (l3) edge node[above]{${\sigma_S}$} (l4)
(2) edge  (l2) (3) edge (l3) (4) edge  (l4);
\end{tikzpicture}.
\end{center}
consisting of $\cC$-morphisms.
%Note that $\pi_1$ is a morphism of $S$-schemes, while each $\sigma$ is only a $\varsigma$-linear scheme morphism. 

\item A morphism $f:(X_0,C_0)\to(Y_0,D_0)$ consists of $\cC$-morphisms $f_0:X_0\to Y_0$, $\hat{f}_0:C_0\to D_0$, making the diagram
$$
 \begin{tikzpicture}
[cross line/.style={preaction={draw=white, -,
line width=4pt}}]
\matrix(m)[matrix of math nodes, row sep=.9em, column sep=.5em, text height=1.5ex, text depth=0.25ex]
{			& |(x0)| {X_0}	&			& |(x1)| {C_0} 	&			& |(x0s)| {X_0}	\\   [.2em]
|(y0)|{Y_0} 	&			& |(y1)|{D_0} 	&			&  |(y0s)| {Y_0}	&			\\  [.4em]
%		& |(d4)| {X}	&			& |(d3)| {Y}	\\   %[.8cm]
			& |(s0)|{S} 		&			& |(s1)|{S} 	&			& |(s0s)|{S} 				\\};
\path[->,font=\scriptsize,>=to, thin]
(x1) edge node[above]{$p_0$} (x0) edge node[above]{$\sigma_0$} (x0s) 
	edge node[left,pos=0.3]{$\hat{f}_0$} (y1) edge (s1)
(x0) edge node[left,pos=0.3]{$f_0$} (y0) edge (s0)
(x0s) edge node[left,pos=0.3]{$f_0$} (y0s) edge (s0s)
(y0) edge (s0)
(y0s) edge (s0s)
(y1) edge [cross line] node[above,pos=0.2]{$p_0$} (y0) edge [cross line] node[above,pos=0.8]{$\sigma_0$} (y0s)  edge (s1)
(s1) edge node[above]{$\mathop{\rm id}$} (s0) edge node[above]{$\sigma$} (s0s) 
%(d4) edge node[pos=0.25,left]{$\lexp{a}{\sigma}$} (u4)  edge node[pos=0.25, below]{$\lexp{a}{\varphi}$} (d3)

%(u1) edge [cross line] node[pos=.8, above]{$\lzexp{a}{\varphi}{0}$} (u2)  edge node[auto]{$i$} (u4)
	
%(d2) edge [cross line] node[pos=0.75, right]{$\lzexp{a}{\tau}{0}$} (u2)  edge node[below]{$j$} (d3)
%(u4) edge node[above]{$\lexp{a}{\varphi}$} (u3)	
%(u2) edge node[pos=0.4,above]{$j$} (u3)	
%(d3) edge node[right]{$\lexp{a}{\tau}$} (u3);
;
\end{tikzpicture}
$$ 
commutative.
\end{enumerate}
%The category of \emph{direct presentations} is the full subcategory $\sdir$ of $\Corr$ consisting of those objects $(X,\Sigma)$ for which the diagram in
%(\ref{dpdiag}) induces a closed immersion $X_1\hookrightarrow X_0\times_SX_0$ %$X_1\hookrightarrow X_0\times_SX_{0,\varsigma}$ 
%for every $\sigma\in\Sigma$.
\end{definition}

\begin{proposition}\label{dir-gen-adj}
If $\cC$ has finite limits and countable  products, the forgetful functor 
$$
\forg{\,}_S^\Corr:\diff\cC_{\ov S}\to\Corr_S
$$
which sends an object $X\to S$ to the $S$-correspondence $X\xleftarrow{\id}X\xrightarrow{\sigma_X}X$
admits a right adjoint
$$
\forg{\,}_S^\Corr \dashv \psig{\,}_S.
$$
\end{proposition}

By reversing all arrows in the above, and by defining a suitable category of cocorrespondences, we obtain the following.

\begin{proposition}\label{dir-pres-adj}
If $\cC$ has finite colimits and countable coproducts, and $S\in \diff\cC$, the forgetful functor
$$
\forg{\,}^S_\coCorr:S/\diff\cC\to \coCorr^S
$$
which sends an object $S\to A$ to the $S$-cocorrespondence $A\xrightarrow{\id}A\xleftarrow{\sigma_A}A$ admits a left adjoint
$$
\ssig{\,}^S\dashv  \forg{\,}^S_\coCorr.
$$ 
\end{proposition}

\subsection{Direct presentation of difference diagrams and groupoid actions}

\begin{definition}\label{corr-diagrams}
Let $\cS$ be a category with pullbacks, and let $\C$ be an internal category in $\diff\cS$.
The objects of the category 
$$
\Corr_\C([\forg{\C},\cS])
$$ 
are spans of the form 
$$
\cF_0\xleftarrow{p_0} \cC_0\xrightarrow{\sigma_0} \sigma_\C^*\cF_0
$$
in the category $[\forg{\C},\cS]$. Morphisms of such spans are defined by analogy to \ref{def-corresp}. 
\end{definition}

\begin{proposition}\label{dir-pres-diag-adj}
If $\cS$ has finite limits and countable  products, and $\C$ is an internal category in $\cS$, the forgetful functor 
$$
\forg{\,}_\C^\Corr:[\C, \diff\cS]\to\Corr_\C([\forg{\C},\cS])
$$
which sends a internal diagram $F$ on $\C$ to the $\C$-correspondence $\forg{F}\xleftarrow{\id}\forg{F}\xrightarrow{\sigma_F}\sigma_\C^*\forg{F}$
admits a right adjoint
$$
\forg{\,}_\C^\Corr \dashv \psig{\,}_\C.
$$
\end{proposition}
%\begin{proof}
%We let
%$$
%\psig{\cF_0\xleftarrow{p_0} \cC_0\xrightarrow{\sigma_0} \sigma_\C^*\cF_0}_\C=
%\begin{tikzcd}[cramped, column sep=normal, ampersand replacement=\&]
%{\Eq\left(\psig{\cC_0}_\C\right.}\ar[yshift=2pt]{r}{\text{shift}\circ\psig{p_0}} \ar[yshift=-2pt]{r}[swap]{\psig{s_0}} \&{\left.\psig{\sigma_\C^*\cF_0}_\C\right)}
%\end{tikzcd}
%$$
%in the category $[\C, \diff\cS]$. The verification that this defines a right adjoint of the above forgetful functor is purely formal, following steps analogous to the proof of \ref{dir-gen-adj}.
%\end{proof}

\subsection{Freely generated and directly presented difference objects}\label{free-dirpres-diffob}

Working in a `geometric' context as in \ref{psig-adj}, \ref{psig-prop}, \ref{psig-diags}, \ref{psig-G-sets}, we shall say that an object is
\begin{enumerate}
\item \emph{freely $\diff$generated}, if it is in the essential image of the appropriate functor $\psig{\,}_*$;
\item \emph{freely $\diff$presented}, if it is an equaliser of a pair of morphisms between two freely $\diff$generated objects;
\item \emph{directly $\diff$presented}, if it is the essential image of the appropriate functor $\psig{\,}_*$ on correspondences (as in \ref{dir-gen-adj} and \ref{dir-pres-diag-adj}).  
\end{enumerate}

Dually, in an `algebraic' context as in \ref{ssig-prop}, \ref{sigmaization-algebras}
we shall say that an object is
\begin{enumerate}
\item \emph{freely $\diff$generated}, if it is in the essential image of the appropriate functor $\ssig{\,}^*$;
\item \emph{freely $\diff$presented}, if it is a coequaliser of a pair of morphisms between two freely $\diff$generated objects;
\item \emph{directly $\diff$presented}, if it is the essential image of the appropriate functor $\ssig{\,}^*$ on correspondences (as in \ref{dir-pres-adj}).  
\end{enumerate}

When the relevant categories admit forgetful functors to $\Set$ or have a meaningful notion of `finite cardinals', we will be able to speak of \emph{finitely $\diff$generated, finitely $\diff$presented} and \emph{finitely directly $\diff$presented} objects.

\subsection{Finitely presented difference algebras}

In this subsection, we show that the categorical notions of finite $\diff$generation and $\diff$presentation from \ref{free-dirpres-diffob} agree with the traditional notions for difference algebras.

Recall (\cite[Section 2.1]{levin}) that a \emph{difference ideal} of a difference ring $R=(\forg{R},\sigma_R)$ is an ideal $I$ of $\forg{R}$ such that $\sigma_R(I)\subseteq I$. In this case the quotient $\forg{R}/I$ is naturally a difference ring. The difference ideal $I$ is \emph{finitely $\diff$generated} if there exists a finite subset $B$ of $I$ such that $I$ is generated by $B,\sigma_R(B),\ldots$ as an ideal.

%As in (\cite[Section 2.1]{levin}) 
A $k$-algebra $R$ is called \emph{finitely $\diff$generated} if there exists a finite subset $B$ of $R$ such that $\forg{R}$ is generated by $B,\sigma_R(B),\ldots$ as a $\forg{k}$-algebra.

A difference algebra over a difference ring $k$ is called \emph{finitely $\diff$presented} if it is a quotient of a difference polynomial ring in finitely many variables over $k$ by a finitely $\diff$generated difference ideal (\cite[Section 7.1]{Wibmer:FinitenessPropertiesOfAffineDifferenceAlgebraicGroups}).

\begin{definition}
Let $k$ be a difference ring. Following \cite[Section 2.2]{levin} the \emph{difference polynomial ring in $n$ variables} over $k$
$$
P_{n,k}=k\{x_1,\ldots,x_n\}={\forg{k}[x_1,\ldots,x_n,\sigma(x_1),\ldots,\sigma(x_n),\ldots]}
$$
{is the polynomial ring in the variables $x_1,\ldots,x_n,\sigma(x_1),\ldots,\sigma(x_n),\ldots$ over $\forg{k}$ with action of $\sigma$ extended from $k$ as suggested by the names of the variables. In other words, $P_{n,k}=\ssig{ \forg{k}[x_1,\ldots,x_n]}^k$ is}
obtained by applying the `free $k$-algebra' functor, i.e., the left adjoint of the forgetful functor $k\Alg\to\forg{k}\Alg\to\Set$ to an $n$-element set.
\end{definition}

\begin{lemma}\label{fin-si-gen}
Let $k$ be a difference ring, and let $A\in k\Alg$ be a $k$-algebra. The following statements are equivalent:
\begin{enumerate}
\item {$A$ is finitely $\diff$generated;}
\item $A$ is a quotient of a difference polynomial ring in finitely many variables over $k$;
\item $A$ is a quotient of $\ssig{A_0}^k$, for some finitely generated $\forg{k}$-algebra $A_0$.
\end{enumerate}
\end{lemma}
%\begin{proof}
%Assuming (1), using the adjunction $\ssig{\,}^k\dashv \forg{\,}^k$, there is a unique $k$-algebra homomorphism $P_{n,k}\to A$ taking $x_i$ to $a_i$, and it is clearly surjective, so we obtain (2). The converse implication is clear.
%
%If we have an epimorphism $\ssig{A_0}\to A$, where $A_0$ is finitely generated over $\forg{k}$, then we have an epimorphism $\forg{k}[x_1,\ldots,x_n]\to A_0$ in $\forg{k}\Alg$, so applying $\ssig{\,}^k$, we have 
%$$
%P_{n,k}\to \ssig{A_0}\to A,
%$$
%where both arrows are epimorphisms, so the composite is an epimorphism $P_{n,k}\to A$. This establishes the equivalence of (2) and (3).
% 
%\end{proof}

\begin{lemma}\label{fin-pres-dir-pres}
Let $k$ be a difference ring, and let $A\in k\Alg$ be a $k$-algebra. The following statements are equivalent:
\begin{enumerate}
\item {$A$ is finitely $\diff$presented;} 
\item $A$ fits into a coequaliser diagram
\begin{center}
 \begin{tikzpicture} 
\matrix(m)[matrix of math nodes, row sep=0em, column sep=1.7em, text height=1.5ex, text depth=0.25ex]
 {
 |(1)|{P_{m,k}}		& |(2)|{P_{n,k}}  & |(3)|{A}	\\
 }; 
\path[->,font=\scriptsize,>=to, thin,yshift=12pt]
(2) edge node[above]{} (3)
([yshift=2pt]1.east) edge node[above]{} ([yshift=2pt]2.west) 
([yshift=-2pt]1.east)edge node[below]{}   ([yshift=-2pt]2.west) 
;
\end{tikzpicture}
\end{center}
of some difference polynomial rings in finitely many variables over $k$;
\item $A$ fits into a coequaliser diagram
\begin{center}
 \begin{tikzpicture} 
\matrix(m)[matrix of math nodes, row sep=0em, column sep=1.7em, text height=1.5ex, text depth=0.25ex]
 {
 |(1)|{\ssig{A_0}^k}		& |(2)|{\ssig{C_0}^k}  & |(3)|{A}	\\
 }; 
\path[->,font=\scriptsize,>=to, thin,yshift=12pt]
(2) edge node[above]{} (3)
([yshift=2pt]1.east) edge node[above]{} ([yshift=2pt]2.west) 
([yshift=-2pt]1.east)edge node[below]{}   ([yshift=-2pt]2.west) 
;
\end{tikzpicture}
\end{center}
where $A_0$ and $C_0$ are finitely presented $\forg{k}$-algebras;
\item there exists a cocorrespondence $(A_0\to C_0 \leftarrow A_0)$ of finitely presented $\forg{k}$-algebras so that
$$
A\simeq \ssig{A_0;C_0}^k.
$$
\end{enumerate}
\end{lemma}

\begin{lemma}[{{\cite[Section 4.3]{levin}}}]\label{ld-lemma}
Let $L/k$ be an {extension of difference fields such that $\forg{L}/\forg{k}$ is algebraic and $L$ is finitely $\diff$generated as a $k$-algebra.}  Let $a=a_1,\ldots,a_n$ be a tuple of $\diff$generators of the $k$-algebra $L$, i.e.,
$L=\forg{k}[a,\sigma(a),\sigma^2(a),\ldots]$.
Let 
$$
L_i=\forg{k}[a,\sigma(a),\ldots,\sigma^i(a)].
$$
The sequence of numbers $$d_i=[L_i:L_{i-1}]$$
is non-increasing and hence stabilizes. {The eventual value does not depend on the choice of $\diff$generators $a$.}
\end{lemma}
\begin{definition}[{{\cite[Section 4.3]{levin}}}]\label{ld-defn}
With notation of \ref{ld-lemma}, we define the \emph{limit degree} of the extension $L/k$ as the integer
$$
\ld(L/k)=\lim_i d_i.
$$
\end{definition} 

%\begin{lemma}\label{diff-fld-fp-fg}
%A finitely $\diff$generated difference fields extension is finitely $\diff$presented.
%\end{lemma}
%\begin{proof}
%With notation of \ref{ld-lemma}, let $i$ be minimal such that $d_i=\ld(L/k)$. Consider the map $P_{n,k}=k\{x_1,\ldots,x_n\}\to L$ sending the tuple $x=x_1,\ldots,x_n$ to $a$. We claim that its kernel $\p$ is finitely generated as a difference ideal by the set $\p[i]$ of its elements of order up to $i$. It suffices to show that $\p[i+1]=(\p[i],\sigma(\p[i]))$. Since the right hand side is contained in the left hand side, it is enough to show that $L_{i+1}\simeq k[x,\sigma x,\ldots,\sigma^{i+1}x]/\p[i+1]$ and $k[x,\sigma(x),\ldots,\sigma^{i+1}x]/(\p[i],\sigma(\p[i]))$ have the same degree over $k$, or, equivalently, the same degree over $L_i$. The task is reduced to showing that 
%$$
%\left[k[x,\sigma(x),\ldots,\sigma^{i+1}x]/(\p[i],\sigma(\p[i])): L_i\right]\leq \ld(L/k),
%$$
%and this follows from the fact that $\sigma: k[a,\ldots,\sigma^i(A)]\to \sigma(k)[\sigma(a),\ldots,\sigma^{i+1}(A)]$ is an isomorphism. 
%\end{proof}
%

\section{Galois theory}\label{s:galth}

In this section, we review Janelidze's pure Galois theory, as set out in \cite{janelidze} and \cite{borceux-janelidze}, as well as  Magid's separable Galois theory of commutative rings, following {\cite{magid}}.

In \ref{si-Janelidze}, we develop an abstract version of Janelidze's formalism for categories of difference objects, and apply it in \ref{ss:diff-gal} to develop our main objective, the Galois theory of difference ring extensions, as well as the theory of the difference fundamental groupoid in \ref{ss:diff-fund}. In the special case of difference fields, we have a more complete understanding of the theory, as explained in \ref{ss:diff-fields}

\subsection{Janelidze's categorical Galois theory}\label{janelidze}

Consider an adjoint pair of functors
\begin{center}
 \begin{tikzpicture} 
 [cross line/.style={preaction={draw=white, -,
line width=3pt}}]
\matrix(m)[matrix of math nodes, minimum size=1.7em,
inner sep=0pt, 
row sep=3.3em, column sep=1em, text height=1.5ex, text depth=0.25ex]
 { 
  |(dc)|{\cA}	\\
 |(c)|{\cP} 	      \\ };
\path[->,font=\scriptsize,>=to, thin]
%
%(dc) edge[draw=none] node (mid) {} (c)
%(dc) edge node (fo) [pos=0.25,right=-3.5pt]{$\forg{\kern1pt}$}  (c)
(dc) edge [bend right=30] node (ss) [left]{$S$} (c)
(c) edge [bend right=30] node (ps) [right]{$C$} (dc)
%(m)id edge[draw=none] node{$\dashv$} (ss)
(ss) edge[draw=none] node{$\dashv$} (ps)
;
\end{tikzpicture}
\end{center}
%i.e., $S:\cA\to \cP$, $C:\cP\to\cA$, with $S\dashv C$, and 
with unit $\eta:\id\to CS$ and counit $\epsilon:SC\to \id$.
If $\cA$ admits pullbacks, using \ref{adj-rel}, for any $X\in\cA$ we obtain an adjunction
\begin{center}
 \begin{tikzpicture} 
 [cross line/.style={preaction={draw=white, -,
line width=3pt}}]
\matrix(m)[matrix of math nodes, minimum size=1.7em,
inner sep=0pt, 
row sep=3.3em, column sep=1em, text height=1.5ex, text depth=0.25ex]
 { 
  |(dc)|{\cA_{\ov X}}	\\
 |(c)|{\cP_{\ov S(X)}} 	      \\ };
\path[->,font=\scriptsize,>=to, thin]
%
%(dc) edge[draw=none] node (mid) {} (c)
%(dc) edge node (fo) [pos=0.25,right=-3.5pt]{$\forg{\kern1pt}$}  (c)
(dc) edge [bend right=30] node (ss) [left]{$S_X$} (c)
(c) edge [bend right=30] node (ps) [right]{$C_X$} (dc)
%(m)id edge[draw=none] node{$\dashv$} (ss)
(ss) edge[draw=none] node{$\dashv$} (ps)
;
\end{tikzpicture}
\end{center}
where  
$$
S_X(A\xrightarrow{a}X)=S(A)\xrightarrow{S(a)}S(X).
$$
%and  $C_X(E\xrightarrow{e}S(X))$ is obtained by forming the pullback
%\begin{center}
% \begin{tikzpicture} 
%\matrix(m)[matrix of math nodes, row sep=2em, column sep=2em, text height=1.9ex, text depth=0.25ex]
% {
% |(1)|{C_X(e)}		& |(2)|{C(E)} 	\\
% |(l1)|{X}		& |(l2)|{CS(X)} 	\\
% }; 
%\path[->,font=\scriptsize,>=to, thin]
%(1) edge  (2) edge   (l1)
%(2) edge node[right]{$C(e)$} (l2) 
%(l1) edge node[above]{$\eta_X$}  (l2);
%\end{tikzpicture}
%\end{center}
%in $\cA$.

A morphism $X\xrightarrow{f}Y$ {in $\cA$} gives rise to the pullback/base change functor
$$
f^*:\cA_{\ov Y}\to\cA_{\ov X},
$$
which admits a left adjoint
$$
f_!:\cA_{\ov X}\to\cA_{\ov Y},  \ \ \  a\mapsto f\circ a.
$$
{Following \cite[Def. 5.1.7]{borceux-janelidze}} an object $A\xrightarrow{a}Y\in\cA_{\ov Y}$ is \emph{split} by $X\xrightarrow{f}Y\in\cA$ when the unit $\eta^X:\id\to C_XS_X$ of adjunction $S_X\dashv C_X$ gives an isomorphism
$$
\eta^X_{f^*a}:f^*a\to C_XS_X(f^*a),
$$
or, equivalently {(\cite[Cor. 5.1.13]{borceux-janelidze})}, if there exists an object $E\xrightarrow{e}S(X)$ such that 
$$
f^*a\simeq C_X(e).
$$
We write
$$
\Split_Y(f)
$$
for the full subcategory of $\cA_{\ov Y}$ of objects split by $f$.

The morphism $f$ is of \emph{relative Galois descent} if
\begin{enumerate}
\item $f^*$ is monadic;
\item the counit $\epsilon^X:S_XC_X\to \id$ of adjunction $S_X\dashv C_X$ is an isomorphism;
\item for every $E\xrightarrow{e}S(X)$ in $\cP_{\ov S(X)}$, the object $f_!\, C_X(e)\in\cA_{\ov Y}$ is split by $f$.
\end{enumerate} 
 If $X\xrightarrow{f}Y$ is of relative Galois descent, we define 
$$
\Gal[f]
$$
as the internal groupoid in $\cP$ given by the data

\begin{center}
 \begin{tikzpicture} 
\matrix(m)[matrix of math nodes, row sep=0em, column sep=3em, text height=1.5ex, text depth=0.25ex]
 {
|(0)|{S(X\times_YX)\times_{S(X)}S(X\times_YX)}  &[3em]  |(1)|{S(X\times_YX)}		&[1em] |(2)|{S(X)} \\
 }; 
\path[->,font=\scriptsize,>=to, thin]
(0) edge node[above]{$(S(\pi_1),S(\pi_4))$} (1)
([yshift=1em]1.east) edge node[above=-2pt]{$S(\pi_1)$} ([yshift=1em]2.west) 
(2)  edge node[above=-2pt]{$S(\Delta)$} (1) 
([yshift=-1em]1.east) edge node[above=-2pt]{$S(\pi_2)$} ([yshift=-1em]2.west) 
%([xshift=-3pt]1.south) edge [loop below] ([xshift=3pt]1.south)
 (1) edge [loop below] node {$S(\tau)$} (1)
;
\end{tikzpicture}
\end{center}
where $\tau$ is the morphism interchanging the copies of $X$, and $\Delta$ is the diagonal.

Janelidze's \emph{Galois theorem} {(\cite[Thm. 5.1.24]{borceux-janelidze})} gives  an equivalence of categories
$$
\Split_Y(f)\simeq [\Gal[f],\cP].
$$
The proof consists in verifying that the monad $\mathbb{T}$ associated to the adjunction 
\begin{center}
 \begin{tikzpicture} 
 [cross line/.style={preaction={draw=white, -,
line width=3pt}}]
\matrix(m)[matrix of math nodes, minimum size=1.7em,
inner sep=0pt, 
row sep=3.3em, column sep=1em, text height=1.5ex, text depth=0.25ex]
 { 
  |(dc)|{\Split_Y(f)}	\\
 |(c)|{\cP_{\ov S(X)}} 	      \\ };
\path[->,font=\scriptsize,>=to, thin]
%
%(dc) edge[draw=none] node (mid) {} (c)
%(dc) edge node (fo) [pos=0.25,right=-3.5pt]{$\forg{\kern1pt}$}  (c)
(dc) edge [bend left=30] node (ss) [right]{$U=S_X\,f^*$} (c)
(c) edge [bend left=30] node (ps) [left]{$F=f_!\,C_X$} (dc)
%(m)id edge[draw=none] node{$\dashv$} (ss)
(ss) edge[draw=none] node{$\dashv$} (ps)
;
\end{tikzpicture}
\end{center}
of the monadic functor $U$ {(\cite[Cor. 5.1.21]{borceux-janelidze})}, with functorial part $T=UF$, is isomorphic to the monad on $\cP_{\ov S(X)}$ whose category of algebras is the category $[\Gal[f],\cP]$ of actions of the internal groupoid $\Gal[f]$ in $\cP$.  

Hence, modulo the identification of the category of $\mathbb{T}$-algebras $(\cP_{\ov S(X)})^{\mathbb{T}}$ with  $[\Gal[f],\cP]$ 
the functors realising the sought-after equivalence are the comparison functor of the monad $\mathbb{T}$,
$$
\Split_Y(f)\xrightarrow{\sim} (\cP_{\ov S(X)})^{\mathbb{T}},\ \ \   A\mapsto (U(A),U(\varepsilon_A)),
$$
and its left adjoint
$$
(X,\nu)\mapsto
\begin{tikzcd}[cramped, column sep=normal, ampersand replacement=\&]
{\text{\rm Coeq}\left(FUF(X) \right.}\ar[yshift=2pt]{r}{F(\nu)} \ar[yshift=-2pt]{r}[swap]{\varepsilon_{FX}} \&{\left.F(X)\right)}
\end{tikzcd},
$$
where $(X,\nu)$ is a $\mathbb{T}$-algebra and $\varepsilon:FU\to \id$ is the counit of the adjunction $F\dashv U$, and the coequaliser exists by the proof of Beck's monadicity criterion as in \cite[3.14]{barr-wells}.

By implicitly remembering the $\Gal[f]$-action, allowing a slight imprecision in terminology, we may say that the above equivalence is realised by the functor $U:\Split_Y(f)\to \cP_{\ov S(X)}$.
% realising this equivalence is
%$$
%(Z\xrightarrow{z}Y)\ \ \mapsto\ \ S_X(f^*(z)).  
%$$

\subsection{Janelidze's theory for categories of difference objects}\label{si-Janelidze}

\begin{notation}\label{gen-diff-context}
Let $\cA_0$, $\cP_0$, $S_0$ and $C_0$ be categories and functors forming a context for categorical Galois theory as in \ref{janelidze}. 
The context for the associated difference Galois theory consists of 
\begin{enumerate} 

\item the  categories of difference objects $$\cA=\diff\cA_0, \ \ \ \cP=\diff\cP_0,$$
%\item Let $$\cP=\diff\cP_0$$ denote the category of difference profinite spaces.
\item the functor  
$$S:\cA\to \cP, \ \ S(X,\sigma_X)=(S_0(X),S_0(\sigma_X)),$$ 

\item the functor
$$
C:\cP\to \cA, \ \ C(E,\sigma_E)=(C_0(E),C_0(\sigma_E)).
$$
\end{enumerate}
\end{notation}
%\mw{The following lemma shows that we have indeed an adjunction $S\dashv C$.}

\begin{lemma}\label{diff-adjunction}
Let $F_0: \cD\to \cC$ and $G_0:\cC\to \cD$ be an adjoint pair of functors,
$$
F_0\dashv G_0,
$$
and consider  the induced functors on difference categories
$$
F:\diff\cD\to \diff\cC, \ \ F(X)=(F_0(\forg{X}),F_0(\forg{\sigma_X})),
$$
and
$$
G:\diff\cC\to \diff\cD, \ \ G(Y)=(G_0(\forg{Y}),G_0(\forg{\sigma_Y})).
$$ 
Then
$$
F\dashv G.
$$
Moreover, if $G_0$ is monadic, then $G$ is also monadic.
\end{lemma}

\begin{proof}
If $\varepsilon_0$ and $\eta_0$ are the counit and unit of the adjunction $F_0\dashv G_0$, then
the counit $\varepsilon: FG\to \id_{\diff\cC}$ of the required adjunction is given by $\varepsilon_X=\varepsilon_{0,\forg{X}}$ for $X\in\diff\cC$, and
the unit $\eta:\id_{\diff\cD}\to GF$ is given by $\eta_Y=\eta_{0,\forg{Y}}$ for $Y\in\diff\cC$. These are difference morphisms since by naturality of $\varepsilon_0$ and $\eta_0$, we have diagrams
\begin{center}
 \begin{tikzpicture} 
\matrix(m)[matrix of math nodes, row sep=2em, column sep=2em, text height=1.9ex, text depth=0.25ex]
 {
 |(1)|{F_0G_0(\forg{X})}		& |(2)|{\forg{X}} 	\\
 |(l1)|{F_0G_0(\forg{X})}		& |(l2)|{\forg{X}} 	\\
 }; 
\path[->,font=\scriptsize,>=to, thin]
(1) edge  node[above]{$\varepsilon_{0,\forg{X}}$} (2) edge  node[left]{$F_0G_0(\forg{\sigma_X})$}   (l1)
(2) edge node[right]{$\forg{\sigma_X}$} (l2) 
(l1) edge node[above]{$\varepsilon_{0,\forg{X}}$}  (l2);
\end{tikzpicture}
 \begin{tikzpicture} 
\matrix(m)[matrix of math nodes, row sep=2em, column sep=2em, text height=1.9ex, text depth=0.25ex]
 {
 |(1)|	{\forg{Y}} 	& |(2)|   {G_0F_0(\forg{Y})}	\\
 |(l1)|	{\forg{Y}}	& |(l2)|   	{G_0F_0(\forg{Y})}\\
 }; 
\path[->,font=\scriptsize,>=to, thin]
(1) edge  node[above]{$\eta_{0,\forg{Y}}$} (2) edge    node[left]{$\forg{\sigma_Y}$}  (l1)
(2) edge node[right]{$G_0F_0(\forg{\sigma_Y})$} (l2) 
(l1) edge node[above]{$\eta_{0,\forg{Y}}$}  (l2);
\end{tikzpicture}
\end{center}
hence the counit-unit equations for $\varepsilon,\eta, F, G$ follow directly from those for $\varepsilon_0,\eta_0,F_0,G_0$.

Let $\mathbb{T}_0=(T_0,\mu_0,\eta_0)$ be the monad induced by the adjunction $F_0\dashv G_0$, where $T_0=G_0F_0:\cD\to\cD$, $\mu_0=G_0\varepsilon_0 F_0: T_0T_0\to T_0$ and $\eta_0: \id\to T_0$ is the unit.
The functor $G_0$ being monadic amounts to saying that the comparison functor
$$
\cC\to \cD^{\mathbb{T}_0}, \ \ X\mapsto (G_0(X), G_0(\varepsilon_X))
$$ 
to the category $ \cD^{\mathbb{T}_0}$ of $\mathbb{T}_0$-algebras gives an equivalence of categories. 

Let $\mathbb{T}=(T,\mu,\eta)$ be the monad induced by the adjunction $F\dashv G$, where 
$T=G F:\diff\cD\to\diff\cD$, $\mu=G \varepsilon  F: T T\to T$ and $\eta: \id\to T$ is the above unit. We obtain that
$$
(\diff\cD)^\mathbb{T}\simeq \diff(\cD^{\mathbb{T}_0})\simeq \diff\cC,
$$
so $G$ is monadic.
\end{proof}

\begin{lemma}
The functors from \ref{gen-diff-context} form an adjunction
$$
S\dashv C.
$$
If $\cA_0$ admits pullbacks, for each $X\in\cA$, we have an adjunction
$$
S_X\dashv C_X
$$
between slice categories $\cA_{\ov X}$ and $\cP_{\ov S(X)}$.
\end{lemma}
\begin{proof}
The adjunction is given directly \ref{diff-adjunction}. If $\cA_0$ admits pullbacks, so does $\cA$, hence \ref{adj-rel} gives the relative adjunction. 
\end{proof}

\begin{lemma}\label{gen-diff-split-normal}
Let $f:X\to Y$ be a morphism in $\cA$.
\begin{enumerate}
\item An object $A\xrightarrow{a} Y\in\cA_{\ov Y}$ is split by $f$ if and only if $\forg{a}\in\cA_{0\ov\forg{Y}}$ is split by $\forg{f}$. 
\item If 
%$\cA_0$ admits coequalisers, and 
$\forg{f}$ is a morphism of relative Galois descent in $\cA_0$, then $f$ is a morphism of relative Galois descent. 
\end{enumerate}
\end{lemma}

\begin{proof}
\noindent{(1)} Pullbacks in $\cA$ are calculated in $\cA_0$ (\ref{forg-creates-limits}), so 
$\forg{f^*A}=\forg{X\times_Y A}\simeq \forg{X}\times_{\forg{Y}}\forg{A}\simeq \forg{f}^*\forg{A}$.

By definition, $A$ is split by $f$ if and only if the unit $\eta^X:f^*A\to C_XS_X(f^*A)$ is an isomorphism, which is equivalent to saying that $\forg{\eta^X}$ is an isomorphism. Given that  $\forg{}\circ S_X\simeq S_{0,\forg{X}}\circ\forg{}$ and $\forg{}\circ C_X\simeq C_{0,\forg{X}}\circ\forg{}$, this is equivalent to
$$
\eta^{\forg{X}}:\forg{f}^*\forg{A}\to C_{0,\forg{X}}S_{0,\forg{X}}(\forg{f}^*\forg{A})
$$
being an isomorphism, i.e., to  $\forg{A}$ being split by $\forg{f}$.

\noindent{(2)} Suppose that $\forg{f}$ is a morphism of relative Galois descent.
We prove that $f^*$ is monadic using Beck's criterion \cite[3.14]{barr-wells}.  Since  $f^*$ has a left adjoint $f_!$, it suffices to verify that $f^*$ reflects isomorphisms and that $\cA_{\ov Y}$ has coequalisers of reflexive $f^*$-contractible coequaliser pairs and $f^*$ preserves them.  This follows from the assumption that $\forg{f}^*$ is monadic so enjoys the corresponding properties, and the fact that $\forg{\,}$ preserves and reflects isomorphisms and creates coequalisers (\ref{forg-creates-limits}). 

%reflects isomorphisms and preserves equalisers. Since $\forg{L}$ is a faithfully flat $\forg{k}$-module, the functor $\forg{L}\otimes_{\forg{k}}\mathord{-}:\forg{k}\Mod\to \forg{k}\Mod$ preserves and reflects exact sequences, and therefore $L\otimes_k\mathord{-}:k\Mod\to k\Mod$ does too. Since $L$-algebra isomorphisms are in particular $k$-module isomorphisms, and equalisers in $k\Alg$ are computed as equalisers in $k\Mod$, both of these properties can be captured by exact sequences in $k\Mod$, so  $L\otimes_k\mathord{-}$ has the required properties.
%
%By assumption, $\forg{f}$ is of relative Galois descent, so we can argue as follows.

The counit $\varepsilon_X:S_XC_X\to\id$ is an isomorphism, since the underlying counit
$\forg{\varepsilon_X}=\varepsilon_{\forg{X}}:S_{0,\forg{X}}C_{0,\forg{X}}\to \id$ is an isomorphism by assumption on $\forg{f}$. 

Moreover, for every $E\in\cP\simeq \cP_{\ov S(X)}$, the object $f_! C_X(E) \in  \cA_{\ov Y}$ is split by $f$, since by (1), this is equivalent to $\forg{f_! C_X(E)}=\forg{f}_! C_{0,\forg{X}}(\forg{E})$ being split by $\forg{f}$, and this holds by assumption on $\forg{f}$.
\end{proof}

The following theorem is an instance of Janelidze's \emph{Galois theorem} (\cite[Thm.~5.1.24]{borceux-janelidze}). However, in conjunction with the Comparison theorem below, one can think of it as a difference version of Janelidze's theorem.

\begin{theorem}[Galois theorem for categories of difference objects]\label{gen-diff-Galois-thm}
With Notation~\ref{gen-diff-context}, and with $\cA_0$ admitting pullbacks, let $f:X\to Y$ be a morphism in $\cA$ such that $\forg{f}$ is of relative Galois descent in $\cA_0$.
The category 
$$
\Split_Y(f)
$$
is is the full subcategory of $\cA_{\ov Y}$ consisting of objects $a$ such that $\forg{a}$ is split by $\forg{f}$.
 The difference categorical Galois groupoid is the groupoid object in $\cP$
%$$
%\Gal[f]=S(L\otimes_k L).   %\simeq \Aut\lbr L\rbr_k\simeq (\Gal(\forg{L}/\forg{k}),()^\sigma).
%$$
$$
\begin{tikzcd}[cramped, column sep=normal, ampersand replacement=\&]
{\Gal[f]=\left(S(X\otimes_YX) \right.}\ar[yshift=2pt]{r}{} \ar[yshift=-2pt]{r}[swap]{} \&{\left.S(X)\right).}
\end{tikzcd}
$$
We have an equivalence of categories
$$
\Split_k(f)\simeq [\Gal[f],\cP],
$$
realised by the functor
$$
A\mapsto S(L\otimes_k A). %\simeq \lbr A,L\rbr_k\simeq (\forg{k}\Alg(\forg{A},\forg{L}),()^\sigma).
$$
\end{theorem}
 \begin{proof}
 By \ref{gen-diff-split-normal}, such an $f$ is of relative Galois descent, and the category of objects split by $f$ is as described, so the statement follows from Janelidze's categorical Galois theorem \ref{janelidze} applied in the difference context \ref{gen-diff-context}.
 \end{proof}

\begin{theorem}[Comparison theorem]\label{gen-gal-free-diff}
With notation of \ref{gen-diff-Galois-thm}, let $G=\Gal[f]$ be the difference Galois groupoid, and denote by
 $$
 \Phi:\Split_Y(f)\to [G,\cP], \ \ \  \Psi:[G,\cP]\to \Split_Y(f) 
$$
the adjoint pair realising the stated equivalence of categories. By the assumption on $\forg{f}$ and Janelidze's theory for $(\cA_0,\cP_0,S_0,C_0)$, we have that 
$$
\forg{G}=\Gal[\forg{f}],
$$
and the associated functors 
 $$
 \Phi_0:\Split_{\forg{Y}}(\forg{f})\to [\forg{G},\cP_0], \ \ \  \Psi_0:[\forg{G},\cP_0]\to \Split_{\forg{Y}}(\forg{f}) 
$$
also define an equivalence of categories.  The solid part of the diagram
$$
\begin{tikzcd}[column sep=normal, ampersand replacement=\&] %,every arrow/.append style={shift left}]
{\Split_Y(f)}\arrow[r,yshift=2pt,"\Phi"] \arrow[d,xshift=-2pt,"\forg{}_{Y}",swap] \&{[G,\cP]} \arrow[l,yshift=-2pt,"\Psi"]\arrow[d,xshift=-2pt,"\forg{}_G",swap]\\
{\Split_{\forg{Y}}(\forg{f})}\arrow[r,yshift=2pt,"\Phi_0"] \arrow[u,xshift=2pt,right,dashed,"\psig{}_{Y}",swap]  \&{[\forg{G},\cP_0]} \arrow[l,yshift=-2pt,"\Psi_0"] \arrow[u,xshift=2pt,right,dashed,"\psig{}_G",swap]
\end{tikzcd}
$$
is commutative. 

Moreover, if $\cA_0$ has countable pullbacks, the dashed part of the diagram also commutes in the sense that
$$
\Phi(\psig{A_0}_{Y})\simeq \psig{\Phi_0 A_0}_G,
$$
as well as 
$$
\Psi(\psig{F_0}_G)\simeq\psig{\Psi_0 F_0}_{Y},
$$
for  a $\forg{G}$-action $F_0\in\cP_0$.
\end{theorem}

\begin{proof}
Using the difference and classical Galois correspondences and the adjunction $\Psi\dashv \Phi$, we have canonical isomorphisms
\begin{multline*}
[G,\cP](F,\Phi(\psig{A_0}_{Y}))
\simeq \cA_{Y}(\Psi(F),\psig{A_0}_{k})\simeq {\cA_0}_{\ov\forg{Y}}(\forg{\Psi(F)},A_0) \\
\simeq {\cA_0}_{\ov\forg{Y}}(\Psi_0\forg{F},A_0)\simeq [\forg{G},\cP_0](\forg{F},\Phi_0A_0)\simeq
[G,\cP](F,\psig{\Phi_0A_0}_G)
\end{multline*}
for all $F\in [G,\cP]$, hence $\Phi$ has the required property. To verify the property of $\Psi$, we perform an analogous calculation using the reverse adjunction $\Phi\dashv\Psi$.  
\end{proof}

\begin{corollary}\label{gal-dif-corr}
With assumptions of \ref{gen-gal-free-diff} and notation from \ref{dir-pres-diag-adj}, the diagram
$$
\begin{tikzcd}[column sep=normal, ampersand replacement=\&] %,every arrow/.append style={shift left}]
{\Split_Y(f)}\arrow[r,yshift=2pt,"\Phi"] \arrow[d,xshift=-2pt,"\forg{}_{Y}^\Corr",swap] \&{[G,\cP]} \arrow[l,yshift=-2pt,"\Psi"]\arrow[d,xshift=-2pt,"\forg{}_G^\Corr",swap]\\
{\Corr_{Y}(\Split_{\forg{Y}}(\forg{f}))}\arrow[r,yshift=2pt,"\Phi_0"] \arrow[u,xshift=2pt,right,"\psig{}_{Y}",swap]  \&{\Corr_G([\forg{G},\cP_0])} \arrow[l,yshift=-2pt,"\Psi_0"] \arrow[u,xshift=2pt,right,"\psig{}_G",swap]
\end{tikzcd}
$$
commutes.
\end{corollary}
\begin{proof}
By \ref{gen-gal-free-diff}, the functors $\Phi$ and $\Psi$ establishing the equivalence of categories commute with free $\diff$generation as well as limits and colimits. Hence, using the construction of the directly $\diff$presented object associated to a $Y$-correspondence $(A_0,C_0)=(A_0\leftarrow C_0\rightarrow A_1)$ in $\cA_{\ov Y}$ given in the proof of \ref{dir-gen-adj}, we have
\begin{multline*}
\Phi(\psig{A_0;C_0}_{Y})=\Phi\left(\Eq(\psig{C_0}_{Y}\doublerightarrow{}{}\psig{A_1}_{Y})\right)\simeq\Eq\left(\Phi(\psig{C_0}_{Y})\doublerightarrow{}{}\Phi(\psig{A_1}_{Y})\right)\\
\simeq\Eq\left(\psig{\Phi_0C_0}_{G}\doublerightarrow{}{}\psig{\Phi_0A_1}_{G}\right)\simeq \psig{\Phi_0A_0,\Phi_0C_0}_G,
\end{multline*}
so $\Phi$ transforms a directly presented object of $\Split_Y(f)$ into a directly presented $G$-action. The converse is analogous using $\Psi$.
\end{proof}

\subsection{Magid's Galois theory for commutative rings}\label{ss:magid-rings}

{Our next big goal is to apply Theorem \ref{gen-diff-Galois-thm} with $\cA$ the category of difference rings, i.e., with $\cA_0$ the category of rings. For this purpose we recall the categorical Galois theory of rings (\cite{magid}, \cite[Chapter 4]{borceux-janelidze}) }

\begin{definition}\label{class-gal-setup}
\begin{enumerate}
\item Let $$\cA_0=\Rng^\op$$ be the opposite category of the category of commutative unital rings, which we {may} think of as the category of affine schemes.
\item Let $$\cP_0$$ denote the category of profinite spaces.
\item The \emph{Pierce spectrum} functor {(\cite[Def. 4.2.4]{borceux-janelidze})}  $$S_0:\cA_0\to \cP_0, \ \ S_0(R)=\spec(E(R))$$ sends a commutative ring $R$ to the Boolean spectrum of ultrafilters of the Boolean algebra $E(R)$ of idempotents of $R$. By Stone duality {(\cite[Section~4.1]{borceux-janelidze})}, it is a profinite space, and it can also be regarded as the space $$\pi_0(\spec(R))$$ of connected components of the affine scheme $\spec({R})$.

The complement $P_x$ of an ultrafilter $x\in S_0(R)$ is a prime ideal in the Boolean algebra $E(X)$, so the ideal $P_x R$ of the ring $R$ is a maximal \emph{regular} ideal (maximal in the poset of ideals generated by idempotent elements), so the ring
$$
R_x=R/P_x R
$$
is connected/indecomposable, i.e., has no nontrivial idempotents. 

\item The \emph{ring of continuous functions} functor
$$
C_0:\cP_0\to \cA_0, \ \ C_0(X)=\text{\rm Cont}(X,\Z)
$$
sends a profinite space $X$ to the ring of continuous functions from $X$ to the discrete ring $\Z$.
\end{enumerate}
\end{definition}

\begin{fact}[{\cite[4.7.15]{borceux-janelidze}}]\label{gal-general-rings}
We have an adjunction 
$$
S_0\dashv C_0
$$
which gives rise to a categorical Galois theory for rings as follows.
 If $R\to L$ is a ring homomorphism such that the associated morphism $f:L\to R$ in $\cA_0$ is  of relative Galois descent, we have an equivalence of categories 
$$
\Split_R(f)\simeq [\Gal[f],\cP_0]
$$
where the subcategory $\Split_R(f)$ of $\cA_0$ is the opposite category of $R$-algebras split by $L$, and
$$
\Gal[f]
$$
is the internal groupoid in the category of profinite spaces with the object of objects $S_0(L)$ and the object of morphisms $S_0(L\otimes_RL)$. The equivalence is realised by the functor 
$$
\Split_R(f) \to {\cP_0}_{\ov S_0(L)}, \ \ \ A\mapsto {S_0}_L(L\otimes_R A).
$$
\end{fact}

Magid \cite{magid} identifies the class of morphisms of relative Galois descent and corresponding split extensions as follows.
\begin{definition}\label{def-clss}
Let $R$ be a ring, and let an $R$-algebra $A$ be an \emph{extension} of $R$, i.e., the structure homomorphism $R\to A$ is injective. We say that the extension $A/R$ is 
\begin{enumerate}
\item \emph{separable}, if $A$ is a projective $(A\otimes_RA)$-module;
\item \emph{strongly separable}, if it is separable and $A$ is a finite projective $R$-module;
%\item \emph{separable componentially strong}, if it is separable and, for each $x\in S_0(R)$, $A_x=R_x\otimes_RA$ is a strongly separable extension of $R_x$. 
\item \emph{locally strongly separable}, if it is a colimit of strongly separable extensions of $R$;
\item \emph{componentially locally strongly separable} (abbreviated \emph{clss}), if for each $x\in S_0(R)$, $A_x=R_x\otimes_RA$ is a locally strongly separable extension of $R_x$.  
\end{enumerate}
\end{definition}

\begin{remark}\label{str-sep-fin-etale}
In algebraic geometry terms, $A$ is a strongly separable extension of $R$ if and only if $\spec(A)\to\spec(R)$ is a finite \'etale morphism (\cite[page 128]{magid}). Hence, $A/R$ is locally strongly separable if and only if $\spec(A)\to\spec(R)$ is pro-(finite \'etale). 
\end{remark}

\begin{fact}[{\cite[Theorem~6.1]{magid}}, {\cite[Section~4]{janelidze}}]\label{magid-th}
If $L$ is an auto-split clss extension of $R$, then the associated morphism $f:L\to R$ in $\cA_0$ is of relative Galois descent,
% \mw{?? this is not stated in \cite[Theorem~6.1]{magid}}, 
and the category $\Split_R(f)$ consists of those $R$-algebras $A$ such that $L\otimes_RA$ is clss over $L$. %\mw{?? this is not stated in \cite[Theorem~6.1]{magid}} \mw{-p} 
The associated Galois groupoid
$$
G=\Gal[f]
$$
is a profinite groupoid (a limit of finite groupoids with surjective transition maps).

Let $F_0\xrightarrow{\gamma_0}G_0$ be the total space of a $G$-action $F$. Following \cite[Definition~38]{magid}, $F$ is called
\begin{enumerate}
\item \emph{relatively finite}, if for each $x\in S_0(R)$, there is a finite set $C_x$ such that $\gamma_0^{-1}(S_0(L_x))$ is isomorphic to $S_0(L_x)\times C_x$ as an $S_0(L_x\otimes_{R_x}L_x)$-space. 
\item \emph{constant relatively finite}, if it is relatively finite with all $C_x$ equal.
\item \emph{pro-relatively finite}, if it is a limit of relatively finite $G$-actions.
\end{enumerate} 
The equivalence of categories from \ref{gal-general-rings} restricts to the
equivalence of full subcategories
$$
\Split^\text{\rm clss}_R(f)\simeq [\Gal[f],\cP_0]_\text{\rm prf}
$$
of the opposite category of clss extensions of $R$ split by $L$ and the category of pro-relatively finite profinite $\Gal[f]$-actions. Moreover, strongly separable objects correspond to constant relatively finite $G$-actions.
\end{fact}

\subsection{Magid's separable closure of a ring}\label{ss:magid-pi1}

\begin{definition}[{{\cite[Chapter 5]{magid}}}]\label{sep-cl}
\begin{enumerate}
\item A ring $L$ is \emph{separably closed} if, for every clss extension $M/L$, there is an $L$-algebra map $M\to L$.
\item A clss extension $L/R$ is \emph{minimal} if every $R$-algebra map $L\to M$ to a clss extension $M/R$ is a monomorphism.
\item A ring $L$ is a \emph{separable closure} of $R$, provided $L$ is separably closed and $L/R$ is a minimal clss extension.
\end{enumerate}
\end{definition}
{Note that in the case that $R$ is a field, the above corresponds to the usual notion of separable closure.} %\mw{?? Is this true, if so please do you know reference?}

\begin{fact}[{\cite[{Thm. }5.3]{magid}}]\label{sep-cl-exists}
A commutative ring $R$ has a separable closure. Any two separable closures of $R$ are isomorphic as $R$-algebras. 
\end{fact}

\begin{definition}[{{\cite[Def. 39]{magid}}}] \label{def: fundamental groupoid}
Let $\bar{R}$ be a separable closure of a ring $R$. The \emph{fundamental groupoid} is the profinite groupoid
$$
\pi_1(R,\bar{R})=\Gal[\bar{R}\to R]
$$
as in \ref{magid-th}.
\end{definition}

\begin{fact}
Let $\bar{R}$ be a separable closure of a ring $R$. 
\begin{enumerate}
\item Galois correspondence \ref{gal-general-rings} for the extension $\bar{R}/R$ gives an anti-equivalence of categories between the category of $R$-algebras $A$ such that $\bar{R}\otimes_R A$ is clss over $\bar{R}$ and the category of profinite $\pi_1(R,\bar{R})$-actions;
\item Galois correspondence \ref{magid-th} for $\bar{R}/R$ gives an anti-equivalence of categories between the category of clss extensions of $R$ and the category of pro-relatively finite $\pi_1(R,\bar{R})$-actions.
\end{enumerate}
 
\end{fact}

\subsection{Galois theory for fields}

The considerations of \ref{ss:magid-rings} and \ref{ss:magid-pi1} in the special case of field extensions are of particular interest.

\begin{fact}[{\cite[4.7.16, 4.5.5]{borceux-janelidze}}]\label{field-ext-gal-desc}
The morphism $f:L\to k$ in $\cA_0$ corresponding to a Galois extension $L/k$ of fields is of relative Galois descent. The group
$$
\Gal[f]\simeq \Gal(L/k)
$$
is the usual profinite Galois group of the field extension $L/k$. The category
$$
\Split_k(f)
$$
is the opposite of the category of $k$-algebras split by $L$, and the category
$$
[\Gal[f],\cP_0]
$$ 
is the category of profinite $\Gal(L/k)$-spaces. The functor
$$
\Split_k(f)\to [\Gal[f],\cP_0], \ \ \ A\mapsto S_0(L\otimes_kA)
$$
realising the equivalence is isomorphic to the functor
$$
A\mapsto k\Alg(A,L),
$$
%\mw{?? It is not good that exactly the same notation is used for algebras and difference algebras, maybe a zero could be added to all things occurring here but then this should also be added to the above? }
where the set $k\Alg(A,L)$ is given the profinite topology by writing $A$ as a colimit of finite-dimensional $k$-algebras $A_\lambda$, and then 
$$
k\Alg(A,L)=k\Alg(\colim_\lambda A_\lambda,L)=\lim_\lambda k\Alg(A_\lambda,L)
$$
is a limit of finite sets. Its quasi-inverse is the functor
$$
X\mapsto (f_! C_{0,L}(X))^{\Gal[f]},
$$
for $X\in[\Gal[f],\cP_0]$.

\end{fact} 

\begin{remark}\label{etale-algebras-class}
In the special case of the extension $\bar{k}/k$, where $\bar{k}$ is a separable closure of $k$, the above result states that the functor
$$
A\mapsto S_0(\bar{k}\otimes_k A)\simeq k\Alg(A,\bar{k})
$$
induces an anti-equivalence of categories between ind-\'etale $k$-algebras and profinite $G=\pi_1(k,\bar{k})=\Gal(\bar{k}/k)$-actions.
\end{remark}

\begin{lemma} \label{lemma: injectivesurjective}
	In the context of Remark \ref{etale-algebras-class}, a morphism of ind-\'etale $k$-algebras is injective/surjectve if and only if the corresponding morphism of profinite spaces is surjective/injective.
\end{lemma}
\begin{proof}
	We think this lemma is well-known, but for a lack of a suitable reference we include the proof. To begin with, note that $S_0(\bar{k}\otimes_k A)$ can also be identified with $\spec(\bar{k}\otimes_k A)$.
	
	If $A\to B$ is an injective morphism of ind-\'etale $k$-algebras, then also $\bar{k}\otimes_k A\to \bar{k}\otimes_k B$ is injective. Since $A$ is ind-\'etale, all prime ideals of $\bar{k}\otimes_k A$ are minimal. For any inclusion of rings, minimal prime ideals can be lifted to (minimal) prime ideals. Therefore $\spec(\bar{k}\otimes_k B)\to\spec(\bar{k}\otimes_k A)$ is surjective.
	
	Conversely, if  $\spec(\bar{k}\otimes_k B)\to\spec(\bar{k}\otimes_k A)$ is surjective, it is dominant and so $\bar{k}\otimes_k A\to \bar{k}\otimes_k B$ is injective. Therefore also $A\to B$ is injective.
	
	If $A\to B$ is surjective, so is $\bar{k}\otimes_k A\to \bar{k}\otimes_k B$ and therefore  $\spec(\bar{k}\otimes_k B)\to\spec(\bar{k}\otimes_k A)$ is a closed immersion. In particular, it is an injective map.
	
	Conversely, if  $\spec(\bar{k}\otimes_k B)\to\spec(\bar{k}\otimes_k A)$ is injective, it induces a homeomorphism onto a closed subset. (This holds for any continuous injective map of profinite spaces.) A morphism of schemes is a closed immersion if it induces a homeomorphism onto a closed subset and the induced maps on the stalks are surjective. The stalk at a point 
	of $\spec(\bar{k}\otimes_k A)$ is simply $\bar{k}$. So the induced maps on stalks are surjective and so $\spec(\bar{k}\otimes_k B)\to\spec(\bar{k}\otimes_k A)$ is a closed immersion of schemes. On the other hand, a morphism of affine schemes is a closed immersion if and only if the dual map is surjective. Thus $\bar{k}\otimes_k A\to \bar{k}\otimes_k B$ is surjective and therefore also $A\to B$ is surjective.
\end{proof}

\begin{remark}\label{zar-spectrum-quot}
If $X$ is the $G$-set associated to an ind-\'etale $k$-algebra $A$ as in \ref{etale-algebras-class}, using that the Zariski spectrum of a split $\bar{k}$-algebra can be identified with the space of its connected components, we obtain a homeomorphism
$$
\spec(A)\simeq X/G,
$$
where we consider the Zariski topology on the left, and the quotient of the profinite topology on the right. Indeed, since $A\simeq (f_! C_{\bar{k}}(X))^G$, it follows that
$$
\spec(A)\simeq\spec(f_! C_{0,\bar{k}}(X))/G\simeq \spec(C_{0,\bar{k}}(X))/G\simeq S_{0,\bar{k}}(C_{0,\bar{k}}(X))/G\simeq X/G.
$$
\end{remark}

\subsection{Galois theory for difference rings}\label{ss:diff-gal}

{We are now ready to combine the Galois theory of commutative rings with the difference Galois Theorem (Theorem~\ref{gen-diff-Galois-thm}) to obtain the Galois theorem for difference rings.}

Let $\cA_0$, $\cP_0$, $S_0$ and $C_0$ be as in \ref{class-gal-setup}.

\begin{definition}\label{difference-context}
\begin{enumerate} 
\item Let $$\cA=\diff\cA_0=(\diff\Rng)^\op$$ be the opposite category of the category of commutative unital difference rings.
\item Let $$\cP=\diff\cP_0$$ denote the category of difference profinite spaces.
\item The \emph{difference Pierce spectrum} is the functor  
$$S:\cA\to \cP, \ \ S(R,\sigma_R)=(S_0(R),S_0(\sigma_R)).$$ 

\item The \emph{difference ring of continuous functions} functor is 
$$
C:\cP\to \cA, \ \ C(X,\sigma_X)=(C_0(X),C_0(\sigma_X)).
$$
\end{enumerate}
\end{definition}

\begin{theorem}\label{diff-Galois-thm}
Let $f:L\to k$ be the morphism in $\cA$ corresponding to a difference ring extension $L/k$ such that $\forg{L}/\forg{k}$ is auto-split clss (as in \ref{def-clss}). 
The category 
$$
\Split_k(f)
$$
is the opposite category of $k$-algebras $A$ such that $\forg{A}$ is split by $\forg{L}$. The categorical Galois groupoid is the difference profinite groupoid {(an internal groupoid in $\cP$)}
%$$
%\Gal[f]=S(L\otimes_k L).   %\simeq \Aut\lbr L\rbr_k\simeq (\Gal(\forg{L}/\forg{k}),()^\sigma).
%$$
$$
\begin{tikzcd}[cramped, column sep=normal, ampersand replacement=\&]
{\Gal[f]=\left(S(L\otimes_kL) \right.}\ar[yshift=2pt]{r}{} \ar[yshift=-2pt]{r}[swap]{} \&{\left.S(L)\right)}
\end{tikzcd}
$$
We have an equivalence of categories
$$
\Split_k(f)\simeq [\Gal[f],\cP],
$$
realised by the functor
$$
A\mapsto S(L\otimes_k A). %\simeq \lbr A,L\rbr_k\simeq (\forg{k}\Alg(\forg{A},\forg{L}),()^\sigma).
$$
\end{theorem}
 \begin{proof}
 By \ref{magid-th}, such an $\forg{f}$ is of relative Galois descent and $\Gal[\forg{f}]$ is a profinite groupoid, so the statement follows from the Galois theorem for categories of difference objects \ref{gen-diff-Galois-thm} applied in the difference context \ref{difference-context}.
 \end{proof}

\begin{remark}
The relationship between our Galois correspondence \ref{diff-Galois-thm} for difference rings and Magid's theory \ref{ss:magid-rings} for rings  can be made precise via the Comparison Theorem~\ref{gen-gal-free-diff}, which applies to this context verbatim and yields the following.
\end{remark}

\begin{corollary}\label{eq-with-subshifts}
With notation of \ref{diff-Galois-thm}, the equivalence of categories 
$$
\Split_k(f)\simeq [\Gal[f],\cP],
$$
restricts to equivalences of the following full subcategories:
\begin{enumerate}
\item  finitely $\diff$generated objects of $\Split_k(f)$ and finitely $\diff$generated $G$-profinite sets;
%, i.e., RELATIVE!!!! subshifts of a full shift on a finite alphabet with an action of $G$;
\item finitely $\diff$presented objects of $\Split_k(f)$ and finitely $\diff$presented $G$-profinite sets,
% i.e., RELATIVE!!!! subshifts of finite type with an action of $G$.
%\item directly finitely presented objects of $\Split_k(f)$ and directly finitely presented $G$-profinite sets, i.e., subshifts of finite type with an action of $G$.
\end{enumerate}
where we say that a $G$-profinite set is \emph{finitely freely $\diff$generated,} if it is obtained by applying $\psig{\,}_G$ to a constant relatively finite $\forg{G}$-action as in \ref{magid-th}, and the notion of finitely $\diff$generated/presented $G$-actions is obtained following \ref{free-dirpres-diffob}.
\end{corollary}

\begin{proof}
%Let us write $\psig{A_0}_{k}$ for the object of $\cA_{\ov k}$ corresponding to the free $k$-algebra $\ssig{A_0}^k$ associated to an $\forg{f}$-split $\forg{k}$-algebra $A_0$.
%\mw{?? In (1) and (2) should it not be argued that the $A_0$ and $C_0$ can be chosen to be split by $f$?}
(1) A finitely $\diff$generated $k$-algebra $A$ split by $f$ is a quotient $\ssig{A_0}^k\to A$ for a finitely generated $\forg{k}$-algebra $A_0$ split by $\forg{f}$. This corresponds to a monomorphism $A\to \psig{A_0}_{k}$ in $\Split_k(f)$, so we conclude by \ref{gen-gal-free-diff} and \ref{magid-th}.

(2)  A finitely $\diff$presented $k$-algebra $A$ split by $f$ has a direct presentation of the form $\ssig{A_0;C_0}^k$ for some finitely presented $\forg{k}$-algebras $A_0$ and $C_0$ split by $\forg{f}$, which corresponds to a direct presentation $\psig{A_0;C_0}_{k}$ in $\Split_k(f)$, so the statement follows from \ref{gal-dif-corr} and \ref{magid-th}.
\end{proof}

\subsection{Difference fundamental groupoid}\label{ss:diff-fund}

{We would like to define the difference fundamental groupoid in analogy to Definition \ref{def: fundamental groupoid}. For this purpose we first need to clarify the notion of separable closure of a difference ring.}

\begin{lemma}\label{si-lift-sep-cl}
Let $k$ be a difference ring, and let $L_0$ be a separable closure of $\forg{k}$ (as in \ref{sep-cl}). Then $L_0$ can be made into a difference ring extension of $k$.
\end{lemma}
\begin{proof}
As in the proof of uniqueness of separable closure \cite[5.3]{magid}, 
since $L_0/\forg{k}$ is clss, the difference twist $L_1=L_{0,\sigma_k}=L_0\otimes_k k$ is again clss over $\forg{k}$. Hence, $L_1\otimes_k L_0$ is clss over $L_0$,
% \mw{?? please add references for the two properties of clss used here} 
and, since $L_0$ is separably closed, there is an $L_0$-algebra homomorphism $L_1\otimes_k L_0\to L_0$, whence we obtain a $\forg{k}$-algebra homomorphism 
$$
L_1\to L_1\otimes_k L_0\to L_0.
$$
Let $\bar{\sigma}$ be the composite
$$
L_0\to L_1\to L_0,
$$
of the canonical $\sigma_k$-linear map and the above $\forg{k}$-algebra map. By construction,
$$
(L_0,\bar{\sigma})
$$
is a difference ring extension of $k$.
\end{proof}
\begin{definition}
A \emph{separable closure} of a difference ring $k$ is a difference ring extension $\bar{k}$ of $k$ such that $\forg{\bar{k}}$ is a separable closure of $\forg{k}$.
\end{definition}

\begin{remark}\label{diff-sep-cl-exists}
Separable closures of a difference ring exist by \ref{sep-cl-exists} and \ref{si-lift-sep-cl}. However, unlike in the case of rings, two separable closures of a difference ring are not necessarily isomorphic. {For example, for $k=(\mathbb{Q},\id)$ the separable closures $(\bar{\mathbb{Q}},\id)$ and $(\bar{\mathbb{Q}},\sigma)$, with $\sigma$ any non-trivial automorphism of $\bar{\mathbb{Q}}$, are not isomorphic.}
\end{remark}

\begin{remark}\label{si-etale}
Let $\bar{k}$ be a separable closure of a difference ring $k$. Using \ref{magid-th} and \ref{gen-diff-split-normal} and, the following statements are equivalent for a $k$-algebra $A$:
\begin{enumerate}
\item $A\in \Split_k(\bar{k}\to k)$;
\item $\forg{\bar{k}\otimes_k A}$ is clss over $\forg{\bar{k}}$.
%\mw{?? Is there a characterization that does not refer to $\bar{k}$?}
\end{enumerate}
\end{remark}

\begin{definition}\label{def-si-etale}
A difference ring extension $A/k$ is \emph{difference locally \'etale} if it satisfies the equivalent conditions from \ref{si-etale}.
\end{definition}

\begin{definition}\label{diff-pi1}
Let $\bar{k}$ be a separable closure of a difference ring  $k$. The associated \emph{fundamental difference groupoid} is the difference profinite groupoid 
$$
\pi_1(k,\bar{k})=\Gal[\bar{k}\to k]
$$
as in \ref{diff-Galois-thm}.
\end{definition}

\begin{corollary}\label{diff-pi1-correspondence}
Let $\bar{k}$ be a separable closure of a difference ring  $k$. The difference Galois correspondence \ref{diff-Galois-thm} for the extension $\bar{k}/k$ establishes an anti-equivalence of categories between the category of difference locally \'etale $k$-algebras  and the category of difference profinite $\pi_1(k,\bar{k})$-actions.
\end{corollary}
We invite the reader to formulate the versions of the above theorem with additional $\diff$finiteness assumptions as in \ref{eq-with-subshifts}.

\subsection{Galois theory of difference fields}\label{ss:diff-fields}

{When working with difference fields instead of general difference rings we can make some of the above constructions more explicit.}

\begin{lemma} \label{lemma:  adjoint for connected}
If $\forg{k}$ is a connected ring {(i.e., $S(k)$ is a singleton)}, then {(with the notation of \ref{difference-context})} the
 right adjoint to 
 $$
 S_k:k\Alg^\op\to \cP_{\ov S(k)}=\cP, \ \ S_k(A)=S(A)
 $$
 is the functor
 $$
 C_k:\cP\to k\Alg^\op, \ \ X\mapsto(\text{\rm Cont}(\forg{X},\forg{k}), f\mapsto \sigma_k f\sigma_X),
 $$
 where $\forg{k}$ is considered with the discrete topology.  
\end{lemma}
\begin{proof}
The adjunction follows directly from \ref{diff-adjunction}. Since $\cA$ admits pullbacks, the 
general construction of relative right adjoints described in \ref{janelidze} {and the fact that $C(S(k))=(\Z,\id)$ because $\forg{k}$ is connected,} gives that 
$$
C_k(X)=C(X)\otimes_{(\Z,\id)}k\simeq (C_0(\forg{X}),C_0(\forg{\sigma_X)})\otimes_{(\Z,\id)}(k,\sigma_k),
$$
so the result for $\forg{k}$ connected follows from \cite[4.3.5]{borceux-janelidze} stating that
$$
\forg{k}\otimes_\Z \text{\rm Cont}(\forg{X},\Z)\simeq \text{\rm Cont}(\forg{X},\forg{k}).
$$
\end{proof}

{For the following lemma, recall the definition of the topology on $\forg{k}\Alg(\forg{A},\forg{L})$ in \ref{field-ext-gal-desc} and the definition of the enriched homs $\lbr A,L\rbr_k$ from Appendix \ref{sec: Appendix C}.}

\begin{lemma} \label{lemma: prepare for fields}
Let $L/k$ be a difference field extension %\mw{such that $\forg{L}/\forg{k}$ is algebraic}, 
and let $A\in k\Alg$ be split by $L/k$. 
\begin{enumerate}
\item There exists a {unique} continuous map $()^\sigma: \forg{k}\Alg(\forg{A},\forg{L}) \to \forg{k}\Alg(\forg{A},\forg{L})$ satisfying 
$$
\sigma_L\circ f^\sigma=f\circ \sigma_A.
$$
\item If $\forg{L}/\forg{k}$ is Galois, we have isomorphisms of difference profinite sets 
$$
S_L(L\otimes_k A)\simeq \lbr A,L\rbr_k\simeq (\forg{k}\Alg(\forg{A},\forg{L}),()^\sigma).
$$
\end{enumerate}
\end{lemma}
\begin{proof}
In view of the fact that $\lbr A,L\rbr_k\simeq \lbr L\otimes_k A,L\rbr_L$, and 
$\forg{k}\Alg(\forg{A},\forg{L})\simeq  \forg{L}\Alg(\forg{L\otimes _kA},\forg{L})$, we can reduce to the case where $L=k$ and $A=C_L(X)$ for some difference profinite set $X$. Note that, by compactness of $\forg{X}$, $\forg{A}=C_{0,\forg{L}}(\forg{X})=\text{\rm Cont}(\forg{X},\forg{L})$ is spanned as an $\forg{L}$-vector space by the idempotents $E(A)$, which are in fact characteristic functions of clopen subsets of $X$. Hence, each $f\in \forg{L}\Alg(\forg{A},\forg{L})$ is uniquely determined by $I(f)=\ker{f}\cap E(A)$, a prime ideal in $E(A)$, hence an element of $\spec( E(A))=S_L(C_L(X))\simeq X$.

Given an $f\in \forg{L}\Alg(\forg{A},\forg{L})$, the morphism $f^\sigma$ is determined by stipulating 
$$
I(f^\sigma)=\sigma_{E(A)}^{-1} (I(f)),
$$
i.e., by the action of $\sigma_X$. 

The assignment $I$ establishes $\cP$-isomorphisms 
$$
(\forg{L}\Alg(\forg{A},\forg{L}),()^\sigma)\simeq X \simeq S_L(A).
$$
The map 
$$
\lbr A,L\rbr_L\to \forg{L}\Alg(\forg{A},\forg{L}), \ \ \ (f_0,f_1,\ldots)\mapsto f_0
$$
is an isomorphism of difference sets with inverse
$$
f\mapsto (f,f^\sigma,f^{\sigma^2},\ldots),
$$
and it remains to show that it is a homeomorphism. It is enough to show that topologies match for $A$ {finitely $\diff$generated} over $k$, because an arbitrary $k$-algebra is a colimit of such. 

Let $a$ be a finite tuple of $\sigma$-generators of $A$ over $k$, so that $\forg{A}={\forg{k}}[a,\sigma(a),\sigma^2(a),\ldots]$.
As explained in \ref{field-ext-gal-desc}, one can describe the topology on $\forg{k}{\Alg}(\forg{A},\forg{L})$ by writing
$$\forg{A}=\colim_n A_n, \ \ \ \text{where}\ \ \ A_n={\forg{k}}[a,\sigma(a),\ldots,\sigma^n(a)],$$
and then considering the profinite topology of
$$
\forg{k}\Alg(\forg{A},\forg{L})\simeq \lim_n \forg{k}\Alg(A_n,\forg{L}).
$$
In terms of the topology of pointwise convergence on the function space $\forg{k}\Alg(\forg{A},\forg{L})$, we can choose a sub-basis for this topology consisting of the sets of form 
$$
\langle \sigma^n{(a)},\lambda\rangle=\{ f\in \forg{k}{\Alg}(\forg{A},\forg{L}): f({\sigma^n(a)})=\lambda\},
$$
for $\lambda\in L$ and $n\in\N$. The existence of $()^\sigma$ shows that  $\langle \sigma^n{(a)},\lambda\rangle=\emptyset$ unless $\lambda=\sigma^n{(b)}$ for some ${b}\in L$, and then the set $\langle \sigma^n{(a)},\sigma^n{(b)}\rangle$ corresponds to the sub-basic set $\langle a,{b};n\rangle$ for the topology on $\lbr A,L\rbr_k$ as in \ref{subbasis-funspacealg}.
\end{proof}

{The following theorem follows by combining Theorem \ref{diff-Galois-thm} and Lemma \ref{lemma: prepare for fields}.}

\begin{theorem}\label{diff-Galois-thm-fields}
Let $f:L\to k$ be the morphism in $\cA$ corresponding to a difference field extension $L/k$ such that $\forg{L}/\forg{k}$ is Galois. 
The category 
$$
\Split_k(f)
$$
is the opposite category of $k$-algebras $A$ such that $\forg{A}$ is split by $\forg{L}$. 
The categorical Galois group is the difference profinite group %(a group object in $\cP$)
$$
\Gal[f]=S(L\otimes_k L)\simeq \Aut\lbr L\rbr_k\simeq (\Gal(\forg{L}/\forg{k}),()^\sigma).
$$
We have an equivalence of categories
$$
\Split_k(f)\simeq [\Gal[f],\cP],
$$
realised by the functor
$$
A\mapsto S(L\otimes_k A)\simeq \lbr A,L\rbr_k\simeq (\forg{k}\Alg(\forg{A},\forg{L}),()^\sigma)
$$
whose quasi-inverse is the functor
$$
X\mapsto (f_! C_L(X))^{\Gal[f]},
$$
for $X\in[\Gal[f],\cP]$. {Here the action of $\Aut\lbr L\rbr_k$ on $\lbr A,L\rbr_k$ and of $(\Gal(\forg{L}/\forg{k}),()^\sigma)$ on $(\forg{k}\Alg(\forg{A},\forg{L}),()^\sigma)$ respectively is given by composition.}
\end{theorem}

\begin{remark} 
An algebra $A$ over a difference field $k$ is difference locally \'etale (as in \ref{def-si-etale}) if and only if $\forg{A}$ is locally strongly separable (i.e.,  ind-\'etale, in view of \ref{str-sep-fin-etale}) over $\forg{k}$.  In other words, $A$ is difference locally \'etale if and only if every $a\in A$ satisfies a separable polynomial over $\forg{k}$.
\end{remark}

\begin{definition}\label{diff-etale} 
%\mw{?? I think this definition is confusing. Can it not just be difference locally \'etale plus finitely $\diff$generated/presented as a $k$-algebra. It seems to me that the equivalence of these definition was implicitly used in the proof of Cor. 3.29 already }
Let $k$ be a difference field and let $A$ be a $k$-algebra. We say that
\begin{enumerate}
\item $A$ is \emph{finitely $\diff$generated \'etale}, if it is difference locally \'etale and a quotient of $\ssig{A_0}^k$ for some finitely generated \'etale $\forg{k}$-algebra $A_0$;
\item $A$ is \emph{finitely $\diff$presented \'etale}, if $A\simeq\ssig{A_0;C_0}^k$ for some cocorrespondence $A_0\to C_0\leftarrow A_0$ of \'etale $\forg{k}$-algebras. 
\end{enumerate}
\end{definition}

{As the special case $L=\bar{k}$ of Theorem \ref{diff-Galois-thm-fields} we obtain the following.}
\begin{corollary}\label{diff-abs-Gal}
Let $\bar{k}$ be a separable closure of the difference field $k$ (with a non-canonical lift $\bar{\sigma}$ of $\sigma_k$), and consider the difference profinite absolute Galois group
$$
G=S(\bar{k}\otimes_k\bar{k})\simeq \Aut\lbr \bar{k}\rbr_k\simeq (\Gal(\forg{\bar{k}}/\forg{k}),()^{\bar{\sigma}}).
$$
The functor 
$$
A\mapsto S(\bar{k}\otimes_k A)\simeq \lbr A,\bar{k}\rbr_k\simeq (\forg{k}\Alg(\forg{A},\forg{\bar{k}}),()^{\bar{\sigma}})
$$
gives rise to an anti-equivalence of categories between the category of difference locally \'etale $k$-algebras 
%$k$-algebras $A$ such that $\forg{A}$ is an ind-\'etale $\forg{k}$-algebra, 
and the category of difference profinite sets with an action of $G$. 
\end{corollary}

In view of \ref{eq-with-subshifts}, we obtain the following.

\begin{corollary}\label{fp-corr-sft}
The anti-equivalence of categories from \ref{diff-abs-Gal} restricts to anti-equivalences of the following full subcategories:
\begin{enumerate}
\item finitely $\diff$generated \'etale $k$-algebras and finitely $\diff$generated $G$-profinite sets. %, i.e., subshifts with an action of $G$;
\item finitely $\diff$presented \'etale $k$-algebras and finitely $\diff$presented $G$-profinite sets. % i.e., subshifts of finite type with an action of $G$.
%\item directly finitely presented \'etale $k$-algebras and directly finitely presented $G$-profinite sets, i.e., subshifts of finite type with an action of $G$.
\end{enumerate}
The appropriate  notion of finite $\diff$generation/presentation is in the sense of terminology \ref{free-dirpres-diffob}, where we stipulate that  a \emph{finitely freely $\diff$generated $G$-profinite set} is of the form $\psig{X_0}_G$ for a finite set $X_0$ with a continuous $\forg{G}$-action (cf. \ref{psig-G-sets}).
%In the above, we say that 
%
%
% the appropriate notion of finite generation to make sense of the terminology in \ref{free-dirpres-diffob} is simply that of a finite (discrete) set, i.e., a finitely freely $\diff$generated $G$-profinite set is of the form $\psig{X}_G$ for a finite set $X$ with a continuous $\forg{G}$-action (cf. \ref{psig-G-sets}).
\end{corollary}

\begin{definition}
Let $R$ be a difference ring. We write
$$
\spec(R)=(\spec(\forg{R}),\lexp{a}{\sigma}_R),
$$
where $\lexp{a}{\sigma}_R=\spec(\sigma_R)$ is the map on $\spec(\forg{R})$ associated with the action of $\sigma_R$ on prime ideals via
$$
\lexp{a}{\sigma}_R(\p)=\sigma_R^{-1}(\p).
$$
\end{definition}

\begin{corollary}\label{diffet-spec}
With notation of \ref{diff-abs-Gal}, let $A$ be a difference \'etale $k$-algebra
%$k$-algebra such that $\forg{A}$ is an ind-\'etale $\forg{k}$-algebra, 
and let $X$ be the associated $G$-profinite space. Then we have an isomorphism of difference topological spaces
$$
\spec(A)\simeq X/G.
$$
\end{corollary}
\begin{proof}
Identifying $X\simeq(\forg{k}\Alg(\forg{A},\forg{\bar{k}}, ()^{\bar{\sigma}})$ as in  \ref{diff-abs-Gal}, the $\forg{G}$-equivariant map $\forg{X}\to \spec(\forg{A})$ 
inducing the homeomorphism 
$$
\spec(\forg{A})\simeq \forg{X}/\forg{G}
$$
from \ref{zar-spectrum-quot} is given by 
$f\mapsto \ker(f)$, for $f\in \forg{k}\Alg(\forg{A},\forg{\bar{k}})$. 

Using the property $\bar{\sigma} \circ f^{\bar{\sigma}}=f\circ\sigma_A$, we readily verify that
$$
\ker(f^{\bar{\sigma}})=\sigma_A^{-1}\ker(f)=\lexp{a}{\sigma}_A(\ker(f)),
$$
so this assignment is compatible with the difference structure, as required.
\end{proof}

\section{Symbolic dynamics}\label{s:symb}

{Subshifts and subshifts of finite type are the central objects of study in symbolic dynamics. The connection between difference algebra and symbolic dynamics manifests itself through the fact that, under the Galois theorem (\ref{diff-Galois-thm-fields}), finitely $\diff$generated difference algebras correspond to subshifts, while finitely $\diff$presented difference algebras correspond to subshifts of finite type. Moreover, classical definitions and constructions on both sides correspond to each other, e.g., the limit degree on the difference algebra side corresponds to the entropy on the symbolic dynamics side.}

{In this section} we give a whirlwind tour of symbolic dynamics following mostly classical texts such as \cite{kitchens} and \cite{lind-marcus}.

\subsection{Subshifts of finite type}

\begin{definition}
The \emph{one-sided full shift} on an alphabet consisting of a finite set $E_0$ is the profinite space 
$$X=E_0^{\N},$$ together with the shift 
$\sigma:X\to X$,
$$
\sigma(x_0,x_1,x_2,\ldots)=(x_1,x_2,\ldots). 
$$
It is naturally a profinite space with a continous self-map. Moreover, 
it is metrisable by the metric
$$
d(x,y)=2^{-\min\{i: x_i\neq y_i\}}.
$$

A \emph{subshift} is a closed subset of some full shift which is {stable} under the shift. 

Let $\Gamma$ be the directed graph with vertices $X_0$ and edges $E_0\xrightarrow{(s,t)} X_0\times X_0$.
The \emph{subshift of finite type} {(SFT)} associated with $\Gamma$ is the set of infinite paths in $\Gamma$,
$$
X_\Gamma=\{(x_0,x_1,x_2,\ldots)\in E_0^\N : s(x_{i+1})=t(x_i)\},
$$
considered both as a difference subset and a subspace of the full shift.
\end{definition}

{In what follows we take the liberty to also refer to objects of $\diff\Prof$ isomorphic to subshifts/SFTs as subshifts/SFTs.}

\begin{remark}\label{dir-si-gen-set-sft}
A full shift on {a finite} alphabet $E_0$ is  a {finitely freely $\diff$generated object of $\diff\Prof$, i.e.,
$$
E_0^\N=\psig{E_0}\in\diff\Prof,
$$
where, as in Proposition \ref{psig-adj}, $\psig{\,}$ is the adjoint of the forgetful functor $\forg{\,}\colon \diff\Prof\to\Prof$ and $E_0$ is considered as a discrete  topological space.}
In categorical terms, a subshift is a subobject of a full shift in $\diff\Prof$.

A finite (discrete) directed graph $\Gamma=(X_0,E_0)$ defines a self-1-correspondence in the category $\Prof$, where $1\in \diff\Prof$ is a terminal object consisting of a singleton with the identity self-map. The subshift of finite type associated with $\Gamma$ is in fact a finitely directly {$\diff$presented object of $\diff\Prof$ (in the sense of \ref{free-dirpres-diffob})},
$$
X_\Gamma=\psig{X_0,E_0}_1\in \diff\Prof.
$$
\end{remark}

\begin{remark}
By adding and relabeling vertices and edges, we may assume that in the directed graph $\Gamma$ we have
$E_0\subseteq X_0\times X_0$ and that $X_0=\{1,\ldots,n\}$, so that $\Gamma$ is determined by its $n\times n$ transition matrix $A$ with entries from $\{0,1\}$. In that case, we can denote the associated subshift  by
$$
X_A=\{ (x_0,x_1,\ldots)\in \{1,\ldots,n\}^\N : A_{x_i,x_{i+1}}=1\}.
$$
\end{remark}

\begin{definition}
Let $X$ be a SFT. A point $x\in X$ is called
\begin{enumerate}
\item \emph{periodic}, if there exists an $i>0$ such that $\sigma^i(x)=x$;
\item \emph{preperiodic}, if some $\sigma^i{(x)}$ is periodic, or, equivalently, if its orbit is finite;
\item \emph{transitive}, if its orbit is dense;
\item \emph{wandering}, if there exists a neighbourhood $U\ni x$ such that $\sigma^i{(U)}\cap U=\emptyset$ for $i>0$.
\end{enumerate}
\end{definition}

\begin{fact}[{\cite[4.4]{lind-marcus}}]\label{comm-classes}
Let $\Gamma=(X_0;E_0)$ be a directed graph. For vertices $x,y\in X_0$, we write 
\begin{enumerate}
\item $x\rightsquigarrow y$ if there is a (possibly empty) path in $\Gamma$ from $x$ to $y$;
\item $x\leftrightsquigarrow y$ if $x\rightsquigarrow y$ and $y\rightsquigarrow x$.
\end{enumerate}
The relation $\leftrightsquigarrow$ is an equivalence relation and its classes are called \emph{communicating classes}.
We form the \emph{graph of communicating classes}
$$
\bar{\Gamma}=(X_0/\leftrightsquigarrow, \bar{E}_0)
$$
whose vertices are the communicating classes, and for two communicating classes $C$, $D$, we insert {an} edge $(C,D)\in \bar{E}_0$ if there exists an $x\in C$ and $y\in D$ such that $(x,y)\in E_0$.

By construction, $\bar{\Gamma}$ is acyclic, and we can arrange the communicating classes $C_1,\ldots,C_k$ in an order so that there is a path in $\bar{\Gamma}$ from $C_j$ to $C_i$ if and only if $j>i$, hence the transition matrix $A$ assumes the block triangular form
$$
\begin{bmatrix}
A_1 & 0 &  \ldots & 0 \\
* & A_2 &  \ldots & 0\\

\vdots & \vdots &  \ddots & \vdots \\
* & * &  \ldots & A_k
\end{bmatrix}
$$
%\begin{bmatrix}
%A_1 & 0 & 0 & \ldots & 0 \\
%* & A_2 & 0 & \ldots & 0\\
%*  &  *  & A_3 &\ldots & 0 \\
%\vdots & \vdots & \vdots & \ddots & \vdots \\
%* & * & * & \ldots & A_k
%\end{bmatrix}
where the matrices $A_i$ are transition matrices of the subgraphs $\Gamma_i$ of $\Gamma$ obtained by restricting the graph structure to the class $C_i$. These matrices %\mw{?? remove because not yet defined?? are \emph{irreducible} because they} 
correspond to \emph{strongly connected} graphs $\Gamma_i$.
\end{fact}

\begin{definition}
A SFT $X$ is called
\begin{enumerate}
\item \emph{irreducible}, if it is of {the} form $X=X_\Gamma$ for a strongly connected directed graph $\Gamma$;
\item \emph{transitive}, if it has a transitive point;
\item \emph{non-wandering}, if every point is non-wandering.
\end{enumerate}
\end{definition}

\begin{fact}\label{prop-irred-sft}
\begin{enumerate}
\item A SFT is transitive if and only if it is irreducible {(\cite[Thm.~1.4.1]{kitchens})}.
\item An irreducible SFT $X$ is either finite consisting of a single $\sigma_X$-orbit, or it is infinite and has countably many periodic points and they are dense.
\end{enumerate}
\end{fact}

\begin{fact}\label{irred-comp-sft} {(\cite[Observation 5.1.1]{kitchens})}
Let $X=X_\Gamma$ be a SFT. The subshifts $X_{\Gamma_i}$ corresponding to the strongly connected components $\Gamma_i$ of $\Gamma$ in \ref{comm-classes} are called \emph{irreducible components} of $X$.

The non-wandering set $\Omega_\Gamma$ of $X$ decomposes as a finite union of \emph{irreducible components} of $X$,
$$
\Omega_\Gamma=X_{\Gamma_1}\coprod \cdots\coprod X_{\Gamma_r}.
$$
\end{fact}

\subsection{Entropy}
{See \cite[Chapter 4]{lind-marcus}.}

\begin{definition}\label{top-entropy}
Let $X$ be a subshift, and let  
$$
W(X,l)
$$
be the set of words of length $l$ appearing in points of $X$. The \emph{topological entropy} of $X$ is
$$
h(X)=\lim_{l\to\infty} \frac{1}{l}\log |W(\Gamma,l)|.
$$
\end{definition}

\begin{fact}
Let $X=X_A$ be the SFT associated to a transition matrix $A$. Then $W(X,l)$ is the set of paths of length $l$ in the graph associated to $A$, whence
$$
|W(X,l)|=\sum_{i,j}(A^{l-1})_{ij}.
$$
Perron-Frobenius theory for the matrix $A$ yields that
$$
h(X)=\log\lambda,
$$
where $\lambda$ is the spectral radius of $A$. If $X$ is irreducible, then $\lambda$ is the Perron value of $A$.
\end{fact}

\subsection{Zeta function of a subshift of finite type}\label{zeta-sft}

Let $X=X_A$ be the subshift of finite type associated to a transition matrix $A$. Let
$$
N(X,n)=|\Fix(\sigma_X^n)|=|\begin{tikzcd}[cramped, column sep=normal, ampersand replacement=\&]
{\Eq\left(X\right.}\ar[yshift=2pt]{r}{\sigma^n} \ar[yshift=-2pt]{r}[swap]{\id} \&{\left.X\right)}
\end{tikzcd}|
$$
be the number of periodic points of period $n$.  The \emph{dynamical zeta function} of $X$ is the formal power series {(\cite[p.~24]{kitchens} and \cite[Section~6.4]{lind-marcus})}
$$
Z(X,t)=\exp\left(\sum_{n=1}^\infty \frac{N(X,n)}{n}t^n\right).
$$ 
Noting that $N(X,n)=\mathop{\rm Tr}(A^n)$, familiar arguments {(\cite[Thm. 6.4.6]{lind-marcus})} give
$$
Z(X,t)=\exp\left(\sum_{n=1}^\infty \frac{\mathop{\rm Tr}(A^n)}{n}t^n\right)=\frac{1}{\det(I-tA)},
$$
so $Z(X,t)$ converges to a rational function in $t$ with radius of convergence $e^{-h(X)}$, where $h(X)$ is the topological entropy of $X$.

\subsection{Block maps}

%\mw{?? Shall we do this section for subshifts and not just SFTs? The two statements should still be true \cite[Thm. 6.2.9]{lind-marcus} }

\begin{definition}
Let $X$ and $X'$ be subshifts, and let 
$$
\varphi_0: W(X,l)\to W(X',1)
$$
be a map consistent with $X$ and $X'$ in the sense that 
$$
[i_0,\ldots,i_l]\in W(X,l+1)\ \text{ implies }\ [\varphi_0([i_0,\ldots,i_{l-1}]),\varphi_0([i_1,\ldots,i_l])]\in W(X',2).
$$
The \emph{$l$-block map determined by $\varphi_0$}  is
$$
\varphi: X\to X',   (\varphi(x))_n=\varphi_0([x_n,\ldots,x_{n+l-1}]).
$$
It is continuous and commutes with the shifts.
\end{definition}

\begin{fact}[{\cite[6.2.9, 1.5.12]{lind-marcus}}]\label{sft-map-block}
A continuous shift-commuting map $\varphi:X\to X'$ between two subshifts is a block map. If $X$ and $X'$ are subshifts of finite type, by relabelling the presentation  of $X$, we may assume that $\varphi$ is a 1-block map, i.e., that it is induced by a  directed graph map $\varphi_0:\Gamma\to\Gamma'$, %\mw{?? That does not seem right to me, I don't think all morphism can be recoded to graph maps}
$$
\varphi=\psig{\varphi_0}:\psig{\Gamma}\to\psig{\Gamma'}.
$$
\end{fact}

\begin{remark}
Fact~\ref{sft-map-block} shows that the category of subshifts of finite type is a full subcategory of the category $\diff\Topl$ of difference topological spaces.
\end{remark}

\section{Interactions between symbolic dynamics and difference algebra}\label{s:interactions}

In this section we {explain} the connection between difference algebra and symbolic dynamics, which we exploit to prove a number of results of shared interest. The translation mechanism between the two areas is provided by the following theorem.

%\mw{?? I decided to fledge out the translation mechanism because 1) if someone from symbolic dynamics indeed takes an interest this is were he/she would look and so it would be good if it is understandable without having to scan the whole paper for definitions  2) to make it look more impressive 3) points (1) and (2) are implicitly used later on and good to know. We could maybe also add sofic and difference field corresponding to transitive action.
%And also finitely freely $\diff$generated corresponds to full shift. What do you think? }
%

%\mw{

\begin{theorem}[Translation mechanism] \label{theo:translation-mech}
	Let $k$ be a difference field and $\bar{k}$ a separable closure of $k$, i.e., $\bar{k}$ is the separable algebraic closure of $\forg{k}$ equipped with an extension of $\sigma_k$. Let $G=\pi_1(k,\bar{k})=\Gal[\bar{k}/k]$ be the absolute Galois group of $\forg{k}$ equipped with the continuous endomorphism $()^\sigma$ as in Lemma \ref{lemma: prepare for fields}.
	Then the functor $$A\mapsto X(A)=(\forg{k}\Alg(\forg{A},\forg{\bar{k}}),()^{\sigma})$$ defines an anti-equivalence of categories between the category of difference locally \'etale $k$-algebras and the category $[G,\diff\Prof]$ of difference profinite spaces equipped with a $G$-action. (The action is understood to be continuous and compatible with $\sigma$ in the sense of Definition \ref{defi: diff group action}.)  
	Moreover, under this correspondence:
	\begin{enumerate}
		\item\label{surj} a surjective map of $k$-algebras corresponds to an injective map of difference profinite spaces;
	\item\label{inj} an injective map of $k$-algebras corresponds to a surjective map of difference profinite spaces;
%		\item the tensor product of $k$-algebras corresponds to the product of difference profinite spaces;
%\item\label{fullsh} the difference algebra $A$ is freely finitely $\diff$generated if and only if $X(A)$ is the full shift on a finite alphabet; FALSE???
		\item\label{subsh} the difference algebra $A$ is finitely $\diff$generated if and only if $X(A)$ is a subshift;
		\item\label{sftsh} the difference algebra $A$ is finitely $\diff$presented if and only if $X(A)$ is a subshift of finite type.
	\end{enumerate}	
\end{theorem}
%}

%\mw{?? For (1) and (2) we discussed using monos and epis. While I think it is true that monos are inj and epis are surj that does not seem obvious. So I think it would be easier if we could find a reference (do you know one?) for the corresponding statement for rings and then our result follows immediately from that. For (4) and (5), it seems strange to me that between here and \ref{fp-corr-sft} the $G$ is lost without explanation. I think to go from finitely $\diff$generated/presented $G$-profinite sets to finitely $\diff$generated/presented difference profinite sets is okay because the underlying construction/coequalizer (without $G$-action) agrees but to go the other way I think needs some work. For example, having a full shift $E_0^\N$ with a $G$-action does not imply that $G$ is acting on $E_0$ and that the action is as in Example \ref{psig-G-sets} but maybe it can always be recoded to be that way. I think the problem is similar to showing that in terms of Definition 6.2.13 of my Habilitation theses a continuous difference actions is always $\sigma$-continuous. I think this is true but needs some kind of compactness argument as in Ribes Zaleski Profinite Groups Lemma 5.6.4. Do you agree?}

%\mw{
\begin{proof}
The equivalence of categories is just a restatement of Corollary \ref{diff-abs-Gal}.
Points (\ref{surj}) and (\ref{inj}) follow from Lemma \ref{lemma: injectivesurjective}.

For (\ref{subsh}),  let $X$ be a subshift of $\psig{X_0}$ with a $G$-action. We will show that there exists a finite discrete set $X_0'$ with a continuous $\forg{G}$-action and an injective morphism  $X\to\psig{X_0'}_G$ in $[G,\diff\Prof]$. Then, \ref{fp-corr-sft} will give that the difference algebra corresponding to $X$ is a quotient of the freely finitely $\diff$generated algebra associated to $\psig{X_0'}_G$, hence it is finitely $\diff$generated.  

Indeed, let $\pi_0: \psig{X_0}\to X_0$ be the projection to the first component, and define an open equivalence relation $R\subseteq X\times X$ by $(x,x')\in R$ if $\pi_0(x)=\pi_0(x')$. As in \cite[Lemma~5.6.4]{ribes-zal}, let $R'=\bigcap_{g\in\forg{G}}gR$ be the associated $\forg{G}$-invariant open equivalence relation, and consider the finite set $X_0'=X/R'$ with the $\forg{G}$-action $g.(x/R')=(gx)/R'$. The morphism
$$
\iota:X\to \psig{X_0'}_G, \ \ \ x\mapsto (x/R',\sigma(x)/R',\ldots)
$$
is injective because $R'\subseteq R$, and it is $G$-equivariant since
%\begin{multline*}
$$
gx\mapsto ((gx)/R,\sigma(gx)/R,\ldots)
=((gx)/R,(\sigma_G(g)\sigma(x))/R,\ldots)=(g.(x/R), \sigma_G(g).(\sigma(x)/R),\ldots),
$$
%\end{multline*}
which agrees with the $G$-action on $\psig{X_0'}_G$ described in \ref{psig-G-sets}. 

%Proof idea: Let $X$ be a subshift on the alphabet, $E_0=\{e_1,\ldots,e_n\}$, i.e., $X\subseteq E_0^\N$ is closed and stable under $\sigma$. For $i=1,\ldots,n$ set $$U_i=\{x=(x_0,x_1,\ldots)\in X|\ x_0=e_i\}.$$ Then the $U_i$ are an open disjoint cover of $X$. They thus define an open equivalence relation $R\subseteq X\times X$. As in Ribes Zaleski ?? show that $R'=\cap_{g\in G}g(R)$ is an open equivalence relation on $X$. Set $E_0'=X/R'$ and let $X\to E_0'=X/R',\ x\mapsto \overline{x}$ be the canonical map. Define  $\phi\colon X\to (E'_0)^\N,\ x\mapsto (\overline{\sigma^i(x)})_{i\in\N}$.??

If $X$ is a subshift of finite type, then \ref{sft-map-block} tells us that $\iota$ is an $l$-block map for some $l$, determined by a map $\iota_0:W(X,l)\to X_0'$. Let $E_0'=\{[\iota_0([x_0,\ldots,x_{l-1}]),\iota_0([x_1,\ldots,x_l])]: [x_0,\ldots,x_l]\in W(X,l+1)\}\subseteq X_0'\times X_0'$. By construction,
$$
X\simeq\psig{X_0';E_0'}_G,
$$
so \ref{fp-corr-sft} gives that the difference algebra corresponding to it is finitely $\diff$presented, proving  (\ref{sftsh}).

%	the difference Galois correspondence \ref{diff-Galois-thm-fields} and \ref{fp-corr-sft} in view of \ref{dir-si-gen-set-sft}.
\end{proof}
%}

%\begin{remark}[Translation mechanism]\label{translation-mech}
%If $\bar{k}$ is a separable closure of $k$, the difference absolute Galois group 
%$$\pi_1(k,\bar{k})=\Gal[\bar{k}/k]$$
%is a profinite difference group and 
%there is an anti-equivalence of categories between:
%\begin{enumerate}
%\item finitely $\diff$generated \'etale $k$-algebras and subshifts with an action of $\pi_1(k,\bar{k})$;
%\item finitely $\diff$presented \'etale $k$-algebras and subshifts of finite type with an action of $\pi_1(k,\bar{k})$.
%\end{enumerate}
%\end{remark}

\begin{remark}
{Theorem \ref{theo:translation-mech}} remains interesting even in the `geometric' case when $k$ is separably closed and $\pi_1(k,k)=1$, whence the category of finitely $\diff$presented \'etale $k$-algebras is equivalent to the category of subshifts of finite type. 
\end{remark}

\begin{remark}[Translation machanism, $\diff$finite stage]\label{translation-mech2}
Let $L/k$ be an extension of difference fields with $\forg{L}/\forg{k}$ Galois {and such that $L$ is finitely $\diff$generated as a $k$-algebra}. By \ref{diff-fld-fp-fg}, $L/k$ is finitely $\diff$presented, so the difference Galois group
$$
G=\Gal[L/k]
$$
is a \emph{group subshift of finite type} (a group object in the category of subshifts of finite type), and there is an anti-equivalence of categories between:
\begin{enumerate}
\item finitely $\diff$generated $k$-algebras split by $L$ and subshifts with an action of $G$;
\item finitely $\diff$presented $k$-algebras split by $L$ and subshifts of finite type with an action of $G$.
\end{enumerate}

Note that, in {\ref{theo:translation-mech}}, $\bar{k}$ is a colimit of finitely $\diff$generated Galois extensions of $k$, so $\pi_1(k,\bar{k})$ is a limit of group subshifts of finite type. 
\end{remark}

To illustrate the above theorem, let us show how to explicitly construct a finitely $\diff$presented $k$-algebra corresponding to a  given subshift of finite type.
\begin{example}
	Let $X\subseteq \{1,\ldots,n\}^\N$ be a subshift of finite type defined by a directed graph with edges $E\subseteq \{1,\ldots,n\}^2$. For simplicity, let us assume that $k$ has $n$ distinct elements $a_1,\ldots,a_n\in k$. Let $I\subseteq k[x,y]$ denote the ideal of all polynomials $f\in k[x,y]$ such that $f(a_i,a_j)=0$ for $(i,j)\in E$. Then the set of solutions of $I$ in $\bar{k}^2$ is exactly $\{(a_i,a_j)|\ (i,j)\in E\}$. Let $J\subseteq k\{x\}$ be the difference ideal $\diff$generated by all $f(x,\sigma(x))\in k\{x\}$, where $f\in I$. Then $A=k\{x\}/J$ is finitely $\diff$presented \'etale and $X(A)$ is isomorphic to $X$. Note that the action of $G$ on $X(A)$ is trivial because $A$ is split over $k$. 
\end{example}	
	
\begin{example}
	Let $k$ be a difference field. A standard example of a difference ideal that is not finitely $\diff$generated is the difference ideal $I$ of $k\{x\}$ $\diff$generated by $x\sigma(x),x\sigma^2(x), x\sigma^3(x),\ldots$ (\cite[Ex. 3, p. 73]{cohn}). So $k\{x\}/I$ is not finitely $\diff$presented, only finitely $\diff$generated. Note that $k\{x\}/I$ is not difference locally \'etale. To make it difference locally \'etale we can add the equation $x(x-1)$. Let $J$ be the difference ideal of $k\{x\}$ \mbox{$\diff$generated} by $x(x-1),x\sigma(x),x\sigma^2(x), x\sigma^3(x),\ldots$. Then $A=k\{x\}/J$ is locally difference \'etale. The subshift $X(A)$ corresponding to $A$ can be identified with the set of all sequences $X(A)\subseteq\{0,1\}^\N$ of $0$'s and $1$'s having at most one $1$. As predicted by Theorem \ref{theo:translation-mech} (5), $X(A)$ is not of finite type. The action of $G=\Gal[\bar{k}|k]$ on $X(A)$ is trivial because $A$ is split over $k$.
	
	To get a non-trivial action of $G$ we can replace the polynomial $x-1$ above by any separable polynomial $g\in k[x]$ that does not have $0$ as a root. Indeed, let $I$ be the difference ideal of $k\{x\}$ $\diff$generated by $xg,x\sigma(x),x\sigma^2(x), x\sigma^3(x),\ldots$ and let $0=a_0,a_1,\ldots,a_n$ denote the roots of $xg$ in $\overline{k}$. To specify a $\forg{k}$-algebra morphism $\forg{k\{x\}}\to \forg{k}$ that vanishes on $xg,\sigma(xg),\sigma^2(xg)\ldots$ is equivalent to choosing an element of $\{0,\ldots,n\}^\N$: A sequence $(x_0,x_1,\ldots)\in \{0,\ldots,n\}^\N$ corresponds to the morphism $\psi\colon \forg{k\{x\}}\to \forg{k} $ determined by $\psi(\sigma^i(x))=\sigma^i(a_{x_i})$. Such a morphism vanishes on $x\sigma^j(x),\sigma(x\sigma^j(x)), \sigma^2(x\sigma^j(x)),\ldots$ if and only if $\sigma^i(a_{x_i})\sigma^{i+j}(a_{x_{i+j}})=0$ for all $i$, i.e., if $x_i\neq 0$, then $x_{i+j}=0$. Thus, for $A=k\{x\}/I$, $\forg{k}\Alg(\forg{A},\forg{k})$ is in bijection with the subshift $X(A)$ of $\{0,\ldots,n\}^\N$ consisting of all sequences with at most one non-zero entry. The action of $G$ on $X(A)$ is determined by the action on the non-zero entry, i.e., if $g\in G$ and $\widetilde{x}\in X(A)$ has a non-zero entry $x_i\in \{0,\ldots,n\}$ at position $i$, then $g(\widetilde{x})\in X(A)$ has a non-zero entry $y\in\{0,
	\ldots,n\}$ at position $i$, where $y$ corresponds to $g(\sigma^i(a_{x_i}))=g^{\sigma^i}(a_{x_i})$ under the bijection
	$$\{0,\ldots,n\}\to \{\sigma^i(a_0),\ldots,\sigma^i(a_n)\},\ j\mapsto \sigma^i(a_j).$$
	\end{example}

\subsection{Difference connected components}

\begin{definition}
Let $X$ be an object of $\diff\Topl$,  i.e., a topological space with a continuous self-map $\sigma_X$. The \emph{$\diff$topology} on $X$ is the topology whose closed subsets are the $\sigma_X$-stable closed subsets of $\forg{X}$. The \emph{$\diff$connected components} of $X$ are the connected components of $X$ in the $\diff$topology.
\end{definition}

\begin{lemma}[{\cite[3.19]{michael-gsft}}]\label{michael-si-comp}
Let $X$ be a subshift of finite type. Then $X$ has finitely many $\diff$connected components.
\end{lemma}
\begin{proof} 
%\mw{?? maybe remove this proof but is referenced later on}
Let $X=X_\Gamma$ for some directed graph $\Gamma$, The graph of communicating classes $\bar{\Gamma}$ constructed in \ref{comm-classes} is acyclic and hence a finite union of maximal subtrees $T_j$. Let $\Theta_j$ be a subgraph of $\Gamma$ induced on the union of all communicating classes appearing as vertices of $T_j$. Then $X_{\Theta_j}$ are the $\diff$connected components of $X$.
\end{proof}

\begin{proposition}\label{finitely-si-comp}
Let $A$ be a finitely $\diff$presented \'etale difference algebra over a difference field $k$. Then $\spec(A)$ has finitely many $\diff$connected components. 
\end{proposition}
\begin{proof}
By \ref{fp-corr-sft} and \ref{diffet-spec},
$$
\spec(A)\simeq X/G,
$$
where $X$ is a subshift of finite type with an action of the absolute difference Galois group $G$. By \ref{michael-si-comp}, $X$ has finitely many $\diff$connected components, and so does $\spec(A)$ as its quotient. 
\end{proof}

\subsection{Difference core conjecture}

\begin{definition}[{\cite[1.7]{ive-mich-babb}}]
Let $k$ be a difference field. 
A $k$-algebra $A$ is called
\begin{enumerate}
\item \emph{$\diff$separable,} if the natural $\sigma_k$-twist $k$-homomorphism (dual to \ref{sigma-twist})
$$
\bar{\sigma}_A: A_{(\sigma_k)}=A\otimes_k k\to A ,\ a\otimes \lambda\mapsto \sigma(a)\lambda 
$$
is {injective}.
\item \emph{strongly $\diff$\'etale}, provided it is $\diff$separable, and $\forg{A}$ is an \'etale $\forg{k}$-algebra. {In this case the map in (1) is automatically an isomorphism.}
\end{enumerate}
\end{definition}

\begin{definition}
Let $A$ be a $k$-algebra. The \emph{strong core}
$$
\pi_0^\sigma(A)
$$
is the union of all strongly $\diff$\'etale $k$-subalgebras of $A$.
\end{definition}

\begin{remark}\label{finite-periodic}
In view of \cite[1.14]{ive-mich-babb}, in the Galois correspondence \ref{diff-abs-Gal}, strongly $\diff$\'etale algebras $A$ correspond precisely to finite inversive difference sets $X=S(\bar{k}\otimes_kA)$ (with a difference Galois action). Such a set $X$ consists of 
$\sigma_X$-periodic elements.  
\end{remark}

\begin{proposition}\label{strong-core-strong}
If $A$ is a finitely $\diff$presented \'etale difference algebra over a difference field $k$, then $\pi_0^\sigma(A)$ is a strongly $\diff$\'etale $k$-algebra.
\end{proposition}
\begin{proof}
Let $X$ be the subshift of finite type (with a difference Galois action) corresponding to $A$ via \ref{diff-abs-Gal}. A strongly $\diff$\'etale subalgebra of $A$ corresponds to a (difference Galois equivariant) shift cover map $X\to X_0$, where $X_0$ is finite consisting of periodic points by \ref{finite-periodic}.

It suffices to show that $X$ admits a maximal such quotient. Indeed, $X_0$ is a finite union of $\sigma_{X_0}$-periods, and the number of those periods is bounded by the number of $\diff$connected components of $X$, which is finite by \ref{michael-si-comp}. In order to show that the length of each $\sigma_{X_0}$-period is bounded, we may assume that we have a cover $X\to X_0$ where $X$ is $\diff$connected and $X_0$ is a $\sigma_{X_0}$-cycle.  By the proof of \ref{michael-si-comp}, there exists an irreducible component $Y\subseteq X$. Moreover, the restriction $Y\to X_0$ of the original cover is still a cover, hence the size of $X_0$ divides the length of the period of any periodic point of $Y$ (they exist by \ref{prop-irred-sft} (2)), so it is bounded.
\end{proof}

\begin{remark}
We conjectured \cite[1.19]{ive-mich-babb} that the above result holds for $A$ finitely $\diff$generated. If there exists a subshift {of a subshift of finite type} with infinitely many $\diff$connected components, the conjecture does not hold in this generality. 
\end{remark}

\subsection{Entropy for difference algebras}

\begin{definition}\label{alg-entropy}
Let $A$ be a finitely $\diff$generated \'etale difference algebra over a difference field $k$.  The \emph{entropy} of $A$ is the number
$$
h(A)=h(X),
$$
where $X$ is the underlying subshift of the $\pi_1(k,\bar{k})$-action associated to $A$ by the Galois correspondence of \ref{theo:translation-mech} (which originates in \ref{diff-abs-Gal}).
\end{definition}

\begin{remark}
The number $h(A)$ from \ref{alg-entropy} does not depend on the choice of the separable closure $\bar{k}$ because we ignore the $G=\pi_1(k,\bar{k})$-action. 

Indeed, if $\bar{k}_1$ is equipped with another lift $\bar{\sigma}_1$  
 of $\sigma_k$ to the separable closure of $\forg{k}$, there exists a $g\in\forg{G}$ such that $\bar{\sigma}=\bar{\sigma}_1 g$. 
 Since $A$ is finitely $\diff$generated, it is a quotient of $\ssig{A_0}^k$ for some finite \'etale $\forg{k}$-algebra. Galois correspondence  \ref{theo:translation-mech} with respect to the extension $\bar{k}/k$   sends $A$ to a subshift $X$ of the full shift $F=\psig{X_0}_G$ corresponding to $\ssig{A_0}^k$. The description of the $G$-action on $F$ from \ref{psig-G-sets} shows that $g$ acts as an isometry. Therefore, the subshift $(X,g\sigma_X)$ corresponding to $A$ with respect to the extension $\bar{k}_1/k$ has the same entropy as $X$.
 
% Let us denote by $G_i=\Gal[\bar{k}_i/k]=\pi_1(k,\bar{k}_i)$ and by $\Phi_i$, the difference Galois group and the functor establishing the Galois correspondence with respect to extension $\bar{k}_i/k$ as in \ref{gal-dif-corr}, for $i=1,2$.
% 
%Then, if we have a presentation $A=\psig{A_0;C_0}_{\ov k}$ in $\cA_k$, we obtain
%$$
%\Phi_i(A)=\psig{\Phi_0(A_0);\Phi_0(C_0)}_{G_i}, \ \ i=1,2,
%$$ 
%whence the underlying subshift of finite type is the same.  Of course, the difference $G_i$-actions may not be isomorphic. 
%On the other hand, there will be an element $g$ of the absolute Galois group $\forg{G_1}=\forg{G_2}$ of $\forg{k}$ so that
%$$
%\sigma_2=\sigma_1 g,
%$$  
%whence we obtain that the difference operators $()^{\sigma_1}$ and $()^{\sigma_2}$ are related by
%$$
%f^{\sigma_2}=g^{-1}\circ f^{\sigma_1},
%$$
%for $f$ in the absolute Galois group of $\forg{k}$. Hence the difference groups $G_1$ and $G_2$ are not in general isomorphic, but there will be an \emph{internal isomorphism} between them. 
\end{remark}

\begin{proposition}\label{entropy-ld}
Let $L/k$ be an extension of difference fields {such that $\forg{L}/\forg{k}$ is separably algebraic and $L$ is finitely $\diff$generated as a $k$-algebra.} Then
$$
h(L)=\log(\ld(L/k))
$$
\end{proposition}
\begin{proof}
By \ref{diff-fld-fp-fg}, $L$ is {a finitely}  $\diff$presented \'etale $k$-algebra, so, as in the proof of \ref{fin-pres-dir-pres}, we can assume 
$$L=\ssig{\forg{k}[\bar{x}]\xrightarrow{\bar{x}\mapsto{\bar{x}_0}} \forg{k}[\bar{x}_0,\bar{x}_1]/I \xleftarrow{\bar{x}_1\mapsfrom\bar{x}} \forg{k}[\bar{x}]}^k,$$
where $I=(f_1,\ldots,f_m)$, and the corresponding subshift of finite type is presented as
$$
X=\Phi(L)=\psig{\Phi_0(\forg{k}[\bar{x}])\leftarrow \Phi_0(\forg{k}[\bar{x}_0,\bar{x}_1]/I)\to \Phi_0(\forg{k}[\bar{x}])}.
$$
The set of words of length $n$ is therefore
$$
W(X,n)=\Phi_0(\forg{k}[\bar{x}_0,\bar{x}_1,\ldots,\bar{x}_n]/(I+\sigma(I)+\cdots+\sigma^{n-1}(I))=\Phi_0(L_n),
$$
where $L_n$ are as in \ref{ld-lemma}. Hence, 
$$
|W(X,n)|=|\Phi_0(L_n)|=|\forg{k}\Alg(L_n, \forg{\bar{k}})\mw{|}=[L_n:k].
$$
If $i$ is the minimal index for which $d_i=\ld(L,k)$, we obtain that
\begin{multline*}
h(X)=\lim_n \frac{1}{n}\log|W(X,n)|=\lim_n \frac{1}{n}\log([L_n:k])=\lim_n\frac{1}{n}\log([L_n:L_{n-1}]\cdots [L_0:k])\\
=\lim_{n\geq i} \frac{1}{n} \log(d_0 d_1\cdots d_{i-1} \ld(L/k)^{n-i})=\log(\ld(L/k)).
\end{multline*}

\end{proof}

\subsection{A difference zeta function}

Let $k=(\F_q,\id)$, {where $\F_q$ is the finite field with $q$ elements} and let us write 
$$\bar{k}_n=(\bar{\F}_q,\varphi_n),
$$
where $\varphi_n(\alpha)=\alpha^{q^n}$ is the $n$-th power of the $q$-Frobenius automorphism.

 Given a $k$-algebra $A\in k\Alg$ (or an affine difference scheme when considered as an object of $\cA_k$), let us consider the number 
 $$
 N_n=|k\Alg(A,\bar{k}_n)|=|\cA_k(\bar{k}_n,A)|
 $$
of $\bar{k}_n$-rational points of $A$,  and form a generating function
 $$
 Z_A(t)=\exp\left( \sum_{n=1}^\infty \frac{N_n}{n}t^n\right).
 $$

%\mw{?? Do you have simple example computation we could put here? I guess it would be difficult to find an example that is not rational? }
Hrushovski's twisted Lang-Weil estimate from \cite{udi-frob} tells us that, if $A$ is finitely $\diff$presented and `absolutely transformally integral of finite total dimension' (as defined in \cite{udi-frob}), the numbers $N_n$ are finite for large enough $n$, and it makes sense to conjecture that $Z_A(t)$ is \emph{near-rational}, i.e., that its logarithmic derivative is rational.

In the very special case when $\sigma_A=\id$, the scheme $X_0=\spec(\forg{A})$ is  of finite type over $\F_q$, and we verify that $N_n=|X_0(\F_{q^n})|$ for all $n$, whence
$$
Z_A(t)=Z_{X_0}(t),
$$ 
i.e., the difference zeta function of $A$ is the classical Weil zeta function of $X_0$ over the finite field $\F_q$, which is known to be rational by the work of Dwork and Grothendieck on classical Weil conjectures. 

Our result below serves as a proof of concept for near-rationality in the `0-dimensional' difference case. Note that even that simplest case becomes rather nontrivial if we remove the assumption that $A$ is transformally integral, because of the emerging link to symbolic dynamics.

\begin{theorem}\label{zeta-near-rat}
Let $A$ be an \'etale finitely $\diff$presented $k$-algebra. Then $Z_A(t)$ is near-rational.
\end{theorem}
\begin{proof}
We shall use the Galois correspondence \ref{diff-abs-Gal} for $\bar{k}=\bar{k}_0=(\bar{\F}_q,\id)$, which assigns to $A$ the subshift of finite type $X$ with an action of the difference group $G=(\Gal(\forg{\bar{k}}/\forg{k}, \id)$, which is topologically generated by the $q$-Frobenius $\varphi_1$. As a set, $X$ can be identified with
$$
\forg{k}\Alg(\forg{A},\forg{\bar{k}}),
$$
and the action $\sigma_X(f)=f\circ\sigma_A$ {(\ref{lemma: prepare for fields})}. Given that $\sigma_G=\id$, we obtain that the action of $\varphi_1$ on $X\simeq\forg{k}\Alg(\forg{A},\forg{\bar{k}})$, $f\mapsto\varphi_{1}\circ f$ commutes with $\sigma_X$ and thus yields a shift map 
$$\phi:X\to X.$$

%\mw{?? Why are there two definitions of $\varphi$? Would it not be enough to work with the second one?}
The same shift map is obtained by applying the Galois correspondence functor to the absolute Frobenius $k$-algebra homomorphism $A\to A$ given by $a\mapsto a^q$.

Note that 
$$
N_n=|\begin{tikzcd}[cramped, column sep=normal, ampersand replacement=\&]
{\Eq\left(X\right.}\ar[yshift=2pt]{r}{\sigma} \ar[yshift=-2pt]{r}[swap]{\phi^n} \&{\left.X\right)}
\end{tikzcd}|.
$$
Indeed, if $f\in \forg{k}\Alg(\forg{A},\forg{\bar{k}})$ is such that $\sigma(f)=\phi^n(f)$, then in fact 
$$
f\circ\sigma_A=\varphi_1^n\circ f=\varphi_n\circ f, \text{ i.e. } f\in k\Alg(A,\bar{k}_n).
$$
By functoriality of Frobenius, we may assume that the subshift $X$ is given as $X=X_T$ for a transition matrix $T$ on a finite alphabet $X_0$ and that $\phi$ is a $1$-block map  given by a bijection $f_0:X_0\to X_0$. Therefore, 
\begin{align*}
\begin{tikzcd}[cramped, column sep=normal, ampersand replacement=\&]
{\Eq\left(X\right.}\ar[yshift=2pt]{r}{\sigma} \ar[yshift=-2pt]{r}[swap]{\phi^n} \&{\left.X\right)}
\end{tikzcd}  & = \{  (x_0,x_1,\ldots)\in X_0^\N : T_{x_i,x_{i+1}}=1, \sigma (x_0,x_1,\ldots)=\phi^n(x_0,x_1,\ldots) \} \\
& = \{  (x_0,x_1,\ldots)\in X_0^\N : T_{x_i,x_{i+1}}=1, (x_1,x_2,\ldots)=(f_0^n{(x_0),f_0^n(x_1)},\ldots)\} \\
& \simeq \{  x_0\in X_0 : T_{x_0,f_0^n{(x_0)}}=1 \}. \\
\end{align*}
Let $F$ be the permutation matrix associated with $f_0$. Then we have
$$
(TF^n)_{x,y}=T_{x,f_0^n{(y)}},
$$
whence we deduce that 
$$
N_n=|\Eq(\sigma_X,\phi^n)|=\mathop{\rm Tr}(TF^n),
$$
and
$$
Z_A(t)=\exp\left(\sum_{n=1}^\infty\frac{\mathop{\rm Tr}(TF^n)}{n}t^n\right).
$$
While our approach is inspired by the calculation of the zeta function of a SFT as in \ref{zeta-sft}, the above series does not sum up to a closed form in terms of determinants. On the other hand, the logarithmic derivative of $Z_A(t)$
is 
$$
\sum_{n=1}^\infty \mathop{\rm Tr}(TF^n)t^{n-1},
$$
so, by using the fact that the permutation matrix $F$ is of finite order and by summing up the resulting geometric series,  we obtain a rational function in $t$.
\end{proof}

\appendix

%\mw{? I think it might be good to put some overview/motivation for the appendices here}  

\section{Groupoids}\label{s:groupoids}

This appendix provides an overview of internal categories/groupoids and diagrams/actions. Further details are available in \cite{johnstone} and \cite[4.6]{borceux-janelidze}.
\subsection{Internal categories and groupoids}\label{internal-cats}

Let $\cS$ be a category with pullbacks. An \emph{internal category} in $\cS$ is a collection $\C$ of objects and morphisms of $\cS$
\begin{center}
 \begin{tikzpicture} 
\matrix(m)[matrix of math nodes, row sep=0em, column sep=3em, text height=1.5ex, text depth=0.25ex]
 {
|(0)|{C_2}  &  |(1)|{C_1}		& |(2)|{C_0} \\
 }; 
\path[->,font=\scriptsize,>=to, thin]
(0) edge node[above]{$m$} (1)
([yshift=1em]1.east) edge node[above=-2pt]{$d_0$} ([yshift=1em]2.west) 
(2)  edge node[above=-2pt]{$e$} (1) 
([yshift=-1em]1.east) edge node[above=-2pt]{$d_1$} ([yshift=-1em]2.west) 
%([xshift=-3pt]1.south) edge [loop below] ([xshift=3pt]1.south)
% (1) edge [loop below] node {$S(\tau)$} (1)
;
\end{tikzpicture}
\end{center}
where $C_2=C_1\times_{C_0}C_1$ is the pullback
 \begin{center}
 \begin{tikzpicture} 
\matrix(m)[matrix of math nodes, row sep=2em, column sep=2em, text height=1.5ex, text depth=0.25ex]
 {
 |(1)|{C_2}		& |(2)|{C_1} 	\\
 |(l1)|{C_1}		& |(l2)|{C_0} 	\\
 }; 
\path[->,font=\scriptsize,>=to, thin]
(1) edge node[above]{$\pi_2$} 
(2) edge node[left]{$\pi_1$}  
(l1)
(2) edge node[right]{$d_1$} (l2) 
(l1) edge  node[above]{$d_0$} (l2);
\end{tikzpicture}
\end{center}
and we adhere to the convention that $C_1$ appearing on the left of the symbol $\times_{C_0}$ is considered with the structure map $d_0$, and $C_1$ appearing on the right is considered with the structure map $d_1$.

These morphisms should fit into suitable commutative diagrams in $\cS$ expressing the axioms for a category (\cite[4.6]{borceux-janelidze}), with the following interpretation:
\begin{enumerate}
\item $C_0$ is the \emph{object of objects};
\item $C_1$ is the \emph{object of morphisms};
\item $d_0$ is the \emph{domain/source morphism};
\item $d_1$ is the \emph{codomain/target morphism};
\item $e$ is the \emph{identity morphism};
\item $C_2$ is the \emph{object of composable pairs};
\item $m$ is the \emph{composition morphism}. 
\end{enumerate}
More precisely, we require that 
\begin{gather*}
d_0 e=d_1 e=\id_{C_0}, \ \ \ d_0 m=d_0 \pi_2, \ \ \ d_1 m= d_1\pi_1, \ \ \  m(\id\times e)=m(e\times\id)=\id_{C_1}, \\
m(\id\times m)=m(m\times\id):C_1\times_{C_0}C_1\times_{C_0}C_1\to C_1.
\end{gather*}
%\mw{?? it seems to me that the $\times$s above are used with two different meanings, maybe better to write $(\id,e)$ for the second meaning, same for \ref{internal-diags} }

An \emph{internal functor}
$$
f:\C\to \bbD
$$
between internal categories $\C$ and $\bbD$ in $\cS$
consists of morphisms $f_0:C_0\to D_0$ and $f_1:C_1\to D_1$ commuting with $d_0$, $d_1$, $e$ and $m$.

An \emph{internal groupoid} in $\cS$ is an internal category $\C$ endowed with an additional morphism $i:C_1\to C_1$, satisfying the additional commutative diagrams in $\cS$ expressing that $i$ inverts all morphisms. 

\begin{example}\label{int-cat-set}
An internal category in the category $\Set$ is a small category in the usual sense, given that we have sets of objects and morphisms.
\end{example}

\subsection{Internal diagrams and groupoid actions}\label{internal-diags}

An \emph{internal diagram} (or a \emph{copresheaf}) on an internal category $\C$ in $\cS$ is a pair of $\cS$-morphisms
$$
F=\left(F_0\xrightarrow{\gamma_0}C_0, F_1=C_1\times_{C_0}F_0\xrightarrow{\mu} F_0\right),
$$
such that
\begin{gather*}
\gamma_0 \mu=d_1\pi_1, \ \ \ \mu(e\times \id)=\id_{F_0}, \\
\mu(m\times \id)=\mu(\id\times \mu): C_1\times_{C_0}C_1\times_{C_0}F_0\to F_0.
\end{gather*}

A \emph{morphism of internal diagrams} on $\C$
$$
f:F\to G
$$
is a $\cS_{\ov C_0}$ morphism $F_0\to G_0$ commuting with the action morphisms $\mu_F$ and $\mu_G$.

The category of internal diagrams on $\C$ is denoted
$$
[\C,\cS].
$$
When $\C$ is a groupoid, we may refer to its objects as \emph{$\C$-actions in $\cS$}.

The diagram
\begin{center}
\begin{tikzpicture} 
\matrix(m)[matrix of math nodes, row sep=2em, column sep=2em, text height=1.5ex, text depth=0.25ex]
 {
|(2)|{F_0}  &  |(1)|{F_1}		& |(0)|{F_0} 	\\
|(l2)|{C_0} & |(l1)|{C_1}		& |(l0)|{C_0} 	\\
 }; 
\path[->,font=\scriptsize,>=to, thin]
(1) edge node[above]{$\pi_2$} (0)
(l1) edge node[above]{$d_0$} (l0)
(0) edge  node[right]{$\gamma_0$} (l0)
(1) edge node[above]{$\mu$} (2) 
(1) edge node[right]{$\pi_1$}   (l1)
(2) edge node[left]{$\gamma_0$} (l2) 
(l1) edge node[above]{$d_1$} (l2) 
;
\end{tikzpicture}
\end{center}
expressing the first property shows that an internal diagram $F$ on $\cC$ can be viewed as an internal category over $\C$ with the property that the right square is a pullback. An internal functor with this property is called a \emph{discrete opfibration}, see \cite[2.15]{johnstone}.

\begin{example}
Let $\C$ be an internal category in $\Set$, and let $F\in [\C,\Set]$. By \ref{int-cat-set}, $\C$ is a small category, and $F$ yields a copresheaf $\cF:\C\to \Set$ by the rule
$$
\cF(c)=\gamma_0^{-1}(\{c\}), \ \ \  \cF(c\xrightarrow{f} d)=\mu(f,\mathord{-}): \cF(c)\to \cF(d),
$$
for $c,d\in C_0$, $f\in C_1$.

Hence $F$ plays the role of the `total space' of the copresheaf $\cF$.
\end{example}

\section{Proofs for Section~\ref{s:diffalg}}\label{ap:diffalg}

\begin{proof}[Proof \ref{psig-adj}]
Given $X_0\in \cC$, we let
$$
\psig{X_0}=\prod_{i\in \N}X_i
$$
be the product of copies $X_i= X_0$, together with the left shift (omitting the first component)
$$
\sigma:\prod_{i\in \N}X_i \to \prod_{i\in \N}X_i.
$$
A straightforward verification shows that we obtain a right adjoint of $\forg{\,}$. 
\end{proof}

\begin{lemma}[{\cite[Lemma 4.3.4]{borceux-janelidze}}]\label{adj-rel}
%\mw{?? Why is the notation here not consistent with 3.1?}
Suppose we have an adjunction 
\begin{center}
 \begin{tikzpicture} 
 [cross line/.style={preaction={draw=white, -,
line width=3pt}}]
\matrix(m)[matrix of math nodes, minimum size=1.7em,
inner sep=0pt, 
row sep=3.3em, column sep=1em, text height=1.5ex, text depth=0.25ex]
 { 
  |(dc)|{\cA}	\\
 |(c)|{\cB} 	      \\ };
\path[->,font=\scriptsize,>=to, thin]
%
%(dc) edge[draw=none] node (mid) {} (c)
%(dc) edge node (fo) [pos=0.25,right=-3.5pt]{$\forg{\kern1pt}$}  (c)
(dc) edge [bend right=30] node (ss) [left]{$L$} (c)
(c) edge [bend right=30] node (ps) [right]{$R$} (dc)
%(m)id edge[draw=none] node{$\dashv$} (ss)
(ss) edge[draw=none] node{$\dashv$} (ps)
;
\end{tikzpicture}
\end{center}
%i.e., $S:\cA\to \cP$, $C:\cP\to\cA$, with $S\dashv C$, and 
with unit $\eta:\id\to RL$ and counit $\epsilon:LR\to \id$.
 If $\cA$ admits pullbacks, for any $X\in\cA$ we obtain an adjunction
\begin{center}
 \begin{tikzpicture} 
 [cross line/.style={preaction={draw=white, -,
line width=3pt}}]
\matrix(m)[matrix of math nodes, minimum size=1.7em,
inner sep=0pt, 
row sep=3.3em, column sep=1em, text height=1.5ex, text depth=0.25ex]
 { 
  |(dc)|{\cA_{\ov X}}	\\
 |(c)|{\cB_{\ov L(X)}} 	      \\ };
\path[->,font=\scriptsize,>=to, thin]
%
%(dc) edge[draw=none] node (mid) {} (c)
%(dc) edge node (fo) [pos=0.25,right=-3.5pt]{$\forg{\kern1pt}$}  (c)
(dc) edge [bend right=30] node (ss) [left]{$L_X$} (c)
(c) edge [bend right=30] node (ps) [right]{$R_X$} (dc)
%(m)id edge[draw=none] node{$\dashv$} (ss)
(ss) edge[draw=none] node{$\dashv$} (ps)
;
\end{tikzpicture}
\end{center}
where  
$$
L_X(A\xrightarrow{a}X)=L(A)\xrightarrow{L(a)}L(X),
$$
and  $R_X(E\xrightarrow{e}L(X))$ is obtained by forming the pullback
\begin{center}
 \begin{tikzpicture} 
\matrix(m)[matrix of math nodes, row sep=2em, column sep=2em, text height=1.9ex, text depth=0.25ex]
 {
 |(1)|{{Y}}		& |(2)|{R(E)} 	\\
 |(l1)|{X}		& |(l2)|{RL(X)} 	\\
 }; 
\path[->,font=\scriptsize,>=to, thin]
(1) edge   (2) edge node[left]{{$R_X(e)$}}  (l1)
(2) edge node[right]{${R}(e)$} (l2) 
(l1) edge node[above]{$\eta_X$}  (l2);
\end{tikzpicture}
\end{center}
in $\cA$.
\end{lemma}

\begin{notation}\label{sigma-twist}
Suppose that $\cC$ has pullbacks.

For an object $Y\in\cC_{\ov\forg{S}}$, its \emph{$\sigma_S$-twist} is
%$$
%Y_\varsigma=Y\times_SS
%$$
%as shown in the cartesian diagram
the pullback
$$
 \begin{tikzpicture} 
 [cross line/.style={preaction={draw=white, -,
line width=3pt}}]
\matrix(m)[matrix of math nodes, minimum size=1.7em,
inner sep=0pt,
row sep=1.5em, column sep=1em, text height=1.5ex, text depth=0.25ex]
 { 
                        |(P)|{Y_{\sigma}} 	& |(2)| {Y}          \\
                       |(1)|{S}             & |(h)|{S} \\};
\path[->,font=\scriptsize,>=to, thin]
(P) edge  (1) edge (2)
(1) edge node[above]{$\sigma$}  (h)
(2) edge  (h)
;
\end{tikzpicture}
$$
We extend the definition to morphisms in a natural way
to obtain the \emph{base change functor via $\sigma_S$}
$$
(\mathord{-})_{\sigma_S}:\cC_{\ov\forg{S}}\to \cC_{\ov\forg{S}}.
$$  

Moreover, the above functor gives rise to the base change functor via $\sigma_S$,
$$
(\mathord{-})_{\sigma_S}:\diff\cC_{\ov S}\to \diff\cC_{\ov S}.
$$

Given $X\in\diff\cC_{\ov S}$, by the universal property of pullbacks, the diagram
$$
 \begin{tikzpicture} 
 [cross line/.style={preaction={draw=white, -,
line width=3pt}}]
\matrix(m)[matrix of math nodes, minimum size=1.7em,
inner sep=0pt, 
row sep=1.5em, column sep=1em, text height=1.5ex, text depth=0.25ex]
 { 
 |(3)|{X}	   &			&[2em]	\\
& |(P)| {X_{\sigma_S}}   &  |(2)|{X} 	&       \\[1em]
& |(1)|{S}  & |(h)|{S}              &\\};
\path[->,font=\scriptsize,>=to, thin]
(P) edge node[left=-3pt]{}  (1) edge (2)
(1) edge node[below]{$\sigma$}  (h)
(2) edge (h)
(3) edge [bend right=10] (1) edge [bend left=10] node[above]{$\sigma_X$} (2) edge[dashed] node[pos=0.6,above right=-2pt]{$\bar{\sigma}_X$} (P)
;
\end{tikzpicture}
$$
yields a $\diff\cC_{\ov S}$-morphism $\bar{\sigma}_X:X\to X_{\sigma_S}$. This construction is functorial in $X$, so we obtain a natural transformation
$$
\bar{\sigma}:\id_{\diff\cC_{\ov S}}\to (\mathord{-})_{\sigma_S}.
$$
\end{notation}

\begin{proof}[Proof of \ref{psig-prop}]
The statement is immediate from \ref{psig-adj} and \ref{adj-rel}.  Nevertheless, we give a direct proof with a more informative construction of the right adjoint.

Let $f_0:X\to\forg{S}$ be an object of $\cC_{\ov\forg{S}}$. Write $f_i:X_i\to\forg{S}$ for the base change %$X_{\varsigma^i}$ of $X$ 
%$f_{0,\forg{\varsigma^i}}$ 
of $f_0$ along $\forg{\sigma^i}:\forg{S}\to\forg{S}$, and let $\sigma_i:X_{i+1}\to X_i$ be the induced morphism satisfying
$$
f_i\circ\sigma_i={\sigma_S}\circ f_{i+1}.
$$ 
Consider the pullback/fibre product/the product in ${\cC_{\ov\forg{S}}}$
$$
\psig{X}_{S}=\prod_{i\in\N}f_i=X_0\times_{\forg{S}}X_1\times_{\forg{S}}\cdots,
$$
together with a morphism
$\sigma:\psig{X}_{S}\to\psig{X}_{S}$ defined as the composite 
$$
\prod_{i\geq0}X_i/S\xrightarrow{\text{proj}}\prod_{i\geq1}X_i/S\xrightarrow{\prod\sigma_i}\prod_{i\geq0}X_i/S.
$$
This construction extends to a functor
$$
\psig{\,}_{S}:\cC_{\ov \forg{S}}\to\diff\cC_{\ov S}
$$
which is right adjoint to $\forg{\,}_S$. Indeed, the $\cC_{\ov \forg{S}}$-projection $\pi_X:\forg{\psig{X}}\to X$ induces a natural bijection
$$
{\diff\cC_{\ov S}}(Z,\psig{X}_{\ov S})\to {\cC_{\ov\forg{S}}}(\forg{Z},X).
$$
For $f\in{\diff\cC_{\ov S}}(Z,\psig{X}_{\ov S})$, the morphism $\pi_X\circ\forg{f}$ is in ${\cC_{\ov\forg{S}}}(\forg{Z},X)$.

Conversely, let $f_0\in{\cC_{\ov\forg{S}}}(\forg{Z},X)$ and let 
$f_i=\forg{\bar{\sigma}_Z^i}\circ f_{0,\forg{{\sigma_S}^i}} \in {\cC_{\ov\forg{S}}}(\forg{Z},X_i)$. We have that
$$
\sigma_{i}\circ f_{i+1}=f_i\circ\sigma_Z,
$$
whence $f=\prod_i f_i\in{\diff\cC_{\ov S}}(Z,\psig{X}_{\ov S})$. The above constructions are mutually inverse. 
\end{proof}

\begin{proof}[Proof of \ref{sigmaization-algebras}]
Suppose $A_0$ is a $\forg{k}$-algebra and denote $\varsigma=\sigma_k$. 
Let $A_i=(A_0)_{\varsigma^i}=A_0\otimes_kk$ be the base change of $A_0$ along $\varsigma^i:k\to k$, and let $\sigma_i:A_i\to A_{i+1}$ be the induced ring morphism. Consider the tensor product over $k$ (the coproduct in $k\Alg$) 
$$
\ssig{A_0}^k=\otimes_{i\in\N}A_i,
$$
together with the ring morphism $\sigma:\ssig{A_0}^k\to \ssig{A_0}^k$,
$$
\sigma(a_0\otimes a_1\otimes\cdots)=1\otimes\sigma_0(a_0)\otimes\sigma_1(a_1)\otimes\cdots.
$$
This construction extends to a functor
$$
\ssig{\,}^k:\forg{k}\Alg\to k\Alg
$$
which is left adjoint to $\forg{\,}^k$, i.e., for any $A_0\in\forg{k}\Alg$ and $B\in k\Alg$,
$$
k\Alg(\ssig{A_0}^k,B)\simeq\forg{k}\Alg(A_0,\forg{B}).
$$
\end{proof}

\begin{proof}[Proof of \ref{psig-diags}]
Since $\psig{\,}$ preserves limits (as a right adjoint), applying it to an internal category in $\cS$ gives an internal category in $\diff\cS$. Let
$$
\eta_\C:\C\to \psig{\forg{\C}}
$$
be the internal functor given by units $\eta_{C_0}$ and $\eta_{C_1}$ of the adjunction $\forg{\,}\dashv\psig{\,}$. 

We let $\psig{\,}_\C$ be the composite
$$
[\forg{\C},\cS]\xrightarrow{\psig{\,}}[\psig{\forg{\C}},\diff\cS]\xrightarrow{\eta_\C^*}[\C,\diff\cS].
$$

An $F\in [\forg{\C},\cS]$ yields a diagram
\begin{center}
\begin{tikzpicture} 
\matrix(m)[matrix of math nodes, row sep=2em, column sep=2em, text height=1.5ex, text depth=0.25ex]
 {
|(2)|{F_0}  &  |(1)|{F_1}		& |(0)|{F_0} 	\\
|(l2)|{\forg{C_0}} & |(l1)|{\forg{C_1}}		& |(l0)|{\forg{C_0}} 	\\
 }; 
\path[->,font=\scriptsize,>=to, thin]
(1) edge node[above]{$\pi_2$} (0)
(l1) edge node[above]{$d_0$} (l0)
(0) edge  node[right]{$\gamma_0$} (l0)
(1) edge node[above]{$\mu$} (2) 
(1) edge node[right]{$\pi_1$}   (l1)
(2) edge node[left]{$\gamma_0$} (l2) 
(l1) edge node[above]{$d_1$} (l2) 
;
\end{tikzpicture}
\end{center}
where the right square is a pullback. Applying $\psig{\,}$ gives the rear face of the parallelogram 
$$
 \begin{tikzpicture}
[cross line/.style={preaction={draw=white, -,
line width=4pt}}]
\matrix(m)[matrix of math nodes, row sep=.1em, column sep=1em,  
text height=1.2ex, text depth=0.25ex]
{
|(h)|{\psig{F_0}} 		&		&[0em] |(a)|{\psig{F_1}} &[0em]		& |(b)|{\psig{F_0}}	&			\\[1.2em]
			& |(H)|{\psig{F_0}_{C_0}}  	& 		&|(A)|{\psig{F_1}_{C_1}} 	&		& |(B)|{\psig{F_0}_{C_0}}\\[.6em]          
|(u)|{\psig{\forg{C_0}}}	&		&|(c)|{\psig{\forg{C_1}}}	&		& |(d)|{\psig{\forg{C_0}}} &			\\[1.2em]
			&|(U)|{C_0}	&		&|(C)|{C_1} &			& |(D)|{C_0}   \\};
\path[->,font=\scriptsize,>=to, thin,inner sep=1pt]
(a) edge node[pos=0.5,above]{$\psig{\mu}$}(h)
(c) edge node[pos=0.3,above]{$\psig{d_1}$}(u)
(h) edge node[pos=0.5,left]{$\psig{\gamma_0}$}(u) 
(U) edge  node[pos=0.5,below left]{$\eta_{C_0}$} (u) 
(H) edge [cross line] (U) edge (h)
(C) edge (U) 
(a)edge node[pos=0.5,above]{$\psig{\pi_2}$}(b)
(b)edge node[pos=0.2, left%=-2pt
]{$\psig{\gamma_0}$}(d)
(a) edge node[pos=0.3,left]{$\psig{\pi_1}$}(c)
(c)edge node[pos=0.7,above%=-2pt
]{$\psig{d_0}$}(d)
(A)edge[cross line, dashed]  (B)
(B)edge[cross line]  (D)
(A)edge[cross line]  (C)
(C)edge[cross line]  (D)
(A)edge %node[pos=0.5,above right]{$\sigma$}
(a)
(C)edge node[pos=0.6,below left]{$\eta_{C_1}$}(c)
(B)edge (b)
(A) edge [cross line] (H)
(D)edge node[pos=0.6,below left]{$\eta_{C_0}$} (d);
\end{tikzpicture}
$$ 
where the right rectangle is again a pullback, so that $\psig{F}$ is an internal diagram on $\psig{\forg{\C}}$. Pulling back via $\eta_{C_0}$ and $\eta_{C_1}$ gives us the front face, where we have used the proof of \ref{psig-prop} via \ref{adj-rel} to write 
$$
\psig{F_0}_{C_0}=\eta_{C_0}^* \psig{\gamma_0}, \text{ and } 
\psig{F_1}_{C_1}=\eta_{C_1}^* \psig{\pi_1}, 
$$
and the right rectangle is still a pullback. Hence we can define the dashed arrow to be the action on $\psig{F}_\C=\eta^*_\C\psig{F}$.

The verification that $\forg{\,}_\C\dashv\psig{\,}_\C$  now follows from the adjunction $\forg{\,}_{C_0}\dashv\psig{\,}_{C_0}$ and the definition of the action.
\end{proof}

\begin{remark}\label{corr-linearised}
If $\cC$ has pullbacks, the use of difference twists from \ref{sigma-twist} 
%
%
%the commutative diagram
%$$
% \begin{tikzpicture} 
% [cross line/.style={preaction={draw=white, -,
%line width=3pt}}]
%\matrix(m)[matrix of math nodes, minimum size=1.7em,
%inner sep=0pt,
%row sep=1.5em, column sep=1em, text height=1.5ex, text depth=0.25ex]
% { 
% 	 |(3)|{C_0}  & 			&[2em]			\\
%                       & |(P)|{X_{0\sigma}} 	& |(2)| {X_0}          \\[1em]
%                       &|(1)|{S}             & |(h)|{S} \\};
%\path[->,font=\scriptsize,>=to, thin]
%%
%(P) edge  (1) edge (2)
%(1) edge node[above]{$\sigma$}  (h)
%%
%(2) edge[cross line]  (h)
%(3) edge [bend right=10] (1) edge [bend left=10] node[above]{$\sigma_0$} (2) edge[dashed] node[pos=0.75,above right=-1pt]{$s_0$} (P)
%;
%\end{tikzpicture}
%$$
%\mw{?? the $\sigma$-twist needs to be defined before it can be used here}
%
%expressing the universal property of the pullback 
yields an alternative description of objects and morphisms of $\Corr_S$ purely in terms of morphisms of $\cC_{\ov\forg{S}}$ as follows.

\begin{itemize}
\item[(1')]\label{dvpdiag} An object $(X_0;C_0)$ is a diagram of $\cC_{\ov \forg{S}}$-morphisms
$$
 \begin{tikzpicture} 
\matrix(m)[matrix of math nodes, row sep=2em, column sep=2em, text height=1.5ex, text depth=0.25ex]
 {
 |(2)|{X_0}		& |(3)|{ C_0}  	& |(4)|{X_{0{\sigma}}}.\\
 }; 
\path[->,font=\scriptsize,>=to, thin]
(3) edge node[above]{${p_0}$} (2) (3) edge node[above]{$s_0$}   (4);
%(l3) edge node[above]{${\mathop{\rm id}}$} (l2) (l3) edge node[above]{${\varsigma}$} (l4)
%(2) edge  (l2) (3) edge (l3) (4) edge  (l4);
\end{tikzpicture}
$$
\item[(2')] A morphism $f:(X_0;C_0)\to(Y_0;D_0)$ consists of  $\cC_{\ov\forg{S}}$-morphisms $f_0:X_0\to Y_0$, $\hat{f}_0:C_0\to D_0$, making the diagram
$$
 \begin{tikzpicture}
[cross line/.style={preaction={draw=white, -,
line width=4pt}}]
\matrix(m)[matrix of math nodes, row sep=2.2em, column sep=2.2em, text height=1.5ex, text depth=0.25ex]
{
|(x0)| {X_0}		&			 |(x1)| {C_0} 			& |(x0s)| {X_{0\sigma}}	\\  
|(y0)|{Y_0} 		&			 |(y1)|{D_0} 			& |(y0s)|{Y_{0\sigma}} 				\\};
\path[->,font=\scriptsize,>=to, thin]
(x1) edge node[above]{$p_0$} (x0) edge node[above]{$s_0$} (x0s) 
	edge node[left,pos=0.5]{$\hat{f}_0$} (y1) 
(x0) edge node[left,pos=0.5]{$f_0$} (y0)
(x0s) edge node[right,pos=0.5]{$f_{0\sigma}$} (y0s)
(y1) edge node[above]{$p_0$} (y0) edge node[above]{$s_0$} (y0s) 
;
\end{tikzpicture}
$$ 
commutative.
\end{itemize}
This description explains why we may informally refer to such objects as \emph{skew-correspondences}.
%The category $\vdir$ is the full subcategory of $\vadir$ consisting of those objects $(X,\Sigma)$ for which the diagram in
%(\ref{dvpdiag}) induces a closed immersion $X_1\hookrightarrow X_0\times_SX_{0,\varsigma}$ 
%for every $\sigma\in\Sigma$.
\end{remark}

\begin{proof}[Proof of \ref{dir-gen-adj}]
Let $(X_0;C_0)$ be an $S$-correspondence. Using its description in terms of \ref{corr-linearised}, and writing 
%in the proof of \ref{psig-prop}, let 
$X_i$ for the base change/difference twist of $X_0$ via $\sigma_S^i$, let 
$$X_i\xleftarrow{p_i} C_i\xrightarrow{s_i}X_{i+1}$$ be the base change of $X_0\xleftarrow{p_0} C_0\xrightarrow{s_0}X_{1}$ via $\sigma_S^i$.

We define the difference object associated to this $S$-correspondence as the equaliser
$$
\psig{X_0\xleftarrow{p_0} C_0\xrightarrow{s_0}X_{1}}_S=\psig{X_0;C_0}_S=
\begin{tikzcd}[cramped, column sep=normal, ampersand replacement=\&]
{\Eq\left(\psig{C_0}_S\right.}\ar[yshift=2pt]{r}{\text{proj}_{X_0}\circ\psig{p_0}} \ar[yshift=-2pt]{r}[swap]{\psig{s_0}} \&{\left.\psig{X_1}_S\right)}
\end{tikzcd}
$$
where the top arrow is $\text{proj}_{X_0}\circ\psig{p_0}_S=\psig{p_1}_S\circ\text{proj}_{C_0}$, and we wrote
$\text{proj}_{X_0}$ and $\text{proj}_{C_0}$ for the projections/shifts 
$\psig{X_0}_S=\prod_{i\geq0}X_i/S\to\prod_{i\geq1}X_i/S=\psig{X_1}_S$ and $\psig{C_0}_S\to \psig{C_1}_S$, and the difference operator is the shift inherited from $\psig{C_0}_S$.

Given an object $Y\in\diff\cC_{\ov S}$, we claim that there is a natural isomorphism
$$
\Corr_S(Y\xleftarrow{\id} Y\xrightarrow{\bar{\sigma}_Y}Y_{\sigma_S}, X_0\xleftarrow{p_0}C_0\xrightarrow{s_0}X_1)\simeq
\diff\cC_{\ov S}(Y,\psig{X_0;C_0}_S).
$$

Indeed, by \ref{psig-prop},
\begin{multline*}
\diff\cC_{\ov S}(Y,\psig{X_0;C_0}_S)=\diff\cC_{\ov S}(Y,\Eq(\psig{C_0}_S\doublerightarrow{}{}\psig{X_1}_S)
\simeq\Eq(\diff\cC_{\ov S}(Y,\psig{C_0}_S)\doublerightarrow{}{}\diff\cC_{\ov S}(Y,\psig{X_1}_S))\\
\simeq \Eq(\cC_{\ov\forg{S}}(\forg{Y},C_0)\doublerightarrow{}{} \cC_{\ov\forg{S}}(\forg{Y}, X_1)),
\end{multline*}
where the top morphism is the composite 
$$
\cC_{\ov\forg{S}}(\forg{Y},C_0)\xrightarrow{p_0\circ\mathord{-}}
\cC_{\ov\forg{S}}(\forg{Y},X_0)\xrightarrow{(\mathord{-})_{\sigma_S}\circ\bar{\sigma}_Y}
\cC_{\ov\forg{S}}(\forg{Y},X_1),
$$
and the bottom morphism is
$$
\cC_{\ov\forg{S}}(\forg{Y},C_0)\xrightarrow{s_0\circ\mathord{-}}
\cC_{\ov\forg{S}}(\forg{Y},X_1),
$$
whence the above equaliser is the set
$$
\{c\in \cC_{\ov\forg{S}}(\forg{Y},C_0) : (p_0\circ c)_{\sigma_S}\circ\bar{\sigma}_Y= s_0\circ c\}\simeq
\Corr_S(Y\xleftarrow{\id} Y\xrightarrow{\bar{\sigma}_Y}Y_{\sigma_S}, X_0\xleftarrow{p_0}C_0\xrightarrow{s_0}X_1),
$$
where we used the description of morphisms between $S$-correspondences from \ref{corr-linearised}.
\end{proof}

\begin{proof}[Proof of \ref{fin-si-gen}] Assuming (1) {and $A=k[a_1,\ldots,a_n,\sigma(a_1),\ldots,\sigma(a_n),\ldots]$,} using the adjunction $\ssig{\,}^k\dashv \forg{\,}^k$, there is a unique $k$-algebra homomorphism $P_{n,k}\to A$ taking $x_i$ to $a_i$, and it is clearly surjective, so we obtain (2). The converse implication is clear.

If we have an epimorphism $\ssig{A_0}\to A$, where $A_0$ is finitely generated over $\forg{k}$, then we have an epimorphism $\forg{k}[x_1,\ldots,x_n]\to A_0$ in $\forg{k}\Alg$, so applying $\ssig{\,}^k$, we have 
$$
P_{n,k}\to \ssig{A_0}\to A,
$$
where both arrows are epimorphisms, so the composite is an epimorphism $P_{n,k}\to A$. This establishes the equivalence of (2) and (3).
 
\end{proof}

\begin{proof}[Proof of \ref{fin-pres-dir-pres}]
The equivalence of (1) and (2) follows from the fact that, if the parallel morphisms in the coequaliser diagram in (2) are denoted $f$ and $g$, then the coequaliser $A$ is the quotient of $P_{n,k}$ by the difference ideal generated by $f(x_1)-g(x_1), \ldots, f(x_m)-g(x_m)$, where $x_1,\ldots,x_m$ are variables of $P_{m,k}$. 

Suppose $A$ has a presentation as in (3). Since $C_0$ is finitely presented as a $\forg{k}$-algebra, it can be written as a coequaliser
\begin{center}
 \begin{tikzpicture} 
\matrix(m)[matrix of math nodes, row sep=0em, column sep=1.7em, text height=1.5ex, text depth=0.25ex]
 {
 |(1)|{F_{m}}		& |(2)|{F_{n}}  & |(3)|{C_0}	\\
 }; 
\path[->,font=\scriptsize,>=to, thin,yshift=12pt]
(2) edge node[above]{} (3)
([yshift=2pt]1.east) edge node[above]{} ([yshift=2pt]2.west) 
([yshift=-2pt]1.east)edge node[below]{}   ([yshift=-2pt]2.west) 
;
\end{tikzpicture}
\end{center}
of polynomial rings $F_m$ and $F_n$ over $\forg{k}$.

Since $\ssig{\,}^k$ commutes with colimits, applying it to the above yields a coequaliser
\begin{center}
 \begin{tikzpicture} 
\matrix(m)[matrix of math nodes, row sep=0em, column sep=1.7em, text height=1.5ex, text depth=0.25ex]
 {
 |(1)|{P_{m,k}}		& |(2)|{P_{n,k}}  & |(3)|{\ssig{C_0}^k}.	\\
 }; 
\path[->,font=\scriptsize,>=to, thin,yshift=12pt]
(2) edge node[above]{} (3)
([yshift=2pt]1.east) edge node[above]{} ([yshift=2pt]2.west) 
([yshift=-2pt]1.east)edge node[below]{}   ([yshift=-2pt]2.west) 
;
\end{tikzpicture}
\end{center}
whence $\ssig{C_0}^k$ satisfies (2) and (1). Similarly we obtain that $\ssig{A_0}^k$ satisfies (1). Thus, the coequaliser expression from (3) makes $A$ a quotient of a $k$-algebra $\ssig{C_0}^k$ satifsying (1) by a finitely $\diff$generated difference ideal, which is again a $k$-algebra satisfying (1).

Assuming $A$ is as in (1), by the standard procedure of adding variables and recoding the system of difference equations captured by the finitely generated difference ideal, we can ensure that the relevant system of difference equations is of order 1, so that
$$
A=P_{n,k}/ \langle f_1(\bar{x}_0,\bar{x}_1),\ldots,f_m(\bar{x}_0,\bar{x}_1)\rangle_\sigma,
$$
where $\bar{x}_0$ is the $n$-tuple of variables of $P_{n,k}$, i.e., $P_{n,k}=k\{\bar{x}_0\}=k[\bar{x}_0,\bar{x}_1,\bar{x}_2,\ldots]$ with $\sigma:\bar{x}_i\mapsto \bar{x}_{i+1}$. Applying $\ssig{\,}^k$ to the {cocorrespondence}
$$
 \begin{tikzpicture}
[cross line/.style={preaction={draw=white, -,
line width=4pt}}]
\matrix(m)[matrix of math nodes, row sep=2.2em, column sep=2.2em, text height=1.5ex, text depth=0.25ex]
{
|(x0)| {\forg{k}[\bar{x}_0]}		&			 |(x1)| {\forg{k}[\bar{x}_0,\bar{x}_1]_{\ov(f_1,\ldots,f_m)}} 			& |(x0s)| {\forg{k}[\bar{x}_0]}	\\  
|(y0)|{\forg{k}} 		&			 |(y1)|{\forg{k}} 			& |(y0s)|{\forg{k}} 				\\};
\path[->,font=\scriptsize,>=to, thin]
(x0) edge node[above]{$\bar{x}_0\mapsto \bar{x}_0$} (x1) 
(x0s) edge node[above]{$\bar{x}_1\mapsfrom \bar{x}_0$} (x1) 
(y1)	edge  (x1) 
(y0) edge  (x0)
(y0s) edge  (x0s)
(y0) edge node[above]{$\id$} (y1) 
(y0s) edge node[above]{$\sigma_k$} (y1) 
;
\end{tikzpicture}
$$ 
recovers $A$ in the form (4).
\end{proof}

%\mw{?? We would maybe remove this lemma and replace its usage with a citations. The usage of $k$ versus $\forg{k}$ is somewhat inconsistent in the proof}
\begin{lemma}\label{diff-fld-fp-fg} 
A finitely $\diff$generated difference fields extension is finitely $\diff$presented.
\end{lemma}
\begin{proof}
With notation of \ref{ld-lemma}, let $i$ be minimal such that $d_i=\ld(L/k)$. Consider the map $P_{n,k}=k\{x_1,\ldots,x_n\}\to L$ sending the tuple $x=x_1,\ldots,x_n$ to $a$. We claim that its kernel $\p$ is finitely generated as a difference ideal by the set $\p[i]$ of its elements of order up to $i$. It suffices to show that $\p[i+1]=(\p[i],\sigma(\p[i]))$. Since the right hand side is contained in the left hand side, it is enough to show that $L_{i+1}\simeq \forg{k}[x,\sigma x,\ldots,\sigma^{i+1}x]/\p[i+1]$ and $\forg{k}[x,\sigma(x),\ldots,\sigma^{i+1}x]/(\p[i],\sigma(\p[i]))$ have the same degree over $k$, or, equivalently, the same degree over $L_i$. The task is reduced to showing that 
$$
\left[\forg{k}[x,\sigma(x),\ldots,\sigma^{i+1}x]/(\p[i],\sigma(\p[i])): L_i\right]\leq \ld(L/k),
$$
and this follows from the fact that $\sigma: \forg{k}[a,\ldots,\sigma^i(A)]\to \sigma(\forg{k})[\sigma(a),\ldots,\sigma^{i+1}(A)]$ is an isomorphism. 
\end{proof}

\begin{proof}[Proof of \ref{dir-pres-diag-adj}]
We let
$$
\psig{\cF_0\xleftarrow{p_0} \cC_0\xrightarrow{\sigma_0} \sigma_\C^*\cF_0}_\C=
\begin{tikzcd}[cramped, column sep=normal, ampersand replacement=\&]
{\Eq\left(\psig{\cC_0}_\C\right.}\ar[yshift=2pt]{r}{\text{shift}\circ\psig{p_0}} \ar[yshift=-2pt]{r}[swap]{\psig{s_0}} \&{\left.\psig{\sigma_\C^*\cF_0}_\C\right)}
\end{tikzcd}
$$
in the category $[\C, \diff\cS]$. The verification that this defines a right adjoint of the above forgetful functor is purely formal, following steps analogous to the proof of \ref{dir-gen-adj}.
\end{proof}

\section{Enriched function spaces}
\label{sec: Appendix C}

\subsection{Enriched Homs}

\begin{definition}
Let $Y,Z\in\diff\cC$.
An \emph{enriched morphism} between $Y$ and $Z$ is  a commutative diagram %$f=(f_i)\in [B,C]$ is  a commutative diagram
 \begin{center}
 \begin{tikzpicture} 
\matrix(m)[matrix of math nodes, row sep=2em, column sep=2em, text height=1.5ex, text depth=0.25ex]
 {
 |(u1)|{\forg{Y}}		& |(u2)|{\forg{Z}} 	\\
 |(1)|{\forg{Y}}		& |(2)|{\forg{Z}} 	\\
 |(l1)|{\forg{Y}}		& |(l2)|{\forg{Z}} 	\\[-.5em]
 |(b1)|{}		& |(b2)|{}\\[0em]
 }; 
\path[->,font=\scriptsize,>=to, thin]
(u1) edge node[above]{$f_0$} (u2) edge node[left]{$\sigma_Y$}   (1)
(u2) edge node[right]{$\sigma_Z$} (2) 
(1) edge node[above]{$f_1$} (2) edge node[left]{$\sigma_Y$}   (l1)
(2) edge node[right]{$\sigma_Z$} (l2) 
(l1) edge  node[above]{$f_2$} (l2)
(l1) edge[dashed] (b1)
(l2) edge[dashed] (b2);
\end{tikzpicture}
\end{center}
associated with a sequence of maps $f_i:\forg{Y}\to\forg{Z}$.

The \emph{enriched hom} object between $Y$ and $Z$ is the difference set
$$
[Y,Z] = \{(f_i)\in \cC(\forg{Y},\forg{Z})^\N: f_{i+1}\,\sigma_Y=\sigma_Z\, f_i\},
$$
together with the shift
$$
s:[Y,Z]\to [Y,Z], \ \ s(f_0,f_1,\ldots)=(f_1, f_2,\ldots).
$$
\end{definition}

\begin{definition}
Let $Y,Z\in\diff\cC$. We define the difference set
$$
\lbr Y,Z\rbr = \{(f_i)\in \cC(\forg{Y},\forg{Z})^\N: f_{i}\,\sigma_Y=\sigma_Z\, f_{i+1}\},
$$
i.e., as the set of diagrams
\begin{center}
 \begin{tikzpicture} 
\matrix(m)[matrix of math nodes, row sep=2em, column sep=2em, text height=1.5ex, text depth=0.25ex]
 {
  |(b1)|{}		& |(b2)|{}\\[0em] 
   |(l1)|{\forg{Y}}		& |(l2)|{\forg{Z}} 	\\
    |(1)|{\forg{Y}}		& |(2)|{\forg{Z}} 	\\
 |(u1)|{\forg{Y}}		& |(u2)|{\forg{Z}} 	\\
 }; 
\path[->,font=\scriptsize,>=to, thin]
(u1) edge node[above]{$f_0$} (u2) 
(1) edge node[left]{$\sigma_Y$}   (u1)
(2) edge node[right]{$\sigma_Z$} (u2) 
(1) edge node[above]{$f_1$} (2) 
(l1) edge node[left]{$\sigma_Y$}   (1)
(l2) edge node[right]{$\sigma_Z$} (2) 
(l1) edge  node[above]{$f_2$} (l2)
(b1) edge[dashed] (l1)
(b2) edge[dashed] (l2);
\end{tikzpicture}
\end{center}
together with the shift
$$
s:\lbr Y,Z\rbr\to \lbr Y,Z\rbr, \ \ s(f_0,f_1,\ldots)=(f_1, f_2,\ldots).
$$

\end{definition}

\begin{remark}
The two enriched hom object structures are related via
$$
\lbr A,B\rbr_{\diff\cC} \simeq [B^\op,A^\op]_{\diff\cC^\op},
$$
where we wrote $A^\op$ for the object of $\diff\cC^\op\simeq\diff(\cC^\op)$ corresponding to the object $A$ in $\diff\cC$. Hence, the structure $\lbr\, , \rbr$ is a convenient notational device for dealing with enriched homs $[\, ,]$ of the opposite category. 
\end{remark}

\begin{remark}
Ordinary difference morphisms  between difference objects $Y,Z\in\diff\cC$ can be recovered as fixed points of the enriched  hom objects, 
$$
\diff\cC(Y,Z)=\Fix([Y,Z])=\Fix(\lbr Y,Z\rbr).
$$
\end{remark}

\begin{remark}
If $Y$ is an inversive object of $\diff\cC$, i.e., $\sigma_Y$ is an isomorphism, then 
$$
[Y,Z]\simeq \cC(\forg{Y},\forg{Z})
$$
together with the action $f\mapsto \sigma_Z\, f\, \sigma_Y^{-1}$.

Dually, if $Z$ is inversive, then
$$
\lbr Y,Z\rbr\simeq \cC(\forg{Y},\forg{Z})
$$
together with the action $f\mapsto \sigma_Z^{-1}\, f\, \sigma_Y$.
\end{remark}

\subsection{The topos of difference sets}

The category of difference sets
$$
\diff\Set
$$
inherits all (small) limits and colimits from the category of sets $\Set$. We write $1$ for its terminal object consisting of a singleton with identity difference operator.

In \cite{ive-topos}, it is viewed as a Grothendieck topos, and it was shown that, for $Y,Z\in\diff\Set$, the difference set
$$
[Y,Z]
$$
is actually the \emph{internal hom} object of this topos. In particular, we obtain the following.

\begin{proposition}\label{diffset-cart-cl}
The category $\diff\Set$ is \emph{cartesian closed,} i.e., for all $X,Y,Z\in\diff\Set$ we have natural isomorphisms
$$
\diff\Set(X\times Y,Z)\simeq \diff\Set(X,[Y,Z]),
$$
 \end{proposition}

\begin{proof}
%We construct functorial isomorphisms
%$$
%\diff\Set(A\times B,C)\simeq \diff\Set(A,[B,C]),
%$$
%for $A,B,C\in\Ob(\diff\Set)$ explicitly as follows. 
%
For a $\diff\Set$-morphism $\phi:X\times Y\to Z$ we define
$\theta=\theta_\phi:X\to[Y,Z]$ by
$$
\theta_\phi(x)_i(y)=\phi(\sigma_X^i(x),y).
$$
Clearly $\theta\circ\sigma_X=s\circ\theta$, as well as $\theta(x)_{i+1}\circ\sigma_Y=\sigma_Z\circ\theta_i(x)$, for $x\in X$ and $i\in\N$.

Conversely, for a $\diff\Set$-morphism $\theta:X\to[Y,Z]$, let $\phi=\phi_\theta:X\times Y\to Z$ be defined by
$$
\phi_\theta(x,y)=\theta(x)_0(y).
$$
It verifies the relation $\phi(\sigma_X(x),\sigma_Y(y))=\sigma_Z\phi(x,y)$, as required. 

The assignments $\phi\mapsto\theta_\phi$ and $\theta\mapsto\phi_\theta$ are mutually inverse and functorial.
\end{proof}

\subsection{Enriched Homs for difference algebras}

The enriched homs in $k\Alg$ are given by
$$
[A,B]_k = \{(f_i)\in \forg{k}\Alg(\forg{A},\forg{B})^\N: f_{i+1}\,\sigma_A=\sigma_B\, f_i\},
$$
and
$$
\lbr A,B\rbr_k = \{(f_i)\in \forg{k}\Alg(\forg{A},\forg{B})^\N: f_{i}\,\sigma_A=\sigma_B\, f_{i+1}\}.
$$

\subsection{Internal automorphisms}

\begin{definition}
Let $X\in\diff\cC$. We let
$$
\Aut[X]=\{f\in[X,X] : f_n \text{ is an isomorphism for all } n\}  
$$
and
$$
\Aut\lbr X\rbr=\{f\in\lbr X,X\rbr : f_n \text{ is an isomorphism for all } n\}.
$$
\end{definition}
Note that 
$$
\Aut[X], \Aut\lbr X\rbr\in \diff\Gp
$$
are difference groups.

\begin{definition}
Let $k\in\diff\Rng$ and $A\in k\Alg$. We define difference groups
$$
\Aut[A]_k=[A,A]_k\cap \Aut[A] , \ \ \Aut\lbr A\rbr_k=\lbr A,A\rbr_k\cap \Aut\lbr A\rbr.
$$
\end{definition}

\subsection{Pointwise topology on internal homs}

\begin{definition}
Let $X,Y\in \diff\Set$. We endow $[X,Y]$ with the topology whose subbasis of open sets consists of sets
$$
\langle x,y;n\rangle=\{f\in[X,Y]: f_n(x)=y\},
$$
for $x\in X, y\in Y$ and $n\in\N$.
\end{definition}

\begin{definition}\label{subbasis-funspacealg}
Let $k$ be a difference ring, and let  $A,B\in k\Alg$. We endow $\lbr A,B\rbr_k$ with the topology whose subbasis of open sets consists of sets
$$
\langle a,b;n\rangle=\{f\in \lbr A,B\rbr_k: f_n(a)=b\}.
$$
\end{definition}

\subsection{Difference group actions and internal automorphisms}\label{diff-gp-ac-intaut}

\begin{remark}
If a difference group $G$ acts on a difference set $X$,  the adjunction (cf.~\ref{diffset-cart-cl})
$$
\diff\Set(G\times X,X)\simeq\diff\Set(G,[X,X])
$$
shows that the action $\mu$ corresponds to a morphism of $\diff\Set$-monoids $G\to [X,X]$. Since $G$ is a group, it actually factors through a $\diff\Set$-group homomorphism
$$
G\to \Aut[X].
$$

\end{remark}

%\mw{?? fixed references [5], [29] and [31] }

\bibliographystyle{plain}
%\bibliography{pib}

\end{document}